\documentclass{amsart}
\usepackage[T1]{fontenc}
\usepackage[utf8]{inputenc}
\usepackage{graphicx}
\usepackage{amssymb,latexsym}
\usepackage{subfigure}
\usepackage{tikz}

\numberwithin{equation}{section}

\def\eps{{\epsilon}}

\def\C{{\mathbb C}}
\def\D{{\mathbb D}}

\def\H{{\mathbb H}}
\def\N{{\mathbb N}}

\def\R{{\mathbb R}}
\def\S{{\mathbb S}}

\def\Z{{\mathbb Z}}
\def\CP{{\mathbb C\mathbb P}}

\def\ov{\overline}
\def\pl{\parallel}

\def\res{\mathrm{Res}}

\usepackage{cleveref}
\theoremstyle{plain}
\newtheorem{lemma}{Lemma}[section]
\newtheorem{proposition}[lemma]{Proposition}
\newtheorem{conjecture}[lemma]{Conjecture}

\theoremstyle{definition}
\newtheorem{example}[lemma]{Example}
\newtheorem{remark}[lemma]{Remark}
\theoremstyle{plain}

\newtheorem{theorem}[lemma]{Theorem}
\newtheorem{corollary}[lemma]{Corollary}

\theoremstyle{definition}
\newtheorem{definition}[lemma]{Definition}

\theoremstyle{remark}

\title{Generic $2$-parameter perturbations of  parabolic singular points of vector fields in $\C$}

\author[M. Klime\v{s}]{Martin Klime\v{s}}
\address{Martin Klime\v{s}, Faculty of Mathematics, University of Vienna,
Oskar-Morgenstern-Platz 1, 1090 Vienna, Austria.}
\email{martin.klimes@univie.ac.at}

\author[C. Rousseau]{Christiane Rousseau}
\address{Christiane Rousseau, D\'epartement de
math\'ematiques et de statistique, Universit\'e de Montr\'eal, C.P. 6128,
Succursale Centre-ville, Montr\'eal (Qc), H3C 3J7, Canada.}
\email{rousseac@dms.umontreal.ca}

\usepackage{xcolor}
\usepackage{xspace}

\thanks{The first author thanks CRM where this research was first initiated. The second author is supported by NSERC in Canada.}


\begin{document}

\date{\today}

\begin{abstract} 
We describe the equivalence classes of germs of generic $2$-parameter families of complex vector fields $\dot z = \omega_\eps(z)$ on $\C$ unfolding a singular parabolic point of multiplicity $k+1$: $\omega_0= z^{k+1} +o(z^{k+1})$. The equivalence is under conjugacy by holomorphic change of coordinate and parameter. As a preparatory step, we present the bifurcation diagram of the family of vector fields $\dot z = z^{k+1}+\eps_1z+\eps_0$ over $\CP^1$. This presentation is done using the new tools of periodgon and star domain. 
We then provide a description of the modulus space and (almost) unique normal forms for the equivalence classes of germs.    
\end{abstract}

\maketitle
\pagestyle{myheadings}\markboth{}{Generic $2$-parameter perturbations of a singular point  of codimension $k$}

\section{Introduction}

Natural local bifurcations of analytic vector fields over $\C$ occur at multiple singular points.  If a singular point has multiplicity $k+1$, hence codimension $k$, it has a universal unfolding in a $k$-parameter family. A normal form has been given by Kostov \cite{Ko} for such vector fields, namely $$\dot z = \frac{z^{k+1} + \eps_{k-1}z^{k-1} + \dots + \eps_1z+\eps_0}{1+A(\eps)z^k}$$ and the almost uniqueness of this normal form was shown in Theorem~3.5 of \cite{RT}, thus showing that the parameters of the normal form are \emph{canonical}. 

But what happens if the codimension $k$ singularity occurs inside a family depending on $m<k$ parameters? The paper \cite{CR} completely settles the generic situation in the case $m=1$.  It occurs that, modulo a change of coordinates,  the singular points are always located at the vertices of a regular polygon: this induces a circular ordering on the singular points $z_j$ and their eigenvalues $\lambda_j$. This in turn induces a circular order (in the inverse direction) on the periods of the singular points (given by $\nu_j=\frac{2\pi i}{\lambda_j}$).  The general study of \cite{CR} starts with studying the family of vector fields $\dot z = z^{k+1} -\eps$ as a model for what can occur in a generic $1$-parameter unfolding. This family has a pole of order $k-1$ at infinity with $2k$ separatrices. In this particular case, the bifurcation diagram is determined by the \emph{periodgon}, a regular polygon, the sides of which are the periods at the singular points. When the parameter rotates, so does the periodgon and the bifurcations of homoclinic loop through $\infty$ (coalescence of two separatrices) occur precisely when two vertices of the periodgon lie on the same horizontal line. Although the periodgon is no more closed when one considers the general case of a $1$-parameter generic family, the argument could be adapted to describe the bifurcations in this case as well. 

The present paper is the first step in addressing the same question for generic 2-parameter deformations. We consider the 2-parameter family of polynomial vector fields
 \begin{equation} \label{vector_field} 
 \dot z =P_\eps(z)= z^{k+1}+\eps_1z+\eps_0,\qquad\eps=(\eps_1,\eps_0)\in \C^2,
 \end{equation}
and describe the bifurcation diagram of its real dynamics. 
This family is the smallest family unifying the two naturally occurring $1$-parameter unfoldings of the parabolic singular point $\dot z = z^{k+1}$, namely
\begin{equation}\dot z = z^{k+1}+\eps_0\label{family0}\end{equation}  
and 
\begin{equation}\dot z= z^{k+1}+\eps_1z.\label{family1}\end{equation} 
In \eqref{family0} the singular points form the vertices of a regular $(k+1)$-gon, while in \eqref{family1} there is a fixed singular point at the origin surrounded by $k$ singular points at the vertices of a regular $k$-gon. 

Note that for $k$ large, the polynomial $P_\eps(z)$ has few nonzero monomials, and hence is a \emph{fewnomial} in the sense of Khovanskii \cite{Kh}. One characteristic of the fewnomials is that the arguments of their roots are relatively equidistributed in $[0,2\pi]$. This was observed when $m=1$ since the roots were the vertices of a regular polygon. A weaker version is observed here: except for $k$ slits in parameter space we have a circular order on the roots, and the roots remain in sectors when the parameter varies. 
When the discriminant vanishes at some nonzero $\eps$, then exactly two singular points coalesce in a \emph{parabolic point}. 

The bifurcation diagram  of the real phase portrait of \eqref{vector_field} has a conic structure provided by rescaling of the vector field. 
It is therefore sufficient to describe only its intersection with a sphere $\S^3=\{\|\eps\|=1\}$, where
\begin{equation}\label{norm}
\|\eps\|=\big(\tfrac{|\eps_0|}k\big)^{\frac{1}{k+1}}+\big(\tfrac{|\eps_1|}{k+1}\big)^{\frac{1}{k}}.
\end{equation}
The sphere $\S^3=\{\|\eps\|=1\}$ can be parameterized by three real coordinates: one radial coordinate $s\in [0,1]$, and two angular coordinates $\theta, \alpha\in[0,2\pi]$,
$$ {\eps_0}=ks^{k+1}e^{i(\theta-(k+1)\alpha)},\quad {\eps_1}=-(k+1)(1-s)^ke^{-i k\alpha},\quad \|\eps\|=1.$$
The parameter $s$ measures the migration of the center singular point outwards when moving from \eqref{family1} to \eqref{family0}. 
When moving from $s=0$ to $s=1$, the movement of the outer singular points is very smooth so as to create the exact needed space for the inner singular point moving outwards.
The parameter $\theta$ determines the direction in which the center singular point will move outwards. The parameter $\alpha$ is a rotation parameter. It plays no role in the relative position of the singular points, but it is responsible for the bifurcations of homoclinic loops through the pole at infinity.

The geometry of the bifurcation diagram restricted to the sphere $\|\eps\|=1$ is that of a generalized Hopf fibration in projecting $\S^3$ over the $2$-sphere $\S^2$ parametrized by $s,\theta$. In a standard Hopf fibration the tori are foliated by torus knots of type $(1,1)$, while here we have torus knots of type $(k+1,k)$,  as explained in Section~\ref{sec:Geometry}.

Together with the bifurcation of parabolic points, the homoclinic bifurcations through $\infty$ are the only bifurcations occurring in the family. 
To describe the bifurcations of homoclinic loops though infinity we use  an approach similar to \cite{CR}.
 We slit the parameter space along the set where ${\big(\frac{\eps_0}{k}\big)^k}/{\big(\frac{-\eps_1}{k+1}\big)^{k+1}}\in[0,1]$, i.e. where $\theta\in\frac{2\pi}k\Z$ and $s\in[0,\frac12]$, 
which links the family \eqref{family1} to the discriminantal locus $\Sigma_\Delta=\{\big(\frac{\eps_0}{k}\big)^k=\big(\frac{-\eps_1}{k+1}\big)^{k+1}\}$. 
Along these slits the central singular point of \eqref{family1} migrates exactly in the direction of an outer singular point. 
Outside of these slits, not only are the roots circularly ordered by their argument, but furthermore their associated periods are the sides of a polygon on the Riemann surface of the rectifying chart $t_\eps(z)=\int_\infty^z \frac{dz}{P_\eps(z)}$ that we conjecture to have no self-intersection, the \emph{periodgon} (or polygon of the periods). 
The homoclinic orbits through $\infty$ that appear in the family occur precisely when two vertices of the periodgon lie on the same horizontal line and the segment in between them is in the interior of the periodgon.
The conjecture on the existence of no self-intersection is supported by numerical evidence. We also explain how to complete the description in case the conjecture would not be satisfied. 

The parameters $s,\theta$ determine the shape of the periodgon, which rotates when $\alpha$ varies. 
The periodgon is degenerate to a segment when $s=0$. When $s$ increases, it starts inflating, first into a nonconvex shape. The sides rotate almost monotonically and their sizes adjust so that the periodgon becomes convex and approaches the regular shape at $s=1$. 
For a parameter that lies on one of the slits in the parameter space, there exist two different limit periodgons, obtained as limits when the parameter approaches the slit from one side or the other.

The periodgon approach to the study of polynomial vector fields is a new technique, initially coming from Ch\'eritat, first used in \cite{CR}, and whose developement in a general situation is an original contribution of this paper.

\medskip

As a second part we derive an almost unique normal form for germs of generic $\ell$-parameter 
families unfolding a parabolic point of codimension $k$, which provides a classification theorem for such germs. We then discuss briefly the bifurcation diagram for this normal form  in the case $\ell=2$. This bifurcation diagram is essentially the same as the bifurcation diagram of \eqref{vector_field} when restricted to a disk $\D_r$,  where any trajectory that escapes the disk plays the same role as a separatrix of $\infty$. 
 Instead of codimension one bifurcation sets of homoclinic loops we observe thin open regions in parameter space where \emph{dividing trajectories} exist. They split the disk in disjoint parts, each containing some singular point(s). On the boundary of these regions some trajectories have double tangencies with the boundary of $\D_r$.

\section{Parameter space of the family \eqref{vector_field}}\label{sec:model}

\subsection{Symmetries of the family}
We consider the action of the transformation:
\begin{equation} (z,t, \eps_1, \eps_0)\mapsto (Z, T, \eta_1, \eta_0)= (Az, A^{-k}t, A^{k}\eps_1, A^{k+1}\eps_0)\label{linear_transf}\end{equation}
on the vector field $\tfrac{dz}{dt} = P_\eps(z)$ in \eqref{vector_field} with $t\in\R$, changing it to $\tfrac{dZ}{dT} = P_\eta(Z)$. 
We will use the particular cases: 
\begin{itemize}
\item $A=r\in \R^+$. This rescaling allows to suppose that $(\eta_1,\eta_0)\in \S^3=\{\|\eta\|=1\}$. 
\item $A= e^{\frac{2\pi i m}k}$. This gives invariance of the vector field under rotations of order $k$, modulo reparametrization. 
\item $A= e^{\frac{\pi i (2m+1)}k}$. This gives invariance under rotations of order $2k$, modulo reparametrization and reversing of time. 
\end{itemize}

\begin{proposition}\label{prop:symmetry} \begin{enumerate} 
\item Let $P_\eps$ and $P_{\eps'}$ be of the form \eqref{vector_field} for $\eps=(\eps_1,\eps_0)$ and $\eps'=(\eps_1',\eps_0')$. If \begin{equation}\begin{cases} \eps_1'=\ov{\eps}_1\\
\eps_0'=e^{\frac{2m}{k}\pi i}\ov{\eps}_0, &m\in\Z_{2k},\end{cases}\label{cond:symmetric}\end{equation} 
 then  the vector fields  $P_\eps$ and $P_{\eps'}$ are conjugate under the change $(z,t)\mapsto (\sigma(z),t)$, where $\sigma(z)$ is the image of $z$ under the reflection with respect to the line $e^{\frac{m}{k}\pi i}\R$. 
\item In particular, when $\eps_1\in \R$ and $\arg(\eps_0) = \frac{\pi m}{k}$, $m\in\Z_{2k}$, then the system is symmetric with respect to the line $e^{\frac{im\pi}k}\R$, i.e. the system is invariant under  $z\mapsto \sigma(z)$, where $\sigma(z)$ is the image of $z$ under the reflection with respect to the line.
\end{enumerate}
\end{proposition}
\begin{proof} 
For real $t$, the reflection $\sigma:z\mapsto Z=e^{\frac{2m}{k}\pi i}\ov z$ sends the vector field \eqref{vector_field}  to 
$$\frac{dZ}{dt}=Z^{k+1} + \ov\eps_1Z+e^{\frac{2m}{k}\pi i}\ov\eps_0.$$ 
\end{proof}

\begin{proposition} \begin{enumerate} 
\item Let $P_\eps$ and $P_{\eps'}$ be of the form \eqref{vector_field} for $\eps=(\eps_1,\eps_0)$ and $\eps'=(\eps_1',\eps_0')$. If \begin{equation}\begin{cases} \eps_1'=-\ov{\eps}_1\\
\eps_0'=-e^{\frac{(2m+1)}{k}\pi i}\ov{\eps}_0, &m\in \Z_{2k},\end{cases}\label{cond:reversible}\end{equation}
then $P_\eps$ and $P_{\eps'}$ are conjugate under the change $(z,t)\mapsto (\sigma(z),-t)$, where $\sigma(z)$ is the image of $z$ under the reflection with respect to the line $e^{\frac{2m+1}{2k}\pi i}\R$. 
\item In particular, when $\eps_1\in i\R$ and  $\arg(\eps_0) = \frac{\pi}2 +  \frac{2m+1}{2k}\pi$, $m\in \Z_{2k}$, then the system is reversible with respect to the line $e^{\frac{2m+1}{2k}\pi i}\R$, i.e.  invariant under $(z,t)\mapsto (\sigma(z),-t)$, where $\sigma(z)$ is the image of $z$ under the reflection with respect to the line.
\end{enumerate}
\end{proposition}
\begin{proof} 
For real $t$, the reflection $\sigma:z\mapsto Z=e^{\frac{2m+1}{k}\pi i}\ov z$ and the time reversal $t\mapsto T=-t$ sends the vector field \eqref{vector_field}  to 
$$\frac{dZ}{dT}=Z^{k+1} - \ov\eps_1Z -e^{\frac{2m+1}{k}\pi i}\ov\eps_0.$$ 
\end{proof}

\subsection{Bifurcation of parabolic points}
The discriminant of $P_\eps(z)=z^{k+1}+\eps_1z+\eps_0$ is given by 
\begin{align} 
\Delta(\eps_1,\eps_0)
&=(-1)^{\lfloor\frac{k+1}2\rfloor}k^k(k+1)^{k+1}\left[\big(\tfrac{\eps_0}k\big)^k-\big(-\tfrac{\eps_1}{k+1}\big)^{k+1}\right].
\end{align}
It vanishes if and only if
$$\eps_1=-(k+1)a^k,\quad \eps_0=ka^{k+1}\quad\text{for some } a\in\C,$$
in which case $z=a$ is a generic parabolic point (double root) of $P_\eps(z)$.
This is the only bifurcation of multiple singular points in the family \eqref{vector_field}.

\subsection{Normalizations}
In view of the symmetries of the family a natural \lq\lq norm\rq\rq\ for the parameter is given by \eqref{norm}.
If $\eps\neq 0$, we can scale $z$ so that $\|\eps\|=1$ in \eqref{vector_field}. Then it is natural to write $|\eps_1|=(k+1)(1-s)^k$ and $|\eps_0|=ks^{k+1}$, with $s\in[0,1]$. 
The two extreme values $s=0$ and $s=1$ correspond to the $1$-parameter vector fields $\dot z = z^{k+1} +\eps_1 z$ and
$\dot z= z^{k+1} +\eps_0$. In the latter there are $k+1$ singular points at the vertices of a regular $(k+1)$-gon, while in the former there are $k$ singular points at the vertices of a regular $k$-gon and one singular point at the middle. Moving $s$ from $0$ to $1$ is the transition from one to the other. The two other natural parameters are the arguments of $\eps_0$ and $\eps_1$. 
But both act on the position of the singular points. Hence we rather choose one parameter $\theta$ that will control the relative position of singular points, and one parameter $\alpha$ that will be a rotational parameter, namely we write the system in the form:
\begin{equation}\label{3_par}
\dot z = z^{k+1} -(k+1)(1-s)^ke^{-i k \alpha} z+ks^{k+1}e^{i(\theta-(k+1)\alpha)}.
\end{equation} 
This corresponds to $\arg\eps_1= -k\alpha+\pi$ and $\arg \eps_0=\theta-(k+1)\alpha$. 
We consider $\theta\in  [-2\pi,0]$. From the form of the constant term it will be natural to consider $\alpha\in [0,2\pi]$.
Using  $s$ and $\theta$ as polar coordinates, we will describe the dynamics over the disk $|e^{i\theta}s|\leq 1$, with cuts along the segments $s\in [0,\tfrac12]$, $\theta=\frac{2m\pi}{k}$, for $m\in \Z_k$,  joining the origin to the parabolic situation.

\subsection{Geometry of the parameter space}\label{sec:Geometry}
The parameter space is the 3-sphere $\S^3=\{\|\eps\|=1\}$ which is a quotient of $[0,1]\times (\S^1)^2$, on which we use coordinates $(s,\theta,\alpha)$:
\begin{equation}\label{parametersonsphere}
{\eps_0}=ks^{k+1}e^{i(\theta-(k+1)\alpha)},\quad {\eps_1}=-(k+1)(1-s)^ke^{-i k\alpha},\quad \|\eps\|=1.
\end{equation}
The quotient consists in identifying 
\begin{align}\begin{split}
(s,\theta,\alpha)&\sim(s,\theta+2\pi,\alpha)\sim(s,\theta,\alpha+2\pi)\sim(s,\theta+\tfrac{2\pi}{k},\alpha+\tfrac{2\pi}{k}),\\
(0,\theta,\alpha)&\sim(0,0,\alpha)\sim(0,0,\alpha+\tfrac{2\pi}{k}),\\ 
(1,\theta,\alpha)&\sim(1, 0,\alpha-\tfrac{\theta}{k+1})\sim(1, 0,\alpha-\tfrac{\theta}{k+1}+\tfrac{2\pi}{k+1}),\end{split}\label{rel:quotient}
\end{align}
for all $s,\theta,\alpha$.
We naturally find a generalized version of the  Hopf fibration of $\S^3$ over $\S^2$ given by the projection $(s,\theta,\alpha)\mapsto(s,\theta\mod \frac{2\pi}{k} )$, with $\S^2$ being the quotient of $[0,1]\times \S^1$ by identifying $(s,\theta)\sim (s,\theta+\frac{2\pi}{k})$ for all $s,\theta$, and $(0,\theta)\sim (0,0)$, $(1,\theta)\sim (1, 0)$ for all $\theta$.  Here $s\in (0,1)$ parametrizes a family of tori in $\S^3$, where each torus is filled by a family of $(k+1,k)$-torus knots, 
each knot corresponding to constant $(s,\theta)=(s_0,\theta_0)$  and being parametrized by $\alpha$ (see Figure~\ref{Knot}). 
For $s=0$, the torus degenerates to a circle parametrized by $\alpha$ and covered $k$ times, and for $s=1$, the torus degenerates to a circle parameterized by $\alpha$ and covered $k+1$ times.
\begin{figure} \begin{center}
\includegraphics[width=6.5cm]{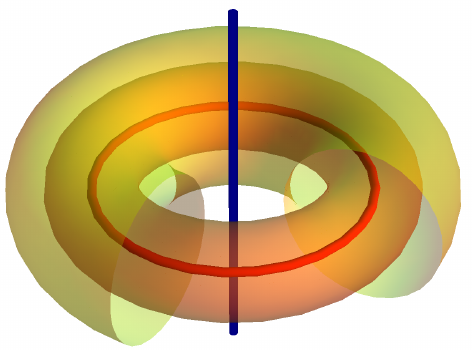}\caption{The nested tori of the Hopf fibration. The red circle corresponds to $s=1$, and the blue line (viewed as a circle of infinite radius with a point at infinity) to $s=0$.  }\label{Two_tori}
\end{center}\end{figure}
\begin{figure} \begin{center}
\subfigure[$k=2$]{\includegraphics[width=4cm]{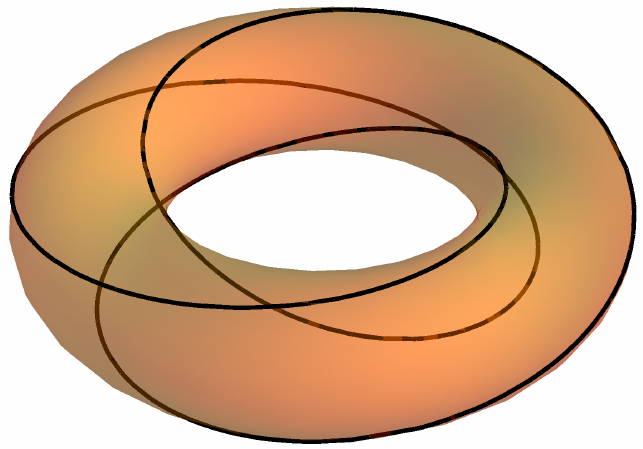}}\qquad \subfigure[$k=4$]{\includegraphics[width=4cm]{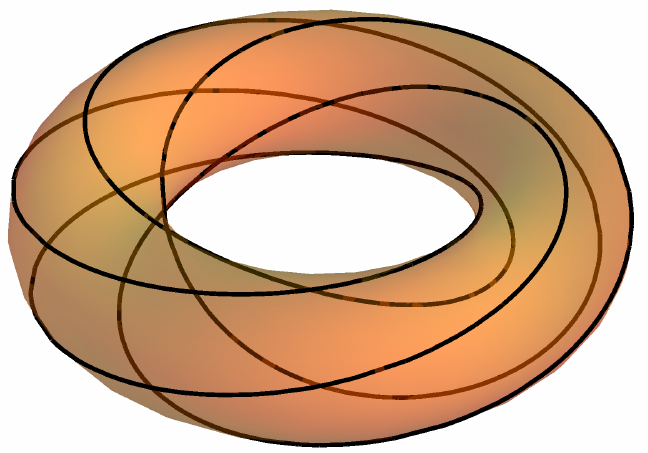}}\caption{The knots corresponding to $s,\theta$ constant and $\alpha\in [0,2k\pi]$.}\label{Knot}
\end{center}\end{figure}

The only bifurcations are homoclinic connections of two separatrices of $\infty$, and the bifurcation of parabolic point. 
The former, of real codimension $1$, will be studied by introducing the periodgon below. The later, of real codimension $2$ will be studied in Section~\ref{sec: parabolic}. 
Several bifurcations can occur simultaneously, yielding higher order bifurcations. The boundaries of the surfaces of homoclinic connections in parameter space occur along the higher order bifurcations, including the parabolic point bifurcation.

\section{Polynomial vector fields in $\C$ and their phase portrait}

The topological organization of the real trajectories of a polynomial vector field is completely determined by the pole at infinity of order $k-1$ and its $2k$ separatrices (see Figure~\ref{infinity}). It has been described by Douady and Sentenac \cite{DS} in the generic case, and by Branner and Dias \cite{BD} in the general case. 
\begin{figure}\begin{center} 
\includegraphics[width=4cm]{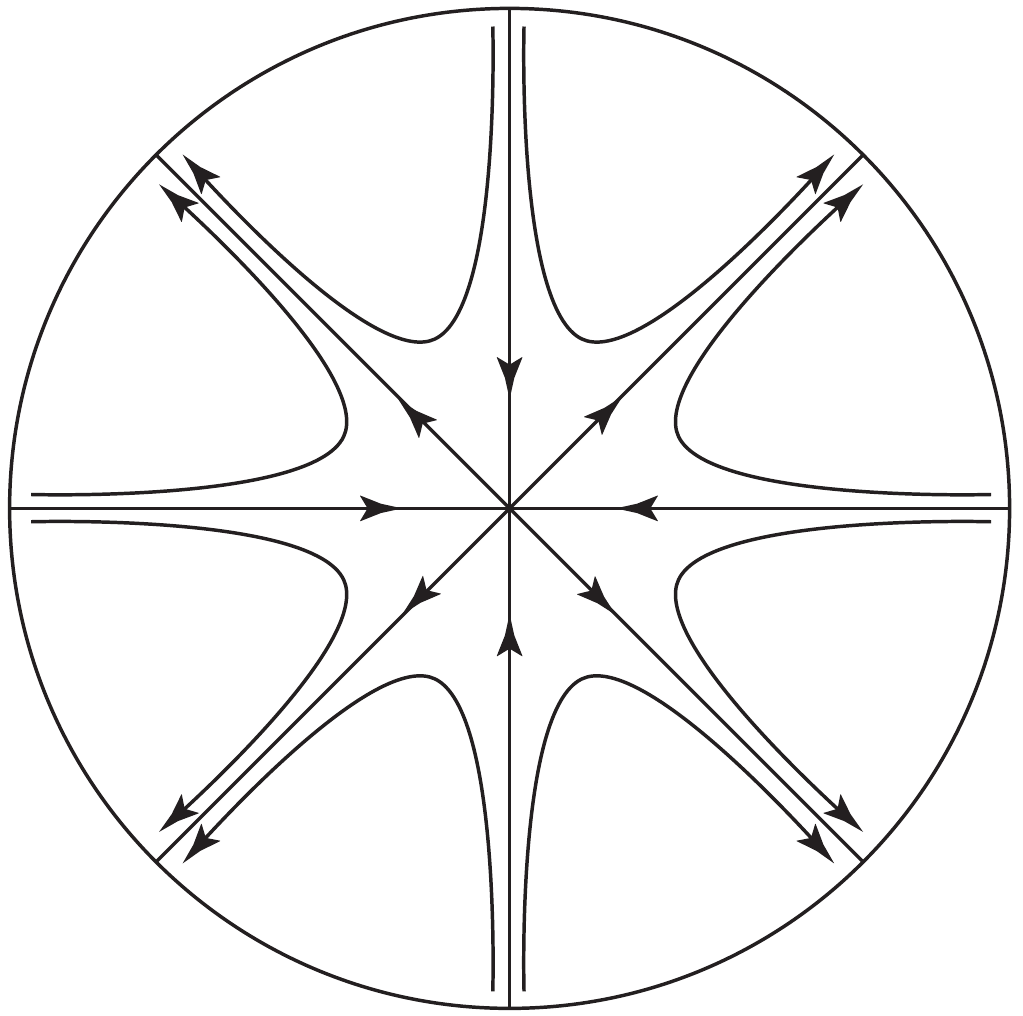}\caption{The pole at infinity and its $2k$ separatrices.}\label{infinity}
\end{center}\end{figure}

Let $\dot z = P(z)$ be a polynomial vector field  of degree $k+1$, with $k\geq1$. 
And let the multi-valued map
\begin{equation}\label{eq:t}
t(z)=\int_\infty^z \frac{dz}{P_\eps(z)}
\end{equation}
be its \emph{time coordinate} (rectifying coordinate).

\begin{definition}\label{def:zones}
\begin{enumerate} 
\item The \emph{separatrix graph} of the vector field is the  union  of the singular points and of the $2k$ separatrices of $\infty$ (which is a pole of order $k-1$ if $k\geq 2$).
 The $2k$ sectors at $\infty$ are called \emph{ends}.
\item The connected components of the complement of the separatrix graph in $\C$ are called \emph{zones}. 
Inside a zone either all orbits are periodic of the same period (\emph{center zone}), 
or all trajectories have the same $\alpha$-limit and the same $\omega$-limit  which are singular points. If the $\alpha$- and $\omega$-limit are distinct, the zone is called an \emph{$\alpha\omega$-zone}, 
otherwise it is called a \emph{sepal zone}, and the common limit is a multiple singular point (parabolic equilibrium). 
\item The \emph{skeleton graph} of the vector field is the oriented graph whose vertices are the singular points, and whose edges are the $\alpha\omega$- and the sepal zones, 
joining the $\alpha$-limit point of the zone to the $\omega$-limit point. Sepal zones correspond to loops. 
\item A polynomial vector field $\dot z = P(z)$ of degree $\geq2$ is called 
\emph{structurally stable} if all its singular points are simple and it has no homoclinic connection through $\infty$. 
\end{enumerate} 
\end{definition}

\smallskip

\begin{proposition}[\cite{DS, BD}] ~
\begin{itemize}
\item The image in the $t$-space of an $\alpha\omega$-zone is a horizontal strip of height given by the imaginary part of the transverse time 
\begin{equation}\label{eq:tau}
 \tau=\int_\gamma\frac{dz}{P(z)},
\end{equation}
over a  curve $\gamma$  from $\infty$ to $\infty$ inside the zone, transverse to the foliation, and with appropriate orientation so that $\Im \tau>0$.\footnote{ In a  generic situation, the zone has two ``ends'' at $\infty$ and the homotopy type of the transverse curve $\gamma$ from one ``end'' at $\infty$ to another is unique. However, in a non-generic situation, if the boundary of the zone contains a homoclinic separatrix, then the zone has more than two ``ends'' at $\infty$ and a pair of them has to be selected (see \cite{BD}). The height is independent of the two chosen ends.} 
\item The image in the $t$-space of a center zone with a center at a point $z_j$ is a vertical half-strip $\H^+\!/\nu_j\Z$ if $\nu_j>0$, resp. $\H^-\!/\nu_j\Z$ if $\nu_j<0$, of width given by the period of $z_j$:
$$\nu_j=\frac{2\pi i}{P'(z_j)}.$$
\item The image in the $t$-space of a sepal zone is an upper/lower half-plane $\H^\pm$. 
\item  The skeleton graph has no cycles other than loops. It is connected if and only if there are no homoclinic loops through infinity. When there is at least one homoclinic loop we say that the skeleton graph is \emph{broken}. The vector field is structurally stable if and only if the skeleton graph is a tree with $k+1$ vertices and $k$ edges. 
\end{itemize}
\end{proposition}

For the reader familiar with the work of Douady and Sentenac we add the following definition which makes the link with their work. 

\begin{definition}
If the vector field is structurally stable, then the skeleton graph is (equivalent to) the \emph{combinatorial invariant of Douady--Sentenac}.
In fact, for structurally stable vector fields each zone is an $\alpha\omega$-zone and has exactly two ends at $\infty$. Therefore the vector field defines a pairing on the set of ends, indexed by $\Z_{2k}$, that is a \emph{non-crossing involution on $\Z_{2k}$}, i.e. each pair of ends can be connected by a curve such that the curves are non-crossing (e.g. the transverse curves inside the zones).
The skeleton graph is the adjacency graph of the cellular decomposition of $\C$ by the transverse curves (cf. \cite{DS}) and it determines the non-crossing involution up to the action of even cyclic rotations on $\Z_{2k}$.
The \emph{analytic invariant of Douady-Sentenac} is given by the $k$-tuple of transversal times $\tau\in\H^+$ \eqref{eq:tau} assigned to the edges of the skeleton graph (i.e. to the zones).
\end{definition}

\subsection{The periodgon and the star domain}\label{section:periodgon}
Let $z_j$ be an equilibrium point of the vector field $\dot z=P(z)$,
and let
$$\nu_j=2\pi i\;\res_{z_j}\frac{dz}{P(z)}$$
be the \emph{period} of the rectifying map $t$ \eqref{eq:t} around $z_j$.
If $z_j$ is a simple equilibrium, then $\nu_j=\frac{2\pi i}{P'(z_j)}$, and the point $z_j$ is a center for the rotated vector field
\begin{equation}\label{eq:rotated-center}
\dot z=e^{i\arg\nu_j}P(z).
\end{equation}

\begin{definition}
\begin{enumerate}
\item  If $z_j$ is a simple equilibrium point, then the periodic zone of \eqref{eq:rotated-center} around $z_j$  is called \emph{the periodic domain of $z_j$}.
The boundary of the periodic domain consists of one or several homoclinic separatrices of \eqref{eq:rotated-center}, called \emph{the homoclinic loops of $z_j$}. 
\item  If $z_j$ is a multiple equilibrium point (parabolic equilibrium), then we call \emph{the parabolic domain of $z_j$} the union of all the sepal zones of $z_j$ in all the rotated vector fields (see Figure~\ref{fig:skel_graph}(b)) \begin{equation}\label{eq:rotated-vf}
 \dot z=e^{i\beta}P(z),\qquad \beta\in\R. 
\end{equation}
The boundary of the parabolic domain consists of one or several homoclinic separatrices of \eqref{eq:rotated-vf} with different $\beta$'s, called \emph{the homoclinic loops of $z_j$}. 
\end{enumerate}\end{definition}

\begin{figure}\begin{center}
\subfigure[]{\includegraphics[width=5.5cm]{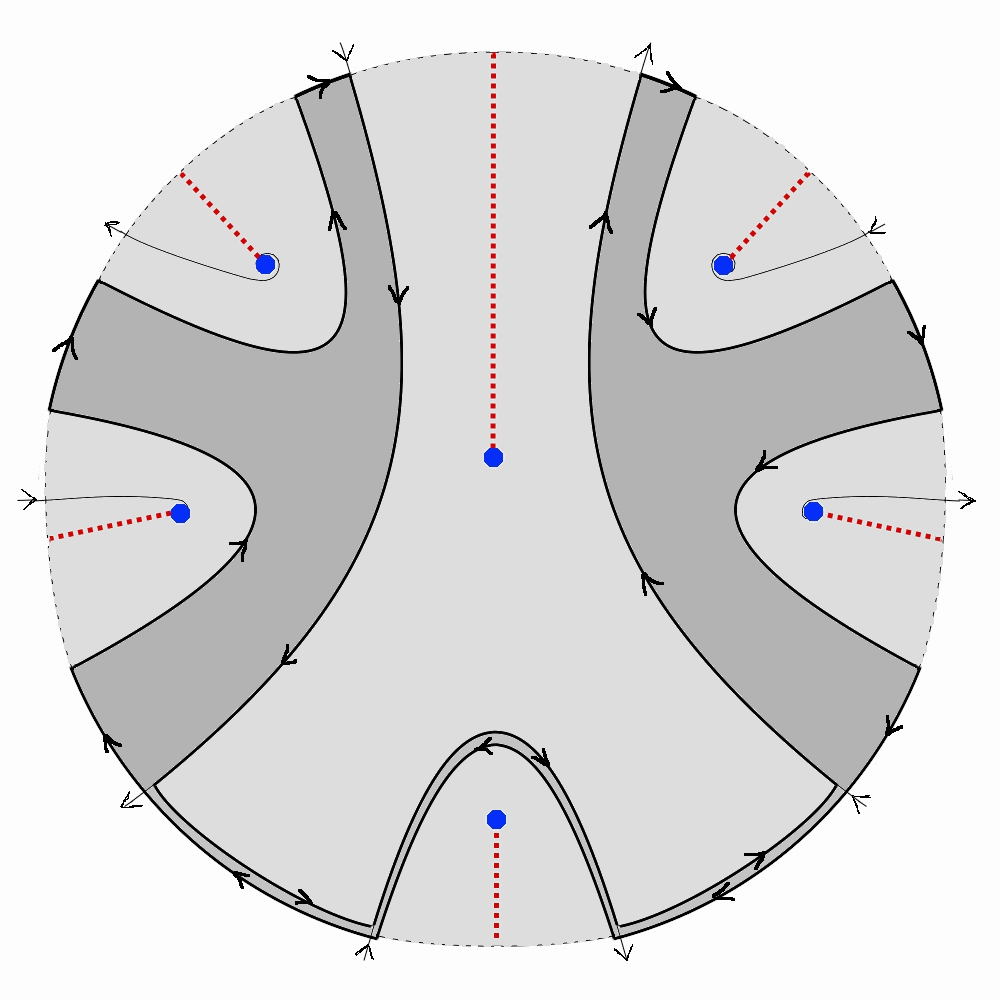}}
\subfigure[]{\includegraphics[width=5.5cm]{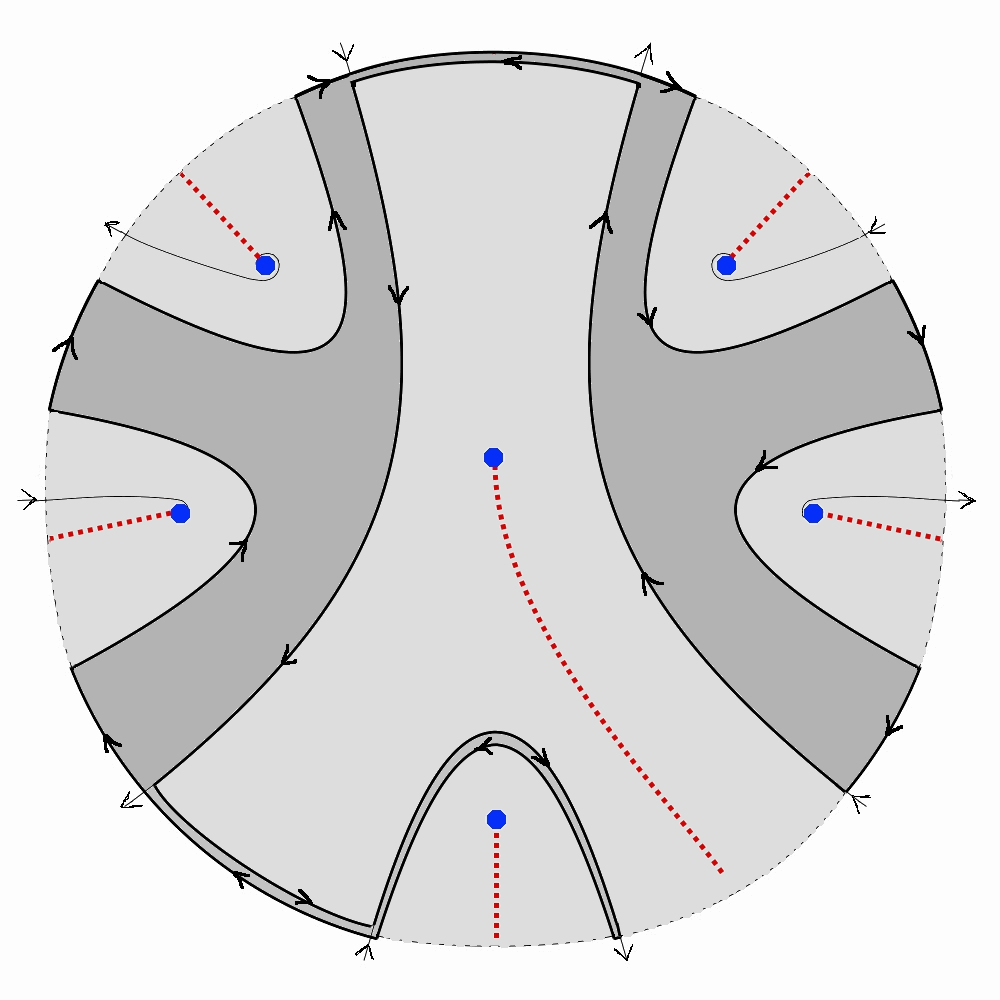}}\\
\subfigure{\includegraphics[width=4.5cm]{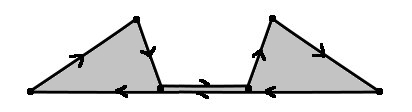}}\qquad\qquad
\subfigure{\includegraphics[width=4.5cm]{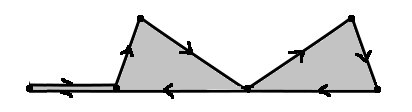}}
\caption{The periodic domains of the same vector field with 6 simple equilibria ($k=5$), two of which are centers. The boundary of the disk corresponds to a polar blow-up of the pole $z=\infty$, and the position of the $2k$ separatrices is marked. Different choices of the cuts (red dotted lines) lead to different periodgons (bottom figure).}
\label{fig:homoclinicloops}
\end{center}\end{figure}

In a generic situation the periodic domain of a simple equilibrium has a single end at $\infty$ of sectoral opening $\frac{\pi}{k}$ and is bounded by a single homoclinic loop.

\begin{definition}[Cuts]\begin{enumerate}
\item Let $z_j$ be a simple equilibrium, and $\nu_j\neq 0$ its period.
We define a \emph{cut} $D_j$ between $z_j$ and $\infty$ as a separatrix of the vector field 
\begin{equation}\dot z=e^{i(\arg\nu_j-\frac{\pi}{2})}P(z)\label{eq:rotated-node}\end{equation} 
that is contained inside the periodic zone.
If the zone has only one end at $\infty$ then the cut $D_j$ is uniquely defined, otherwise there is one possible cut $D_j$ for each end at $\infty$ of the periodic domain and one of them is chosen.
\item Let $z_j$ be a multiple equilibrium. We define a \emph{cut} $D_j$ between $z_j$ and $\infty$ as a separatrix of the vector field \eqref{eq:rotated-vf} for some chosen $\beta$
incoming to $z_j$ that is contained inside the parabolic domain.
In the case of the vector field \eqref{vector_field}, the period $\nu_j$ of the parabolic point will be always nonzero (it is given in \eqref{nu_par}), and we will chose $\beta=\arg\nu_j-\frac{\pi}{2}$ 
(see Figure~\ref{fig:skel_graph}(b)).
\end{enumerate}\end{definition}

\begin{lemma}\label{lemma:periodicdomains}
The periodic and parabolic domains of different  equilibria are disjoint. Therefore, the cuts are pairwise disjoint.
The homoclinic loops of different $z_j$ are either disjoint or they agree up to orientation.
\end{lemma}
\begin{proof}
The second statement follows from the first, which is a simple fact from the theory of rotated vector fields \cite[Theorem 4]{Du}.
Indeed, if for example the periodic domain of a point $z_j$ would intersect the periodic domain of another point $z_l$, then some periodic trajectory of \eqref{eq:rotated-center} around $z_j$ would intersect a periodic trajectory of $z_l$. But the vector field \eqref{eq:rotated-center} has a constant angle (equal to $\arg\nu_l-\arg\nu_j$) with the periodic trajectories of $z_l$, which are compact and closed, therefore none of the trajectories of \eqref{eq:rotated-center} can cross a periodic trajectory of $z_l$ twice.
Similarly, if a parabolic domain is involved.
\end{proof}

\begin{figure}\begin{center}
\subfigure[The parabolic (in the center) and periodic domains of the singular points and their cuts]{\includegraphics[width=4cm]{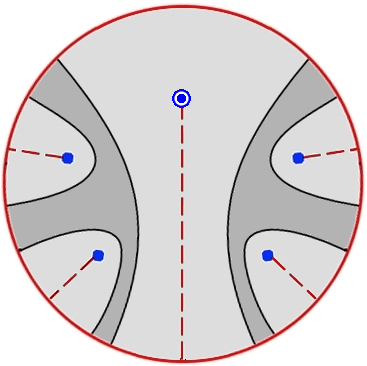}}\qquad\qquad
\subfigure[The associated periodgon]{\includegraphics[width=4cm]{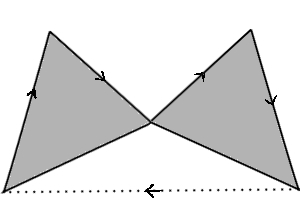}}
\caption{The \lq\lq side\rq\rq\ corresponding to the period $\nu_j$ (the dotted arrow in (b))  of a multiple equilibrium $z_j$ may be the sum of several consecutive segments between other vertices of the periodgon, not necessarily aligned.}\label{fig:parabolic-periodgon}
\end{center}\end{figure}

\begin{definition}[Periodgon]\label{def:periodgon} 
Let $D_j$ be the selected cuts of the equilibria $z_j$ labeled in their anti-clockwise order by $j\in\Z_{k+1}$. \begin{enumerate}
\item We define \emph{a period curve} as the lifting on the Riemann surface of $t(z)$ \eqref{eq:t} of the curve that follows the homoclinic loops that form the boundaries of 
the domains of $z_k,z_{k-1},\ldots,z_0$. Since the sum of all the periods vanishes, the period curve is a closed \emph{negatively oriented curve}.
If all the equilibria are simple, then the projection of the period curve in the $t$-plane is a (possibly self-intersecting) polygonal curve with edges $\nu_{k},\nu_{k-1},\ldots,\nu_0$ and vertices $\nu_k,\ \nu_k+\nu_{k-1},\ \ldots,\ \nu_k+\ldots+\nu_0=0$. 
If a multiple equilibrium $z_j$ is involved, depending on the number of homoclinic loops bounding the parabolic domain then the \lq\lq side\rq\rq\ corresponding to the period $\nu_j$ may be the sum of several consecutive segments between other vertices of the periodgon, not necessarily aligned (see Figure~\ref{fig:parabolic-periodgon}). 
In the case of system \eqref{vector_field} with a double equilibrium, $\eps\neq0$, the parabolic domain will be bounded by a single homoclinic loop and the corresponding side of the periodgon will consist of a single segment only (see Figure~\ref{fig:skel_graph}(b) and (c)).
\item We define a \emph{periodgon} of $P(z)$ as the interior of a period curve on the Riemann surface.\end{enumerate}
\end{definition}

\begin{figure}\begin{center}
\subfigure[The vector field \eqref{3_par} with a parabolic point for $k=4$ and $(s,\theta,\alpha)=(\frac12,0,0)$]{\reflectbox{\rotatebox[origin=c]{180}{\includegraphics[width=3.5cm]{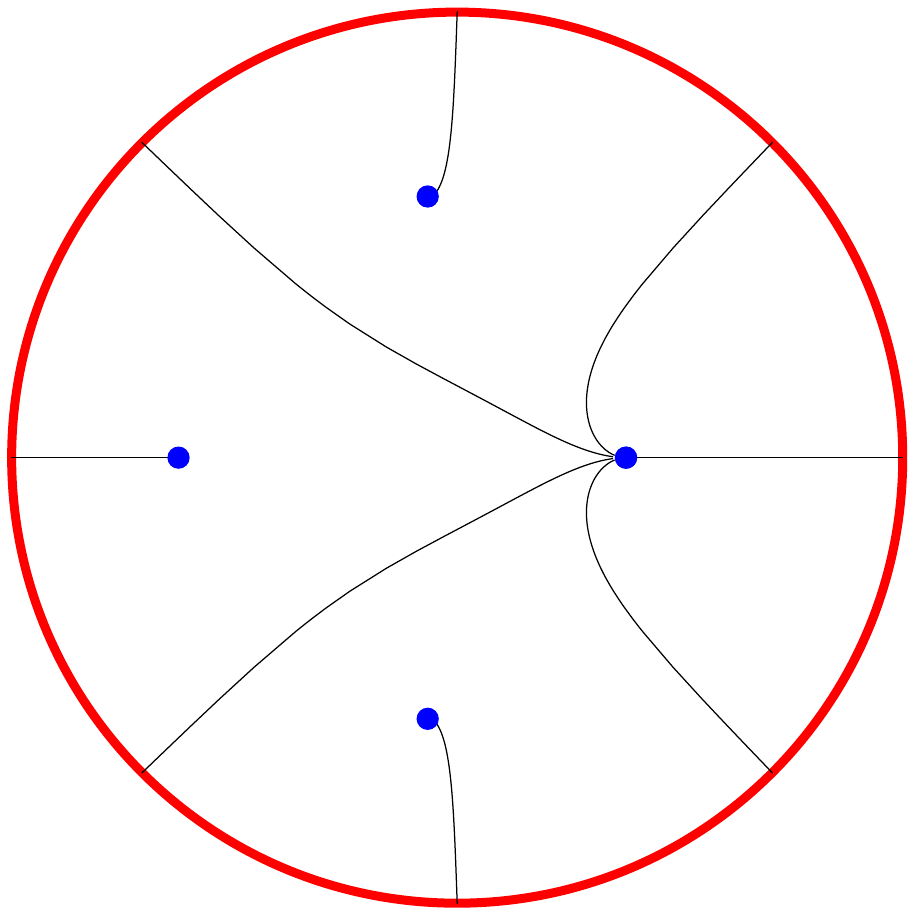}}}}\qquad
\subfigure[The parabolic and periodic domains of the singular points and their cuts]{\includegraphics[width=3.5cm]{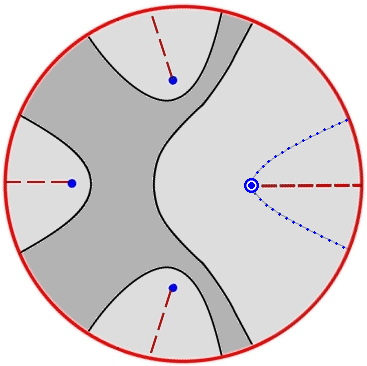}}\qquad
\subfigure[The associated periodgon and the star domain]{\includegraphics[height=4.5cm]{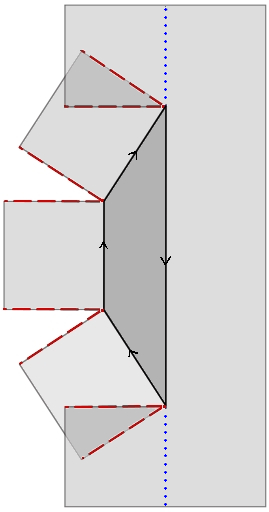}}
\caption{In the case of the vector field \eqref{3_par} with $(s,\theta,\alpha)=(\frac12,0,0)$, the parabolic domain (b) of the parabolic point is the union of the two sepal zones of the rotated vector field with $\alpha=\frac{\pi}{2k}$.}\label{fig:skel_graph}
\end{center}\end{figure}

\begin{proposition}\label{prop:periodgon}
The periodgon is well-defined and compact. The map $z\mapsto t(z)$ is an isomorphism between the complement in $\C$ of the union of the closures of the periodic and parabolic domains of all equilibria $z_j$ and the interior of the periodgon.  The vertices of the periodgon are all distinct on the Riemann surface.
\end{proposition}
\begin{proof}
The curve following along the boundaries of the periodic/parabolic domains one-by-one in positive direction is retractable, hence the corresponding period curve on the Riemann surface of $t(z)$ is closed.
The different points $z=\infty$ around the path of this curve correspond to the vertices of the periodgon.
The complement of the union of all the domains is bounded by this curve which encircles it in the negative direction. 
Since the point $z=\infty$ is a pole of order $k-1\geq 0$ of the vector field and can be reached in finite time, the periodgon is compact.
\end{proof}

\begin{remark}
\begin{enumerate} 
\item If some equilibrium point has more than one end at $\infty$, several periodgons may be defined with different order of sides depending on how the cuts are chosen (see Figure~\ref{fig:homoclinicloops}).
\item If a homoclinic loop on the boundary of the periodic/parabolic domain of one point lies also on the boundary of the periodic domain of another point (except with opposite orientation), then this loop is followed both forwards and backwards (see Figure~\ref{fig:homoclinicloops}). 
\item The anti-clockwise circular order of the cuts $D_j$ induces a circular  order on the singular points $z_j$. This circular order may be different from the one given by the arguments of the $z_j$. In our 2-dimensional family \eqref{vector_field} we conjecture  that these orders are the same. This conjecture is suggested by numerical simulations and we prove the conjecture in several regions of parameter space. \end{enumerate} \end{remark}

\begin{proposition}\label{prop:multipleends}
If the periodic domain of some simple equilibrium point $z_j$ has $m_j\geq 1$ ends at $\infty$, then there are $m_j-1$ other vertices lying on the side $\nu_j$ of the periodgon, dividing it into $m_j$ segments, corresponding to the $m_j$ homoclinic loops on the boundary of the periodic domain (see Figure~\ref{fig:homoclinicloops}).
The $m_j$ different choices for the cut from $z_j$ correspond to different cyclic permutations of the $m_j$ segments. 
\end{proposition}
\begin{proof}
The complement of the periodic domain of $z_j$ has $m_j$ components, each bounded by one homoclinic loop. 
The homoclinic loops of the equilibrium points in each component are obviously contained inside the corresponding component, and the curve consisting of these loops and of the homoclinic loops of $z_j$ bounding the  component is retractable. Therefore the sides of the periodgon corresponding to the points in each component of the complement are successive and their sum is equal to the opposite of the time along the corresponding homoclinic loop of $z_j$, which is a positive fraction the period $\nu_j$. 
\end{proof}

\begin{proposition}
 If all the equilibria are simple, then the sum of the interior angles in the periodgon is equal to $(k-1)\pi$, i.e. the \emph{turning number} of the periodgon is equal to -1 (the orientation of the periodgon is negative).
Note that the angles are considered on the Riemann surface of $t(z)$ and therefore  we cannot exclude that  some may be greater than $2\pi$.
\end{proposition}
\begin{proof}
Suppose first that we are in the generic situation where the periodic domain of each singularity has a single end at $\infty$, hence with a sectoral opening $\frac{\pi}{k}$.
The interior angle of the periodgon between $\nu_j$ and $\nu_{j+1}$ is then equal to $k$ times the sectoral opening of the complementary region between the two ends of the periodic domains.
The sum of the sectoral openings of the complementary regions is $2\pi-(k+1)\frac{\pi}{k}=\frac{k-1}{k}\pi$.
The periodgon in a non-generic situation is a degenerate limit of generic ones, and the result remains true under the right inerpretation of what the interior angles are.   
\end{proof}

\begin{definition}[Star domain]
The cut plane $\C\smallsetminus\bigcup_j D_j$ is simply connected, and we define the \emph{star domain} as the closure of its connected image on the Riemann surface of $t(z)$ \eqref{eq:t}.
\end{definition}

\begin{figure}\begin{center}
\subfigure{\includegraphics[width=8cm]{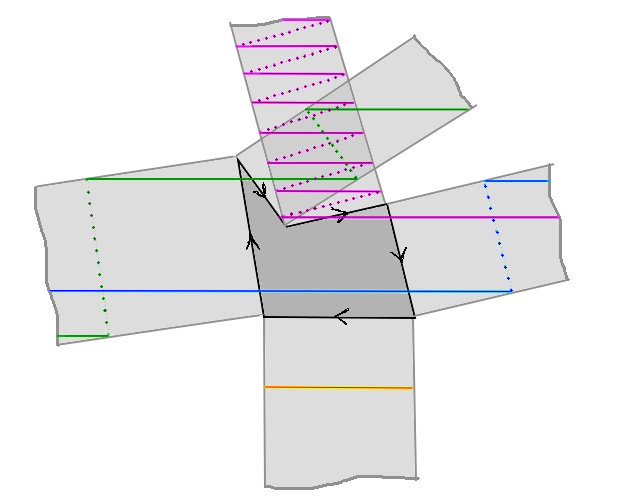}}
\caption{Example of a star domain and its periodgon on the Riemann surface of $t(z)$ with a few sample trajectories of the vector field (in color).}
\label{fig:startdomain}
\end{center}\end{figure}

\begin{proposition}\label{prop:star-domain}
Suppose that all the equilibria are simple, then the star domain is obtained by
gluing to each side $\nu_j$ of the periodgon a perpendicular infinite half-strip of width $|\nu_j|$ on the exterior of the periodgon, i.e. on the left side of $\nu_j$ (see Figure \ref{fig:startdomain}).
The cut plane $\C\smallsetminus\bigcup_j D_j$ is isomorphic through the map $z\mapsto t(z)$ to the interior of the star domain and each cut $D_j$ is mapped to the pair of rays bounding the half-strip in the star domain attached to a side $\nu_j$ of the periodgon,
which are identified by the period shift by $\nu_j$. 
\end{proposition}

\begin{proof}
The cut plane is connected and simply connected and the reparametrization by the time $t(z)$ is well defined on it with values on its Riemann surface. Because the Riemann surface of the time is ramified at the images of $\infty$, the branches of the star do not intersect.
The cut plane is the union of the open cut periodic domains of the equilibria, of the homoclinic loops and of the open complement. The homoclinic loops are mapped to the period curve and the open complement is mapped onto the interior of the periodgon (Proposition~\ref{prop:periodgon}). 
The images of the cuts $D_j$ are two parallel rays at a distance $|\nu_j|$ from each other starting from the two ends of the side $\nu_j$ of the periodgon and perpendicular to it,
and the image of the cut periodic domain of $z_j$ is the half strip in between the two rays.
\end{proof}

\begin{example}[\hbox{\cite{CR}}]\label{example:periodgon}
Consider $P(z)=z^{k+1}+\epsilon_0$, $\eps_0\neq 0$, whose roots $z_j=e^{\frac{2\pi i j}{k+1}}z_0$, $z_0=(-\eps_0)^{\frac1{k+1}}$ are located at the vertices of a regular $(k+1)$-gon. The period of $z_j$ is $\nu_j=-\frac{2\pi i}{(k+1)\eps_0}z_j$,
and the line $z_j\R$ is invariant for the vector field \eqref{eq:rotated-node}. The associated cuts are the straight segments $D_j=[z_j,+\infty e^{i\arg z_j}]$,
and the periodgon is a regular $(k+1)$-gon with sides $\nu_k,\ldots,\nu_0$ and vertices at $\nu_k,\ \nu_k+\nu_{k-1},\ \ldots,\ \nu_k+\ldots+\nu_0=0$ (see Figure~\ref{regular_gon}).
\end{example}

\begin{figure}\begin{center} 
\includegraphics[width=4.5cm]{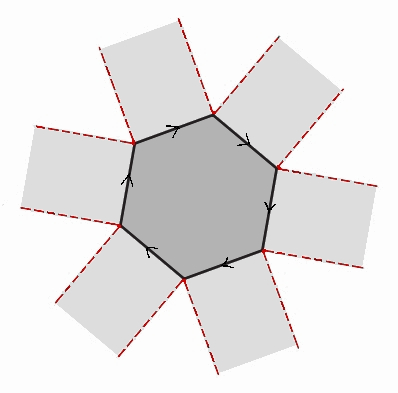}\caption{ The regular periodgon and the associated star domain when $\eps_1=0$.}
\label{regular_gon}\end{center}\end{figure}

\begin{proposition}\label{proposition:homoclinicorbits}
The vector field $\dot z=P(z)$ posseses a homoclinic separatrix if and only if two vertices of the periodgon lie on the same horizontal line
and the straight segment between them lies inside the closed periodgon (sides included).
\end{proposition}
\begin{proof}
A homoclinic separatrix of the vector field corresponds in the rectifying chart $t(z)$ to a horizontal segment joining two images of $z=\infty$.
By the same reasoning as in the proof of Lemma~\ref{lemma:periodicdomains}, a homoclinic separatrix cannot enter a periodic/parabolic domain of any point. Therefore it is either a side of the periodgon, or it is contained in the complement of the closure of all the periodic domains and hence some diagonal of the periodgon.
\end{proof}

\subsubsection{Family of rotated vector fields}
Given a vector field $\dot z=P(z)$, it is natural to consider the associated family 
of \emph{rotated vector fields} \eqref{eq:rotated-vf}.

\begin{remark}\label{remark:rotation} For our family $\dot z=P_\eps(z)$, with $\eps$ given in \eqref{parametersonsphere}, which we denote as  $\dot z=P_{(s,\theta, \alpha)}(z)$ \eqref{3_par}),
then the conjugate family through $z\mapsto Z=e^{i\alpha} z$ (with same periodgon!) becomes 
$\dot Z= e^{-ik\alpha}P_{(s,\theta,0)}(Z)$, i.e. of the form \eqref{eq:rotated-vf} with $\beta=-k\alpha$.\end{remark}

The great advantage of the star domain description of the vector field over the zone decomposition of Douady--Sentenac
is that, with varying $\beta$ (or $\alpha$ in the case of \eqref{3_par}) 
the shape of the star domain stays the same while the domain rotates by $\beta$ (or $-k\alpha$).
This allows visualizing all the homoclinic orbits that arise in the family \eqref{eq:rotated-vf} as segments of argument $-\beta+\pi\Z$ joining pairs of vertices inside the periodgon of \eqref{eq:rotated-vf} for $\beta=0$.
In particular, if the periodgon is strictly convex (all the interior angles are less than $\pi$) then there are exactly $k(k+1)$ homoclinic separatrices that occur in the family for $\beta\in[0,2\pi)$, while there are less if the periodgon is not strictly convex.

\subsubsection{Bifurcations of the shape of the periodgon}

When $P=P_\epsilon$ depends continuously on a parameter $\epsilon$,
we have a uniform description of the dynamics on any domain in parameter space where the star domain can be continuously defined. (Such a domain cannot contain $\eps=0$, but this is no problem because of the conical structure of the family.)
 We define the \emph{bifurcation locus of the  periodgon} (and hence of the star domain) 
as the set of parameters where the shape of the star domain changes discontinuously.

\begin{proposition}\label{prop: bif_shape}
The \emph{bifurcation locus of the periodgon} of $\dot z=P_\eps(z)$ is the union $\Sigma=\Sigma_\Delta\cup\Sigma_0$,
where 
\begin{enumerate} 
\item $\Sigma_\Delta$ is the discriminantal locus, i.e. the set of $\eps$ for which $P_\eps(z)$ has a multiple  root;
\item  $\Sigma_0$ is the set of $\eps\notin\Sigma_\Delta$ for which one of the roots of $P_\eps(z)$ has multiple homoclinic loops.
\end{enumerate}
When $\eps\in\Sigma_0$ there are several periodgons as described in Proposition~\ref{prop:multipleends}.
\end{proposition}
\begin{proof} 
The periods $\nu_j$ depend continuously on the parameter. A bifurcation of the periodgon occurs when either $\nu_j\to \infty$, i.e. when the discriminant vanishes,
or when the order of the sides of the periodgon changes. This happens either when $\epsilon\in\Sigma_{\Delta}$, or when the period domain of some equilibrium $z_j$ has several ends at $\infty$, yielding
several possibilities for the cut $D_j$.   
\end{proof}

\section{The bifurcation diagram of \eqref{vector_field}}

\subsection{The slit domain}
From now on, let $\dot z=P_\eps(z)$ be the vector field \eqref{vector_field}, 
and let $(s,\theta,\alpha)$ be the reduced parameters on the real 3-sphere $\S^3$ corresponding to $\|\eps\|=1$ (see Section~\ref{sec:Geometry}). 

We conjecture that the star domain depends continuously on the parameters provided we slit $\S^3$ along the segments between $s=0$ and $s=\frac12,\ \theta\in\frac{2\pi}k\Z$. This conjecture will be further discussed below.

\begin{conjecture}\label{conjecture_0} 
The bifurcation locus of the periodgon of the family \eqref{3_par} with parameters $(s,\theta,\alpha)$ is the set 
$$\Sigma=\{(s,\theta,\alpha): s\in[0,\tfrac12],\ \theta\in\tfrac{\pi}k\Z\}.$$ 
\end{conjecture}

\begin{figure}\begin{center}
\includegraphics[height=3.2cm]{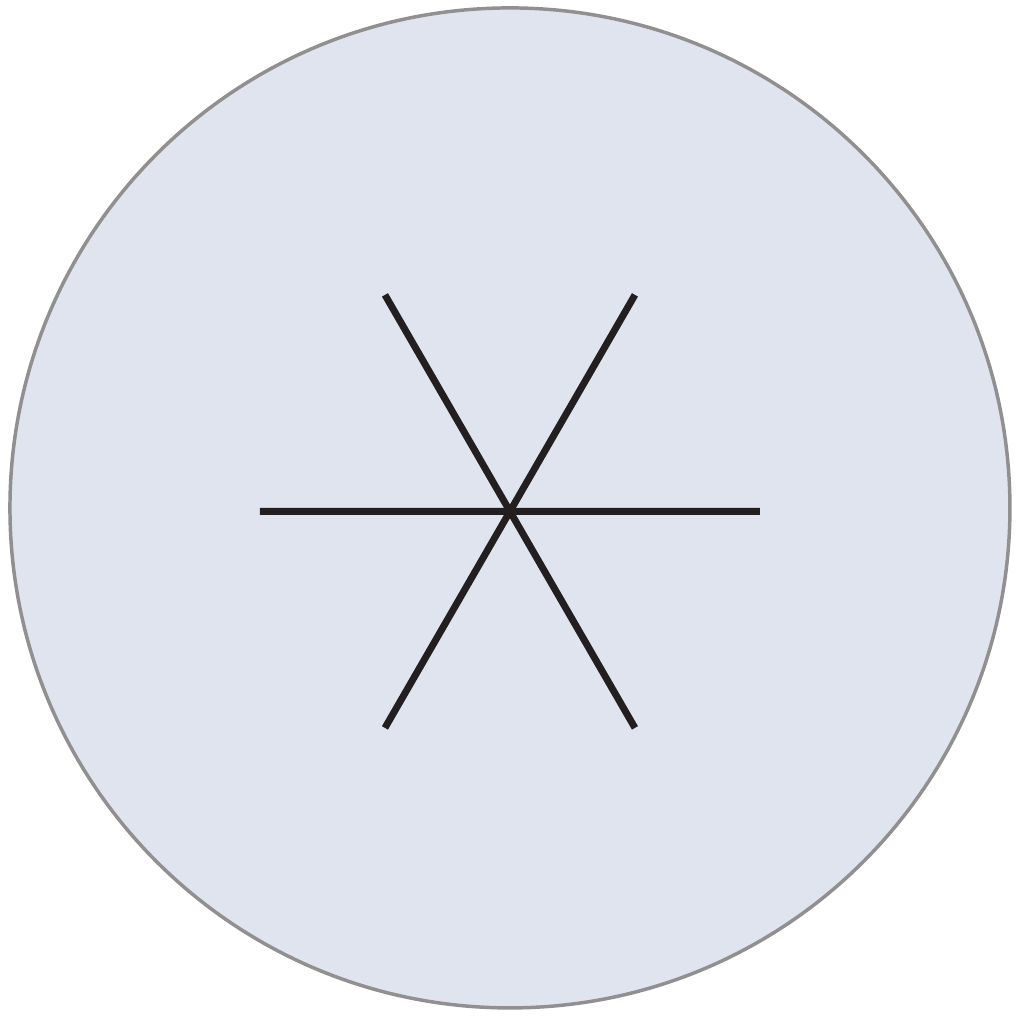}\caption{The slit domain $\D$  in $(s,\theta)$-space.}\label{domain_theta}
\end{center}\end{figure}

In case the conjecture were not true, the shape of the periodgon and star domain could undergo bifurcations elsewhere. 
Then, it would suffice to split the parameter space into a union of regions on which the star domain is continuously defined, and then to describe the dynamics over each region. 

\begin{definition}[The \emph{slit domain}]\label{def:slit_disk}
Let $s\in [0,1]$, $\theta\in [0,2\pi]$, $\alpha\in [0,2\pi]$ be the coordinates \eqref{parametersonsphere} covering the parameter space restricted on the sphere $\|\eps\|=1$.
In order to get a uniform description of the periodgon, we need to cut the disk $|se^{i\theta}|\leq1$ in radial coordinates $(s,\theta)$ along the rays $s\in [0,1/2]$, $\theta = \frac{2m\pi}{k}$ (see Figure~\ref{domain_theta}). 
And we will consider the closure $\D$ of this slit disk: the slits will be covered twice depending whether we approach them with $\theta>\frac{2m\pi}{k}$ or  $\theta< \frac{2m\pi}{k}$. \end{definition}

On each of the slits $s\in [0,1/2]$, $\theta = \frac{2m}{k}\pi^\pm$, we define the periodgon and the star domain by taking its limit.

\subsection{The bifurcation diagram of $\dot z=P_\eps(z)$}

We can now characterize homoclinic separatrices of the vector field as horizontal segments between vertices in the periodgon, and characterize all the bifurcations of the separatrices. 

\begin{theorem}\label{main_thm} The bifurcation diagram of the vector field \eqref{vector_field} is formed of 
\begin{enumerate} 
\item Codimension 1 bifurcations: each bifurcation of codimension $1$ consists in the coalescence of one attracting and one repelling separatrix of $\infty$ in a homoclinic loop. 
\item Codimension 2 bifurcations: these are
\begin{enumerate} 
\item Bifurcations of parabolic points when $s=\frac12$ and $\theta=\frac{2\ell\pi}{k}$: the bifurcation is of codimension $2$ when $\alpha\neq \frac{(2m+1)\pi}{2k}$. In this case three adjacent separatrices of $\infty$ end in the parabolic point (see Figures~\ref{fig:bif_parabolic} (a) and (c) and \ref{parabolic_k_4}).
\item Simultaneous occurrence of two homoclinic loops: we call them \emph{double homoclinic loops} (see Figure~\ref{fig:double_homoclinic}). Such loops occur in particular when two segments joining vertices of the periodgon are parallel or when three vertices of the periodgon are aligned.
\item When the system is reversible, i.e. $\theta=\frac{\ell\pi}{k}$, $\alpha=\frac{(2m+1)\pi}{2k}$ and $s\neq0$,  simultaneous occurrence of $N\geq2$ homoclinic loops. When $\ell$ is even, then $N=\lfloor\frac{k}2\rfloor+1$ for  $s\in(0,\frac12)$ and  $N=\lfloor\frac{k}2\rfloor$ for
$s\in(\frac12,1]$. When $\ell$ is odd, then  $N=\lfloor\frac{k+1}2\rfloor$ (see Figures~\ref{fig:reversible}, \ref{fig:reversible_theta} and \ref{fig:homoclinicperiodgon}).
\item Simultaneous occurrence of $k+1$ centers separated by $k$ homoclinic loops when $s=0$ and $\alpha=\frac{(2m+1)\pi}{2k}$.
\item For $k\geq5$, homoclinic loops appering at a potential bifurcation of the peridgon other then the one on the slits, described precisely in Section~\ref{sec:birf_loops}.  These will however not occur under the Conjecture~\ref{conjecture_0}. 
\end{enumerate}
\item Codimension 3 bifurcations of parabolic point when $s=\frac12$, $\theta=\frac{2\ell\pi}{k}$ and $\alpha= \frac{(2m+1)\pi}{2k}$. 
In that case only two adjacent separatrices of $\infty$ end in the parabolic point and $\left\lfloor\frac{k}{2}\right\rfloor$ homoclinic separatrices occur.
\item Codimension 4 bifurcations
\begin{enumerate}
\item  The bifurcation at $\epsilon=0$. 
\end{enumerate}
\end{enumerate}
The boundaries of the surfaces of homoclinic bifurcations inside the compact parameter space are formed of curves of codimension 2 bifurcations of parabolic point and curves of double homoclinic loops, as well as some points of codimension $3$ bifurcations.
\end{theorem}
\begin{proof}  
Bifurcations of the real phase portrait of the vector field \eqref{3_par} happen through either an occurrence of a homoclinic orbit (codimension 1) or multiple homoclinic orbits (codimension $\geq 2$), or through an occurrence of a multiple equilibrium (codimension 2), or a combination of both (codimension $\geq 3$).
Proposition~\ref{proposition:homoclinicorbits} gives an easy description of the homoclinic loops: a homoclinic loop occurs precisely when two vertices of the periodgon lie on a horizontal line and the corresponding segment is contained inside the periodgon (which may be non-convex).
Multiple homoclinic loops correspond to multiple horizontal segments between vertices of the periodgon. 
The shape of the periodgon is completely determined by $s$ and $\theta$ and the periodgon rotates by $-k\alpha$ when $\alpha$ varies. 
Therefore the bifurcations of homoclinic loops are completely determined by the shape of the periodgon. Because of the symmetries it is sufficient to describe the periodgon for $\alpha=0$ and $\theta\in [-\frac{\pi}{k},0]$. 
The exact shape of the periodgon is still conjectural for some regions in parameter space. While we conjecture that its projection on $\C$ has no self-intersection on the  interior of the slit domain in parameter space, 
we could only prove in Section~\ref{sec:periodgon} that at most two kinds of simple self-intersections can occur, leading to the corresponding codimension 2 bifurcations of case (2)(e). These bifurcations are precisely described in  Section~\ref{sec:birf_loops}.  

A complete description of the surfaces of homoclinic bifurcations in the case $k=2$ can be found in \cite{R2}.
\end{proof}

\begin{figure}\begin{center}
\subfigure[]{\reflectbox{\rotatebox[origin=c]{180}{\includegraphics[width=3.5cm]{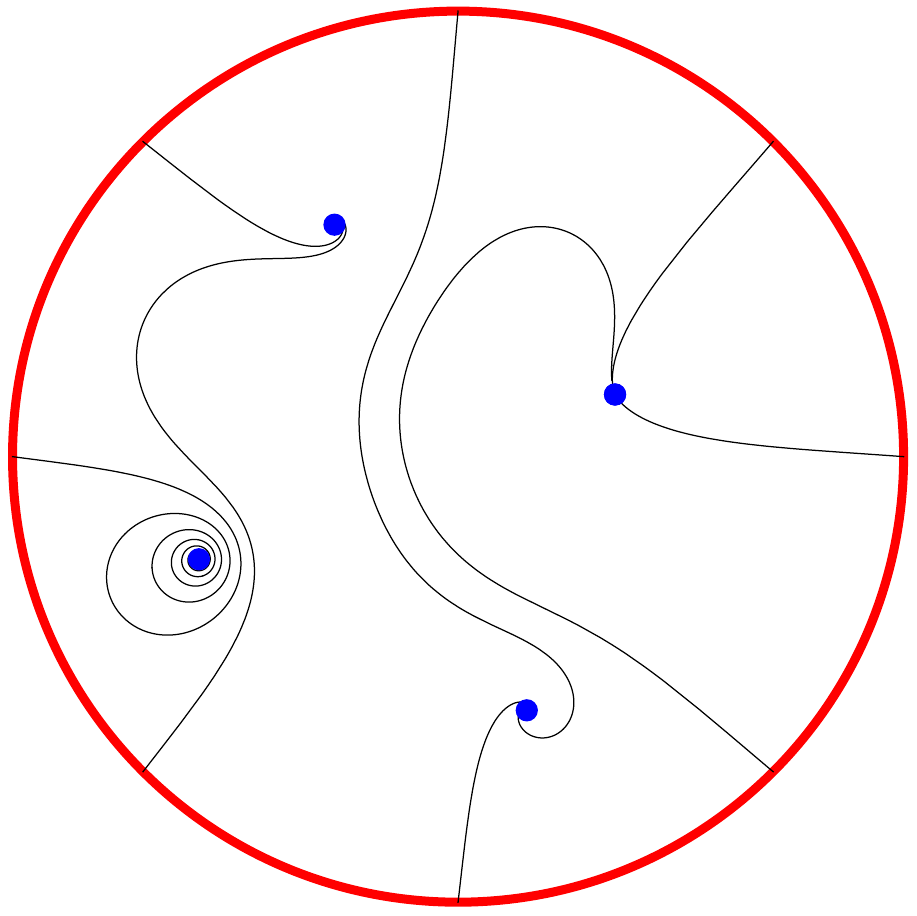}}}}\qquad
\subfigure[$\alpha=\frac{\pi}{2k}$]{\reflectbox{\rotatebox[origin=c]{180}{\includegraphics[width=3.5cm]{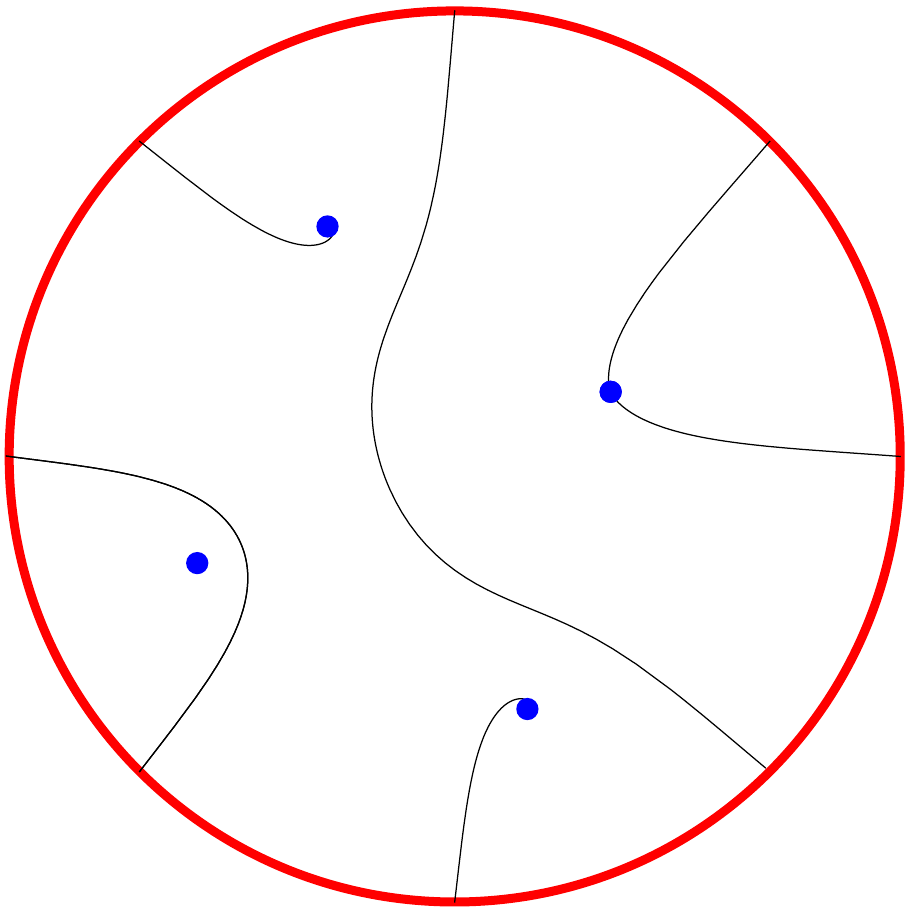}}}}\qquad 
\subfigure[$\alpha=\frac{\pi}{2k}$]{\reflectbox{\rotatebox[origin=c]{180}{\includegraphics[width=3.5cm]{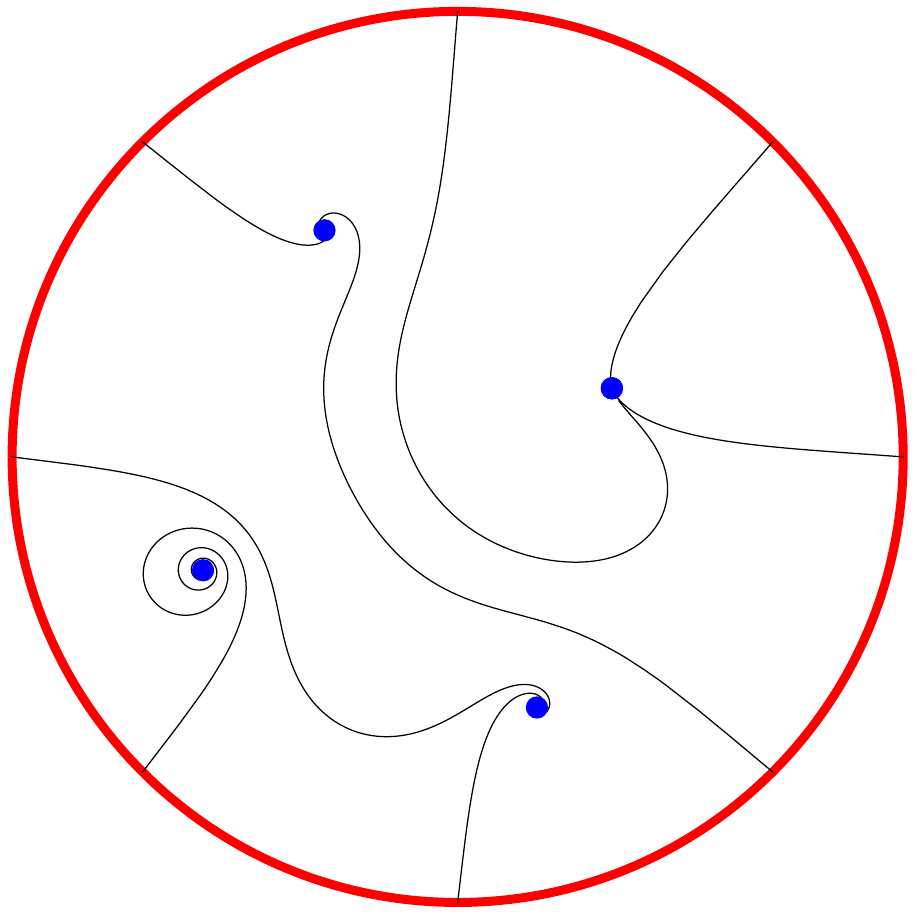}}}}
\caption{For $k=4$, the bifurcation of the separatrices at the parabolic point when $\alpha=\frac{\pi}{2k}$.}\label{fig:bif_parabolic}
\end{center}\end{figure}

\begin{figure}\begin{center}
\subfigure[Non adjacent homoclinic loops]{\includegraphics[height=4.5cm]{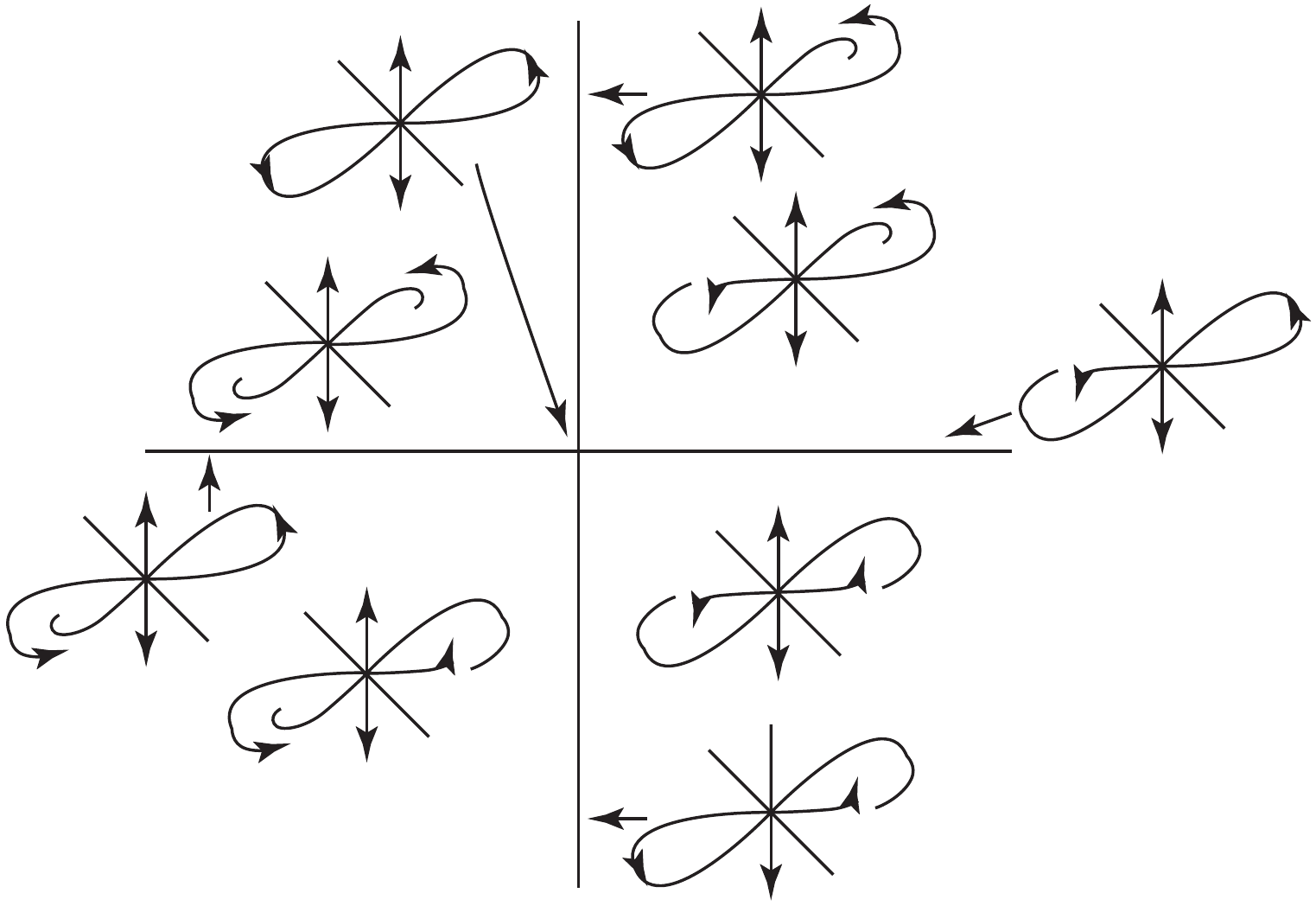}}\qquad
\subfigure[Adjacent homoclinic loops]{\includegraphics[width=5.2cm]{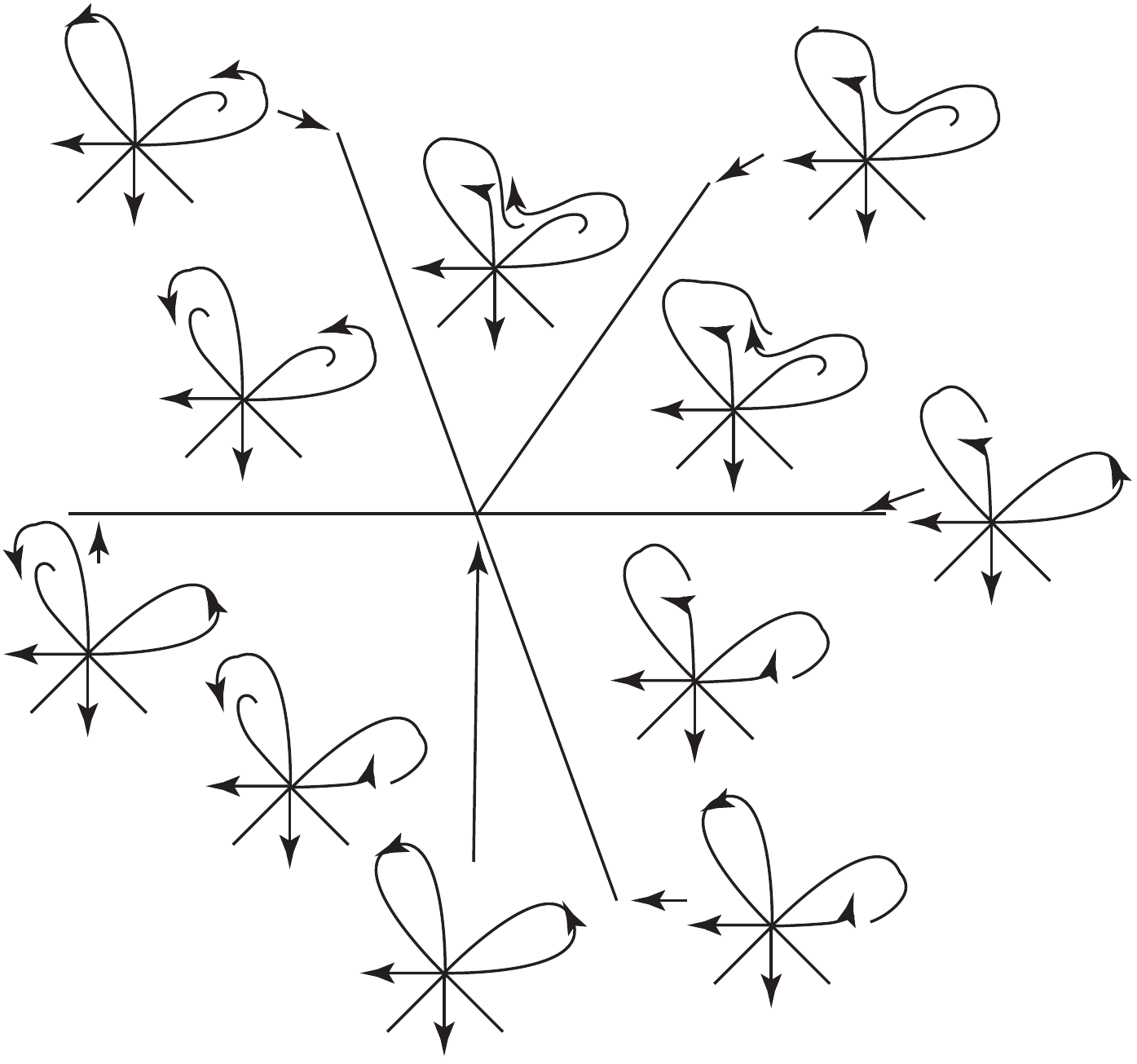}}
\caption{Bifurcations of double homoclinic loops: the bifurcation diagram includes other homoclinic loops when the two homoclinic loops are adjacent. Case (a) occurs when two segments joining vertices of the periodgon are parallel, while case (b) occurs when three vertices of the periodgon are aligned.}\label{fig:double_homoclinic}
\end{center}\end{figure}

\begin{figure}\begin{center}\hskip-24pt
\subfigure[$s=0$]{\rotatebox[origin=c]{135}{\includegraphics[width=3.5cm]{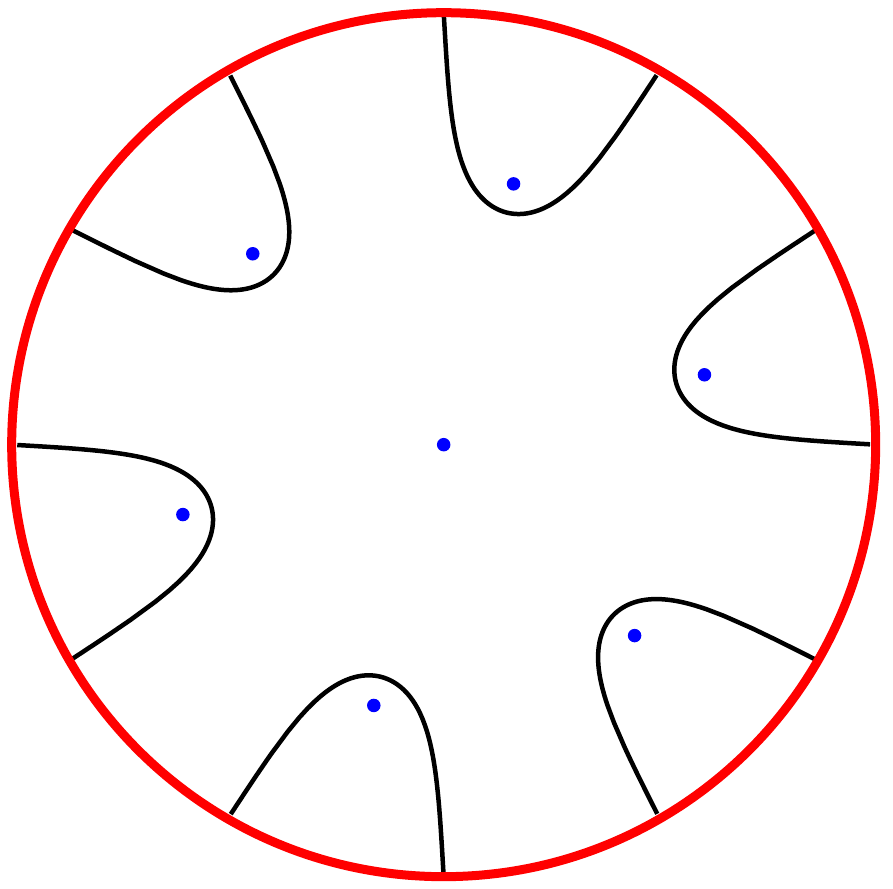}}}\hskip-24pt
\subfigure[$s$ close to $0$]{\rotatebox[origin=c]{135}{\includegraphics[width=3.5cm]{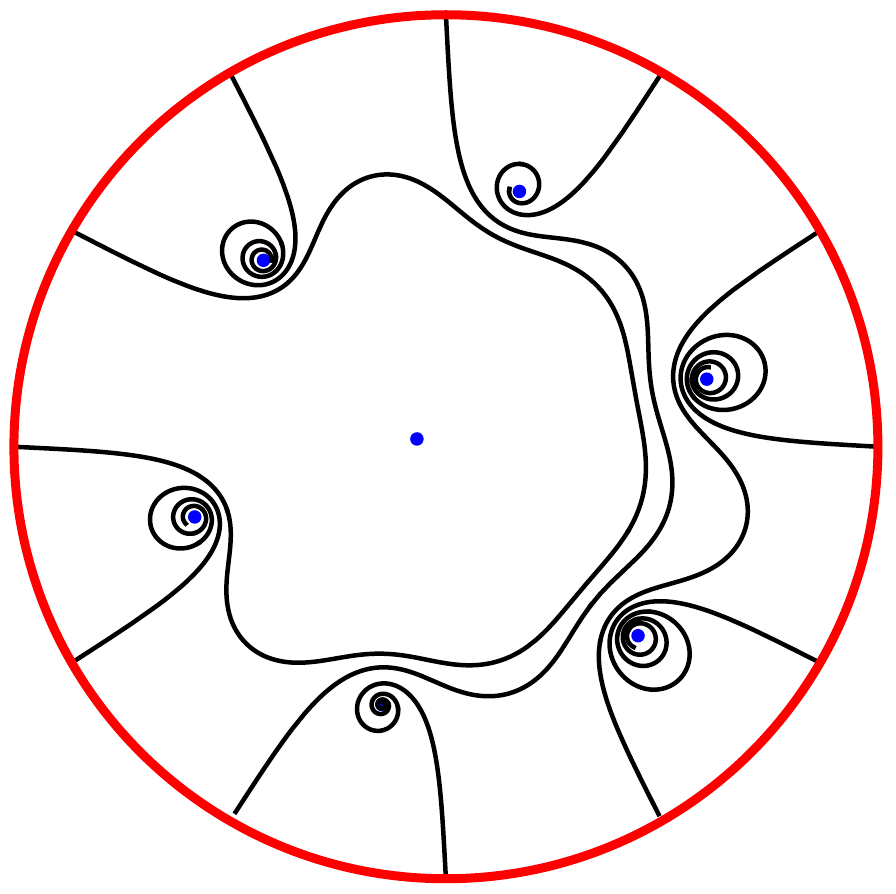}}}\hskip-24pt
\subfigure[$s=1$]{\rotatebox[origin=c]{135}{\includegraphics[width=3.5cm]{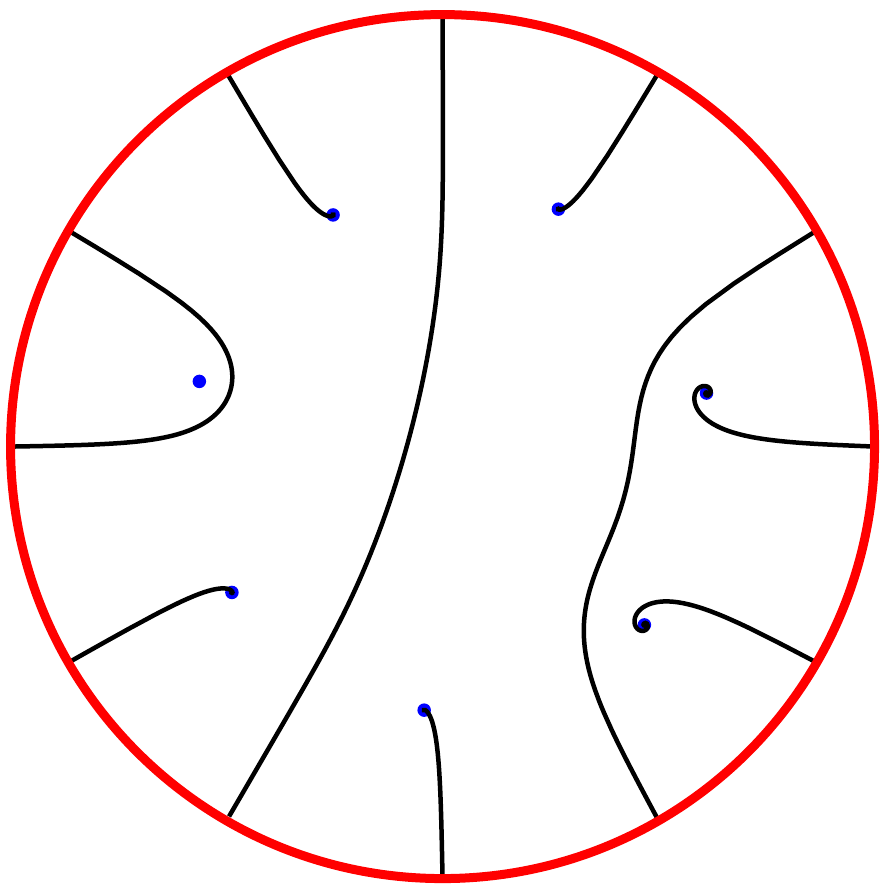}}}\hskip-24pt
\caption{Simultaneous homoclinic loops for $k=6$, $\theta=-\frac{\pi}{6}$  and $\alpha= \frac{\pi}{12}$. }\label{fig:reversible}
\end{center}\end{figure}
 
\begin{figure}\begin{center}\hskip-24pt
\subfigure[$s=0$]{\rotatebox[origin=c]{-165}{\includegraphics[width=3.5cm]{fig/reversible1}}}\hskip-6pt
\subfigure[$s$ close to $0$]{\rotatebox[origin=c]{-165}{\includegraphics[width=3.5cm]{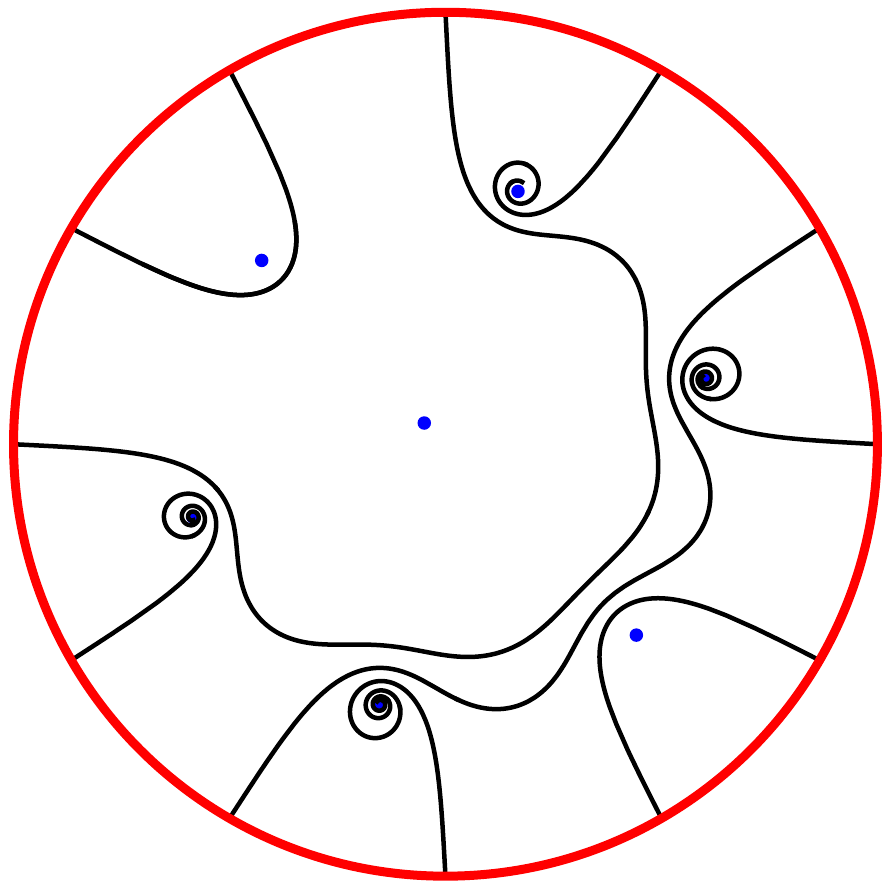}}}\hskip-6pt
\subfigure[$0<s<\frac12 $]{\rotatebox[origin=c]{-165}{\includegraphics[width=3.5cm]{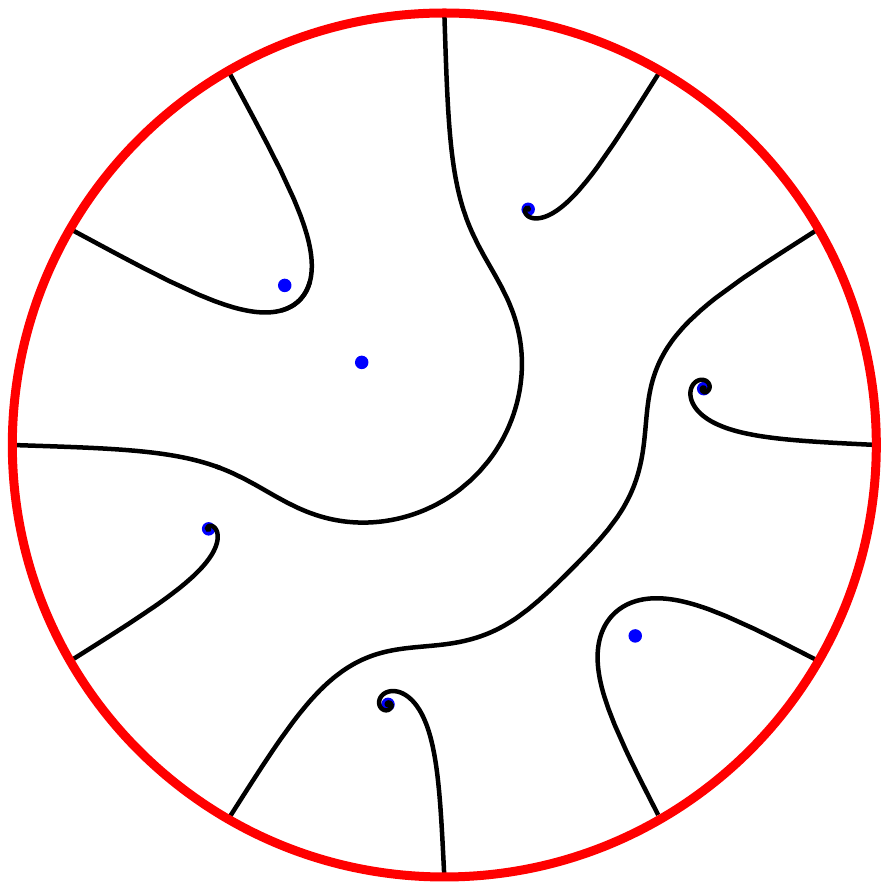}}}\hskip-24pt\\ 
\subfigure[$\frac12<s<1$]{\rotatebox[origin=c]{-165}{\includegraphics[width=3.5cm]{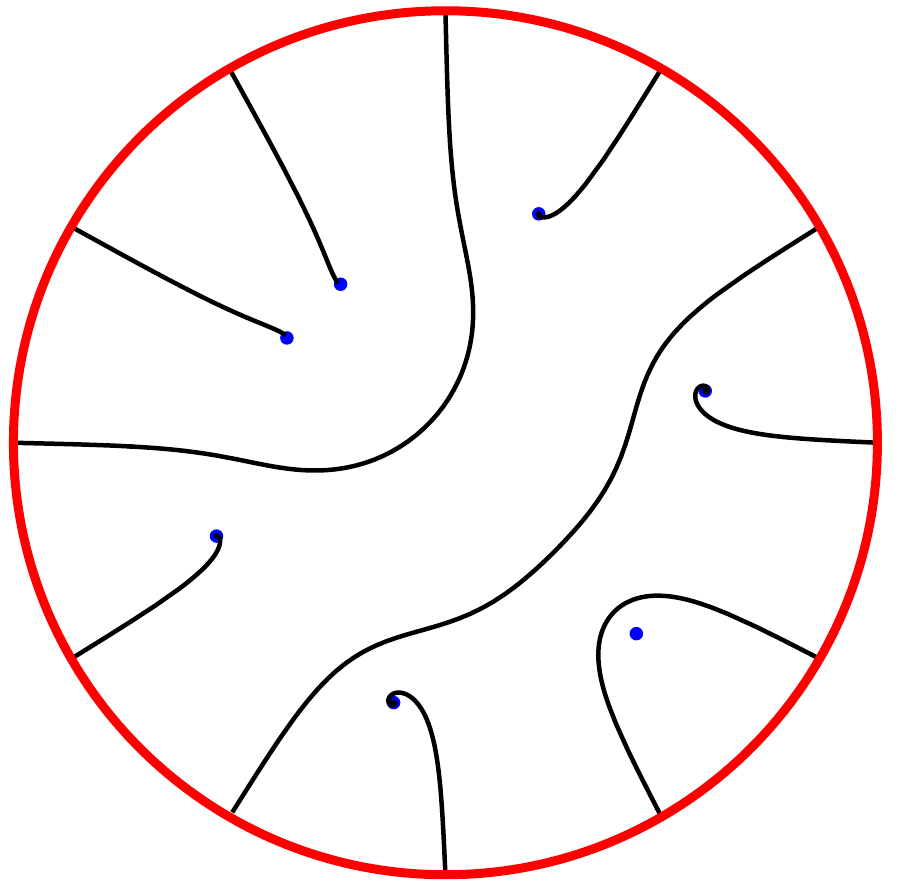}}}\hskip-6pt
\subfigure[$s=1$]{\rotatebox[origin=c]{-165}{\includegraphics[width=3.5cm]{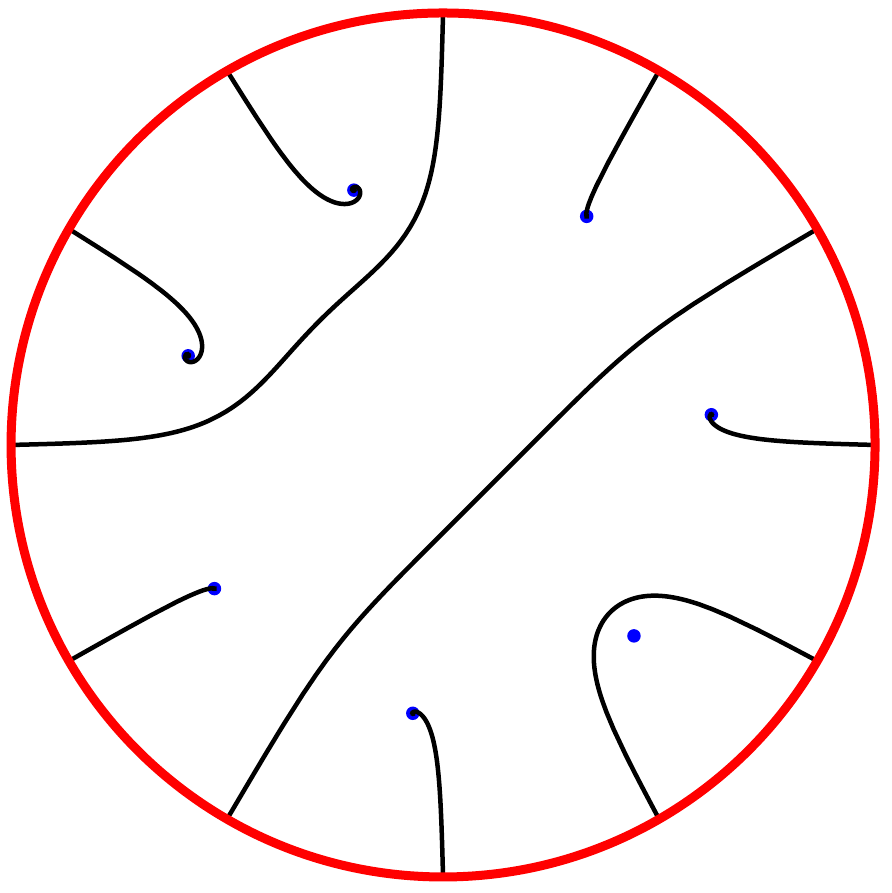}}}
\caption{Simultaneous homoclinic loops for $k=6$,  $\theta=0$ and $\alpha= \frac{\pi}{12}$. Note the passage through the parabolic point between (c) and (d).}
\label{fig:reversible_theta}
\end{center}\end{figure}

\begin{figure}\begin{center} 
\subfigure[$\theta=0_-$, $s<\frac12$]{\quad\includegraphics[scale=0.4]{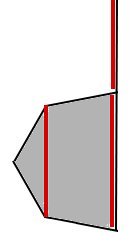}\quad}
\subfigure[$\theta=0$, $s>\frac12$]{\quad\includegraphics[scale=0.4]{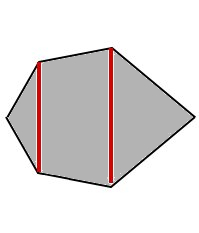}\quad}
\subfigure[$\theta=\frac\pi{k}$]{\quad\includegraphics[scale=0.4]{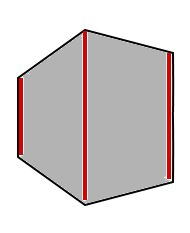}\quad}\\
\subfigure[$\theta=0_-$, $s<\frac12$]{\quad\includegraphics[scale=0.4]{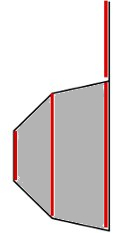}\quad}
\subfigure[$\theta=0$, $s>\frac12$]{\quad\includegraphics[scale=0.4]{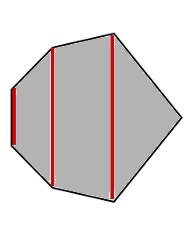}\quad}
\subfigure[$\theta=\frac\pi{k}$]{\quad\includegraphics[scale=0.4]{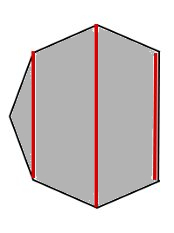}\quad}\\
\caption{
The periodgon for $\alpha=0$, $\theta=\frac{\pi l}{k}$, $k$ odd, resp. even. After a rotation by $\frac{(2m+1)\pi}{2}$, i.e. for $\alpha=\frac{(2m+1)\pi}{2k}$, 
the fat red lines will correspond to homoclinic orbits. Cases (a) and (d) are for $(s,\theta)$ below the slit.}
\label{fig:homoclinicperiodgon}
\end{center}\end{figure}

\section{The regularity of the singular points} 

In this section we first analyze the position of the singular points during the transition from \eqref{family0} to \eqref{family1} when $s$ varies in $[0,1]$ depending on the value of $\theta$. A natural domain for doing this will be the slit disk $|se^{i\theta}|\leq1$. We will then apply this to the analysis of the regularity of the periodgon in Section~\ref{sec:periodgon}. 
\medskip

Let $z_0(s,\theta,\alpha), \dots z_{k}(s,\theta,\alpha)$ be the singular points of \eqref{3_par}
depending continuously on the parameters,
such that for $s=1$
$$z_j(1,\theta,\alpha)=k^{\frac{1}{k+1}}e^{\frac{i\theta+\pi i (2j-1)}{k+1}-i\alpha},\qquad  j\in\Z_{k+1}.$$
Hence
$$ z_j(s, \theta-\tfrac{2\pi}k ,\alpha)=z_{j+1}(s,\theta,\alpha).$$
The following proposition shows that $z_0(s,\theta,\alpha), \dots z_{k}(s,\theta,\alpha)$ are ordered by their argument for all $(s,\theta)$ from the closure of the slit disk.

\medskip
Note that, up to a rotation by a $\beta=\alpha+\frac{2\ell}k\pi$ for some $\ell\in\Z$, that is up to a transformation \eqref{linear_transf}
with $A=e^{i\beta}$, 
it is sufficient to consider the singular points of \eqref{3_par} with $\alpha=0$, $\theta\in[-\frac\pi k,\frac\pi k]$:
\begin{equation}\dot z = z^{k+1}-(k+1)(1-s)^k z+ks^{k+1}e^{i\theta}.\label{vf_theta}\end{equation}
And after a reflection $(z,t)\to(\bar z,\bar t)$,
we can restrict the parameters from the closure of the slit disk to the following \emph{fundamental sector}  
\begin{equation}  s\in [0,1]\quad \text{and}\quad  \theta \in [-\tfrac{\pi}k,0], \quad\alpha=0. \label{fund_sector}\end{equation}

\begin{proposition} We consider \eqref{vf_theta} with  $\theta \in [-\tfrac{\pi}k,0]$ and $s\in[0,1]$. 

\begin{enumerate}
\item The singular points  $z_0(s,\theta,0),\ldots z_{k+1}(s,\theta,0)$ have distinct arguments for all $s\in [0,1]$, unless $\theta=0$, in which case the two roots $z_0(s,0,0)$ and $z_1(s,0,0)$ have both zero argument for $s\leq\frac{1}{2}$. 
\item If $z_j(s,\theta,0)\notin e^{i\theta}\R$, then the absolute value of $\arg(e^{-i\theta}z_j(s,\theta,0))\in(-\pi,\pi)$ increases monotonically with $s$.
\item This implies that 
$$z_0(0,\theta,0)=0,\quad\text{and}\quad z_j(0,\theta,0)=(k+1)^{\frac{1}{k}}e^{\frac{2\pi i (j-1)}{k}},\ j=1,\ldots,k,$$
and the roots are caught for all $s\in[0,1]$ in the following disjoint sectors:
\begin{align*}
 \arg z_0(s,\theta,0)&\in[\tfrac{\theta-\pi}{k+1},\theta],\\
 \arg z_j(s,\theta,0)&\in\left[\tfrac{2\pi(j-1)}{k},\tfrac{\theta+(2j-1)\pi}{k+1}\right],\quad\text{for } 1\leq j\leq\tfrac{k+1}{2},\\
 \arg z_j(s,\theta,0)&\in\left[\tfrac{\theta+(2j-1)\pi}{k+1},\tfrac{2\pi(j-1)}{k}\right],\quad\text{for } \tfrac{k+2}{2}\leq j\leq k.
\end{align*}
\item For $\theta=0$, the two roots $z_0(s,0,0)$ and $z_1(s,0,0)$ are real for $s\in [0,\frac{1}{2}]$ and merge for $s=\frac{1}{2}$. For $s>\frac{1}{2}$, they then split apart in the imaginary direction. 
\end{enumerate} 
\end{proposition}

\begin{proof} \begin{enumerate}
\item Let us suppose that the distinct points $z_j(s,\theta,0)$ and $z_\ell(s,\theta,0)$ have the same argument for some value $s$: $z_j(s,\theta,0)= r_j e^{i\phi}$ and $z_\ell(s,\theta,0)= r_\ell e^{i\phi}$. Then
$(r_j^{k+1}-r_\ell^{k+1})e^{ik\phi} -(k+1)(1-s)^k(r_j-r_k)=0.$ Hence, necessarily  $e^{ik\phi}=1$. Replacing in $\dot z|_{z=z_j(s,\theta,0)}=0$ yields
$(r_j^{k+1} -(k+1)(1-s)^kr_j)e^{i\phi} +ks^{k+1}e^{i\theta}=0$, and hence $e^{i\theta}= \pm e^{i\phi}=\pm 1$. Since $\theta \in [-\pi/k,0]$, then necessarily $\theta=0$. 
Hence $\phi=0,\pi$ if $k$ is even and $\phi=0$ if $k$ is odd. 
Using Descartes' rule of signs, there are at most three real roots. There is one negative root for $k$ even and none for $k$ odd. As for the positive roots, there are two (resp zero) positive roots for $s\leq \frac{1}{2}$ (resp. $s> \frac{1}{2}$).
\item 
Denote $P(z):=z^{k+1}-(k+1)(1-s)^k z+ks^{k+1}e^{i\theta}$.
Differentiating $P(z_j(s,\theta,0))=0$ with respect to $s$ gives
\begin{equation}\label{dz_jds}
 \lambda_j\frac{dz_j}{ds}+k(k+1)\big((1-s)^{k-1}z_j+s^ke^{i\theta}\big)=0,
\end{equation}
where
\begin{equation}\label{lambda}
 \lambda_j=P'(z_j)=(k+1)\big(z_j^k-(1-s)^k\big)=k(k+1)\big((1-s)^k-s^{k+1}\tfrac{e^{i\theta}}{z_j}\big),\quad\text{if}\ z_j\neq 0.
\end{equation}
Let $x_j=e^{-i\theta}z_j$, then
$$
 \frac{dx_j}{ds}\cdot \left(\tfrac{s^{k+1}}{x_j}-(1-s)^k\right)=(1-s)^{k-1}x_j+s^k, $$ and 
\begin{align} \begin{split} (1-s)\frac{dx_j}{ds}&=-x_j+\frac{s^kx_j}{s^{k+1}-(1-s)^kx_j} \\
&=-x_j+s^k x_j \frac{s^{k+1}-(1-s)^k\bar x_j}{\left|s^{k+1}-(1-s)^kx_j\right|^2}. \label{dxds}\end{split}
\end{align}

Therefore $\arg x_j$ increases with $s$ if and only if $\Im(x_j)>0$, i.e. if $0<\arg x_j<\pi$, and the argument $\arg x_j$ decreases if $-\pi<\arg x_j<0$, and stays constant if $x_j\in\R$.
\item Follows from (2).
\item For $\theta=0,s=\frac{1}{2}$, $z_j(\frac{1}2,0,0)=\frac{1}2$, $j=0,1$, and the vector field \eqref{dz_jds} has a  parabolic singularity at $z=\frac12$. Letting $w=z-\frac12$ and $s'=s-\frac12$ yields $\dot w= \frac{k(k+1)}{2^k}w^2+\frac{k(k+1)}{2^{k-1}}s' +O(s'w)+O(w^3)+O({s'}^2)$, yielding that the two singular points split in the real (resp. imaginary) direction when $s'<0$ (resp. $s'>0$). 
\end{enumerate} 
\end{proof}

\begin{lemma}\label{norm:sp} Let $z_j(s,\theta,0)$ and $z_\ell(s,\theta,0)$ be two distinct singular points of \eqref{vf_theta}. 
Then, for $s\in (0,1)$, $|z_j|=|z_\ell|$ if and only if 
$e^{ik\theta}\in\R$ and $z_j$ and $z_\ell$ are symmetric with respect to the line $e^{i\theta}\R$.
\end{lemma}
\begin{proof}
If $|z_j|=|z_\ell|$ and $z_j$ and $z_\ell$ are roots of \eqref{vf_theta}, then 
$|(k+1)(1-s)^kz_j-ks^{k+1}e^{i\theta}|=|(k+1)(1-s)^kz_\ell-ks^{k+1}e^{i\theta}|$. 
If $s(1-s)\neq0$, then this means that $z_j$ and $z_\ell$ lie on an intersection of two circles with centers at 0 and $\frac{ks^{k+1}e^{i\theta}}{(k+1)(1-s)^k}$, whence they
are symmetric with respect to the line $e^{i\theta}\R$.
Let $x_j=e^{-i\theta}z_j$, $x_\ell=e^{-i\theta}z_\ell$, $x_\ell=\bar x_j$, 
then 
$e^{ik\theta}(x_j^{k+1}-\bar x_j^{k+1})=(k+1)(1-s)^k (x_j-\bar x_j)$, from which $e^{ik\theta}\in\R$.\end{proof}

\begin{corollary} The order of magnitude of the $|z_j|$ cannot change on the interior of the fundamental sector \eqref{fund_sector}.\end{corollary}

\begin{proposition}\label{prop:norm_sp} 
\begin{enumerate}
\item 
Let $\Re_\theta(z):=\Re(e^{-i\theta}z)$ be the signed projection of a point $z$ on the oriented line $e^{i\theta}\R$.
For $(s,\theta)$ in the interior of the slit disk, 
$$|z_j(s,\theta,0)|\leq |z_\ell(s,\theta,0)|\quad\text{if and only if}\quad \Re_\theta(z_j(1,\theta,0))\geq \Re_\theta(z_l(1,\theta,0)).$$
\item
The magnitude $|z_j(s,\theta,0)|$ decreases with $\theta$ if $\Im_\theta(z_j)>0$ and increases with $\theta$ if $\Im_\theta(z_j)< 0$.
\item
In particular, the order of $|z_j|$ by their magnitude  on the interior of the fundamental sector $\theta\in(-\tfrac{\pi}k,0)$, is the same given by
$$\begin{array}{l llll}
 &    |z_0(s,-\tfrac{\pi}k,\alpha)| 		&< |z_0(s,\theta,\alpha)| \hskip-6pt&< |z_0(s,0,\alpha)| &\leq \\
\leq\hskip-6pt& |z_1(s,0,\alpha)| 			&< |z_1(s,\theta,\alpha)| \hskip-6pt&< |z_1(s,-\tfrac{\pi}k,\alpha)| \hskip-6pt&=\\
=\hskip-6pt&    |z_k(s,-\tfrac{\pi}k,\alpha)| 		&< |z_k(s,\theta,\alpha)| \hskip-6pt&< |z_k(s,0,\alpha)| &=\\
=\hskip-6pt&    |z_2(s,0,\alpha)| 			&< |z_2(s,\theta,\alpha)| \hskip-6pt&< |z_2(s,-\tfrac{\pi}k,\alpha)| \hskip-6pt&=\\
=\hskip-6pt&    |z_{k-1}(s,-\tfrac{\pi}k,\alpha)|\hskip-6pt	&< \ldots.&&
\end{array}$$
\end{enumerate}
\end{proposition}

\begin{proof} 
\begin{enumerate}
\item 
By Lemma~\ref{norm:sp} the order of $|z_j(s,\theta,0)|$ by their magnitude for $\theta\notin\frac{\pi}{k}\Z$ is the same for all $s\in(0,1)$, and
can be computed from the derivatives of $z_j$ with respect to $s$ at $s=1$ where all $|z_j|$ are equal. 
From \eqref{dxds}
\begin{align*}
 (1-s)\frac{d|z_j|^2}{ds}+|z_j|^2&=2s^k|z_j|^2 \frac{s^{k+1}-(1-s)^k\Re_\theta(z_j)}{\left|s^{k+1}-(1-s)^ke^{-i\theta}z_j\right|^2}\\
&=\frac{2|z_j|^2}{s^{k+2}} \big[s^{k+1}+(1-s)^k\Re_\theta(z_j)+O((1-s)^{2k})\big],
\end{align*}
and the statement follows.
Finally note that it holds also for $\theta=\frac{2m\pi}{k}$, $s\in(\frac{1}{2},1)$, and for $\theta=\frac{(2m-1)\pi}{k}$, $s\in(0,1)$.
\item
The derivative of $x_j=e^{-i\theta}z_j$ with respect to $\theta$ is
\begin{equation}\label{dxdtheta}
 \frac{dx_j}{d\theta}=-ix_j-\frac{ix_js^{k+1}}{(k+1)[(1-s)^kx_j-s^{k+1}]},
\end{equation}
therefore 
$$\frac{d|x_j|^2}{d\theta}=-2|x_j|^2\frac{s^{k+1}(1-s)^k\Im x_j}{(k+1)|(1-s)^kx_j-s^{k+1}|^2}.$$
\end{enumerate}
\end{proof}

\section{The regularity of the periodgon} \label{sec:periodgon}

\subsection{The ad hoc periodgon}

The homoclinic loops of each simple singular point  (see Section~\ref{section:periodgon}) depend continuously on the parameter and can change their location only when the parameter crosses the bifurcation locus $\Sigma$ (see Proposition~\ref{prop: bif_shape}).
In terms of the periodgon, the crossing of $\Sigma$ corresponds to either a vertex of the periodgon crossing an edge, or two sides becoming infinite.
As studied in \cite{CR}, we know that for $s=1$ the edges of the periodgon are given  in the clockwise order by $\nu_k, \nu_{k-1}, \dots, \nu_1, \nu_0$ (see Example~\ref{example:periodgon}).   
We will consider the polygon formed this way for all $(s,\theta)$ in the interior of the slit disk,  and we will call it  the \emph{ad hoc periodgon}. 
If we show that this polygon has no self-intersection, then this will mean that it is indeed the \emph{intrinsic} periodgon defined in Definition~\ref{def:periodgon}. We will be able to show that this is the case on the boundary of the slit disk (the proof is only numerical along the slits $s\in [0,\frac12]$, $\theta=\frac{2\pi m}{k}$). Hence if the ad hoc periodgon were to have self-intersection, then this would only occur in isolated islands.

\begin{conjecture}\label{conjecture_2} 
It it conjectured that  the ad hoc periodgon has no self-intersection when $s\neq0$ and $(s,\theta)\notin(0,\tfrac12)\times\frac{2\pi}{k}\Z$.
This is equivalent to Conjecture~\ref{conjecture_0}. In particular this means that the (intrinsic) periodgon of Definition~\ref{def:periodgon} is the same as the ad hoc periodgon.
\end{conjecture}

\begin{proposition}\label{prop:period_gon} The intrinsic periodgon and ad hoc periodgon (see Figure~\ref{fig:period}) 
have the following properties:
\begin{enumerate}
\item For fixed $s$ and $\theta$, when $e^{i\alpha}$ rotates, the periodgon keeps a fixed shape and rotates at rate $e^{ ik\alpha}$.
\item It is a regular $(k+1)$-gon when $s=1$ (see Figures~\ref{theta_k}(e) and \ref{periodgon_theta_0}(e)). 
\item It is degenerate when $s=0$, with $k$ equal sides in the direction $e^{ ik\alpha}i\R^+$ and one side ($k$ times larger) in the opposite direction (See Figures~\ref{theta_k}(a) and \ref{periodgon_theta_0}(a)).
\item When $\theta=\frac{2m-1}{k}\pi$, the periodgon is symmetric with respect to the axis $e^{ik\alpha}\R$ (see Figure~\ref{theta_k}), which passes through the center of the side $\nu_m$. \item When $\theta=\frac{2m}{k}\pi$ and $s\in[\tfrac12,1]$, the periodgon is symmetric with respect to the axis $e^{ik\alpha}\R$ (see Figure~\ref{periodgon_theta_0}(d)), which passes through the vertex in between the sides $\nu_m$ and $\nu_{m+1}$, $m\in\Z_{k+1}$. 
\item 
Its only self-intersection on $\theta=\frac{2m\pi}{k}$ and $s\in (0,\frac12)$ comes from the alignment of  $\nu_m$ and $\nu_{m+1}$. 

On the slit, $s\in[0,\frac12)$ and $\theta=\tfrac{2m}{k}\pi^\pm$, the periodgon has two consecutive sides 
$\nu_m(s,\tfrac{2m}{k}\pi^-,\alpha)=\nu_{m+1}(s,\tfrac{2m}{k}\pi^+,\alpha)$ and $\nu_{m+1}(s,\tfrac{2m}{k}\pi^-,\alpha)=\nu_m(s,\tfrac{2m}{k}\pi^+,\alpha)$, $m\in\Z_{k+1}$, aligned and of opposite orientation, i.e. one is a part of the other (see Figure~\ref{periodgon_theta_0}(b)) and the periodgon has no other self-intersection. 
There are two different configurations depending on which side of the slit we are. 
The length of these consecutive sides becomes infinite when $s\to \frac12^-$, however their sum $\nu_m+\nu_{m+1}$ remains bounded and continuous, and is equal at the limit to \eqref{nu_par}.
The rest of the periodgon, consisting of the consecutive sides $\nu_{m-1},\nu_{m-2},\ldots,\nu_{m-k}$, $m\in\Z_{k+1}$, is symmetric with respect to the axis $e^{ik\alpha}\R$. 
\item For $s=\frac12$ and $\theta=0$, the degenerate ad hoc periodgon has one side less, the sides $\nu_m, \nu_{m+1}$ being replaced by 
\begin{equation}\nu_{par}=\res_{z_m=z_{m+1}}\frac{2\pi i}{P_\eps(z)}=-e^{ ik\alpha}\frac{4\pi i(k-1)}{3k(k+1)} 2^k,\label{nu_par}\end{equation}
and it has no self-intersection. The periodgon is symmetric with respect to the axis $e^{ik\alpha}\R$, and all its sides except for $\nu_{par}$ have positive projection on $e^{ik\alpha+\frac{\pi}2}\R^+$ (Figure~\ref{periodgon_theta_0}(c)).
\end{enumerate} \end{proposition}

\begin{proof}
\begin{enumerate}
\item This comes from Remark~\ref{remark:rotation}. \item See Example~\ref{example:periodgon}.
\item When $\alpha=0$, $P_\eps'(z_j)= (k+1)k$, for $j\neq0$ and $P_\eps'(0)=-(k+1)$. 
\item The symmetry of the ad hoc periodgon  comes from the symmetry of the vector field with respect to the invariant line $e^{\frac{2m-1}{k}\pi i}\R$ for $\theta=\frac{2m-1}{k}\pi$, $\alpha=0$ (see Proposition~\ref{prop:symmetry}).
The intrinsic periodgon and the ad hoc periodgon are the same for these parameters: 
indeed they are the same for $s=1$ and if there were a bifurcation of the intrinsic periodgon, it would happen through one of its vertices crossing its side (Proposition~\ref{prop:multipleends}), however this cannot happen because the symmetry would force this to happen through a merging of vertices which is impossible (Proposition~\ref{prop:periodgon}).   
\item Same kind of argument as in the previous case.  
\item When $s\in[0,\frac12)$, $\theta=0^\pm$ and $\alpha=0$, $P_\eps$ has two real positive roots $0\leq z_0(s,0^-,0)=z_1(s,0^+,0)<\frac12<z_1(s,0^-,0)=z_0(s,0^+,0)$ 
with periods oriented in the opposite direction. The rest of the ad hoc periodgon is symmetric because the vector field is.
The only way how the intrinsic periodgon could bifurcate and differ from the ad hoc periodgon would be by changing the order of the two periods $\nu_0(s,0^-,0)=\nu_1(s,0^+,0)$ and $\nu_1(s,0^-,0)=\nu_0(s,0^+,0)$. In fact the periodic domain of $\nu_0(s,0^-,0)=\nu_1(s,0^+,0)$ has two ends at $\infty$ and the choice where the cut is made is such that it would depend continuously on $\theta\to 0^\pm$. 
\item For $s=\frac12$, $\theta=0$ and $\alpha=0$, all $z_j$, $j>1$ are repelling. Indeed, $z_0=z_1=\frac12$, yielding by Proposition~\ref{prop:norm_sp} that all $z_j$, $j\geq2$,  
satisfy $|z_j|>\frac12.$ The corresponding eigenvalues $\lambda_j=k(k+1)\left(\frac1{2^k}-\frac1{2^{k+1}z_j}\right)$ then have positive real part. 
The imaginary part of $\lambda_j$ has the inverse sign of that of $z_j$. This yields that for the $z_j$ in the upper (resp. lower) half-plane, then the corresponding period is oriented to the inner of the first (resp. second quadrant). If $z_j\in\R^-$, then its period is in $i\R^+$. This prevents any self-intersection. 

The period 
$\nu_{par}=2\pi i\,\res_{z=\frac{e^{i\alpha}}2}\frac{dz}{P_{(\frac12,0,\alpha)}(z)}$ is calculated using the factorization \eqref{parabolic_alpha}.
\end{enumerate}\end{proof}

\begin{figure}\begin{center}
\subfigure[\!$s\!=\!0$]{\includegraphics[scale=0.3]{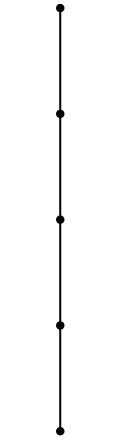}}
\subfigure[]{\includegraphics[scale=0.3]{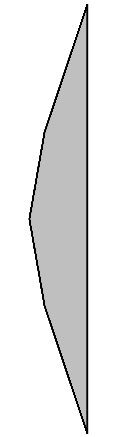}}\ 
\subfigure[]{\includegraphics[scale=0.3]{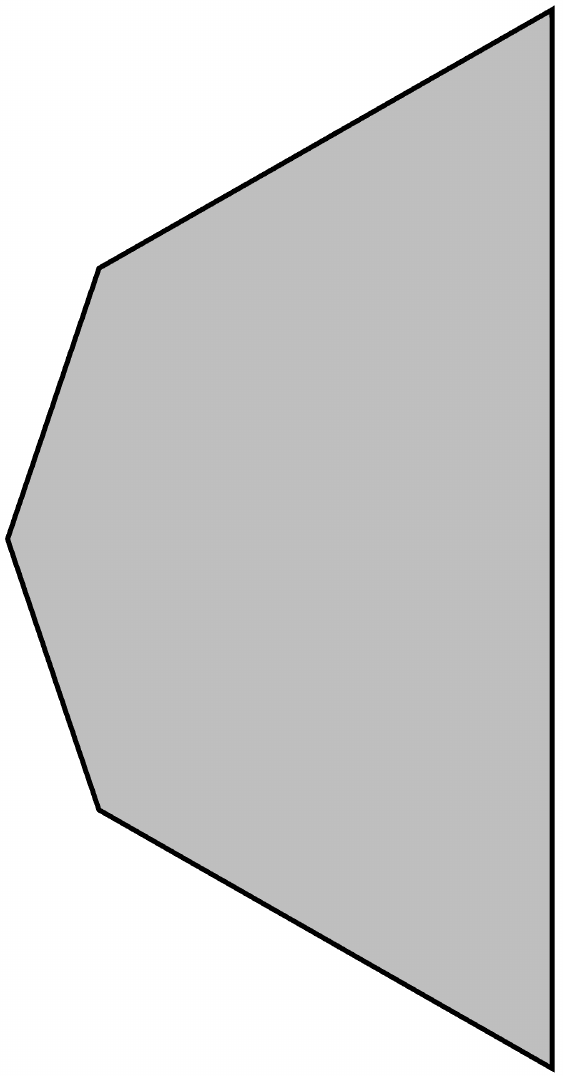}}\quad
\subfigure[]{\includegraphics[scale=0.3]{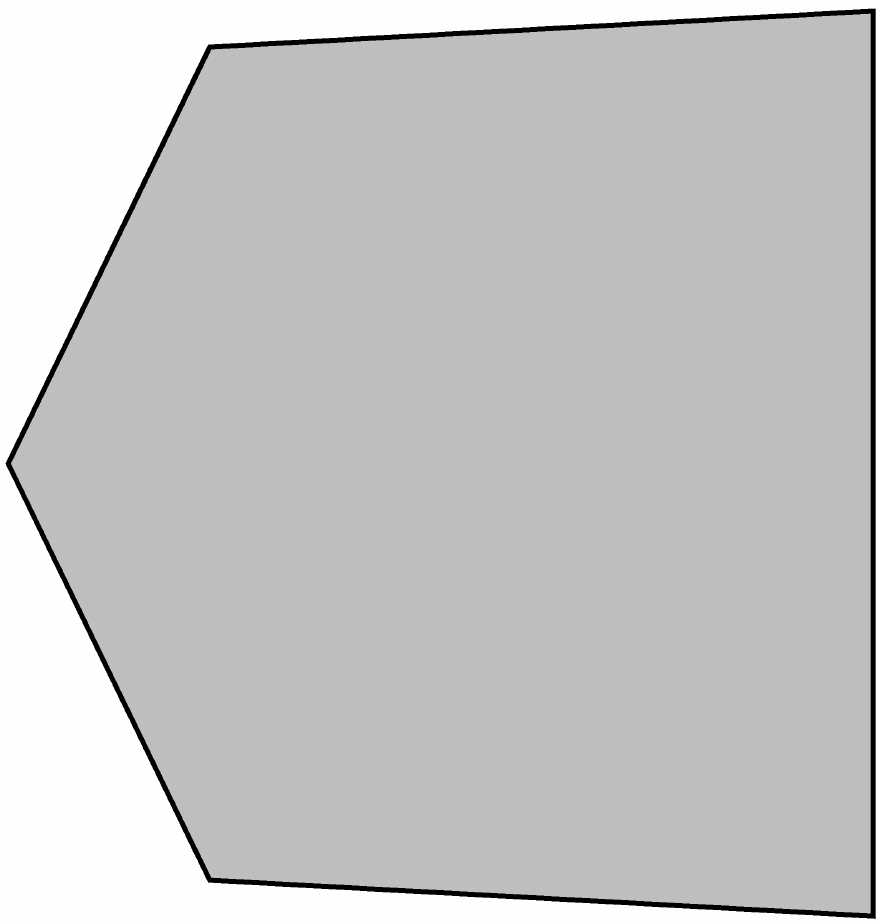}}\quad
\subfigure[\!$s\!=\!1$]{\includegraphics[scale=0.3]{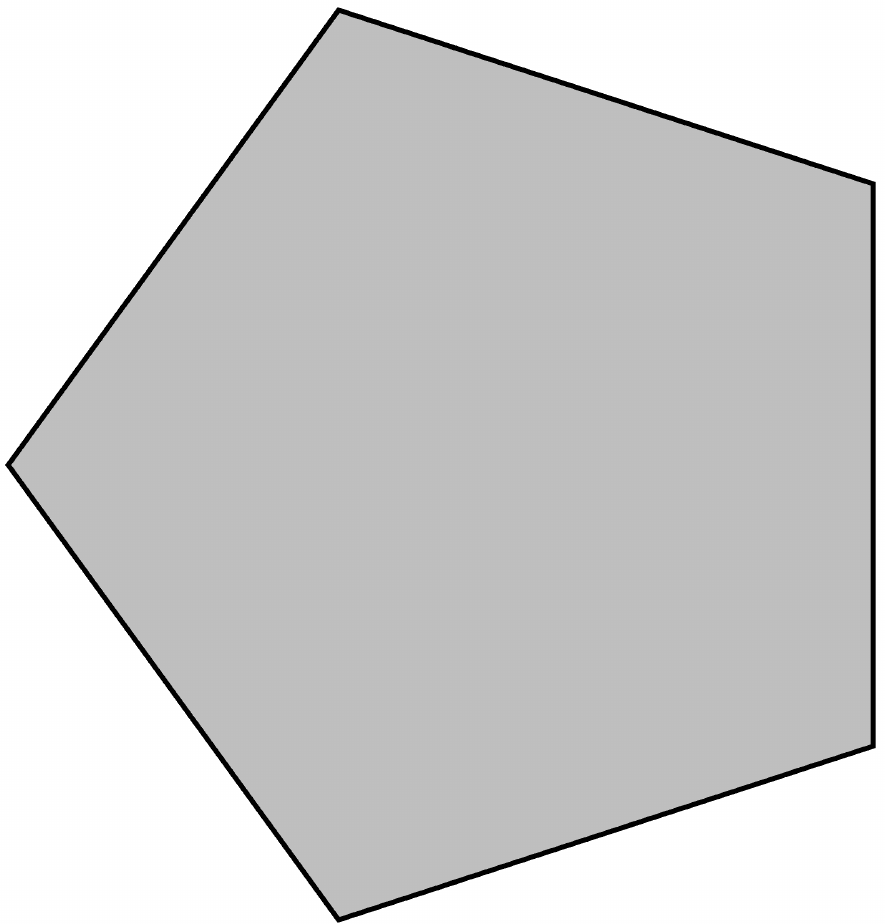}}
\caption{The periodgon for $k=4$, $\theta=\frac{\pi}{k}$, $\alpha=0$, and increasing values of $s\in [0,1]$. The ad hoc periodgon is symmetric.
}\label{theta_k}\end{center}\end{figure}

\begin{figure}\begin{center}
\subfigure[\!$s\!=\!0$]{\includegraphics[scale=0.25]{fig/theta_0_0.jpeg}}\ 
\subfigure[$\!0\!<\!s\!<\!\frac12$\hskip-6pt]{\includegraphics[scale=0.4]{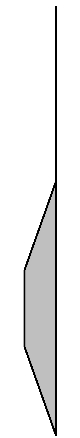}}
\subfigure[\!$s\!=\!\frac12$]{\includegraphics[scale=0.22]{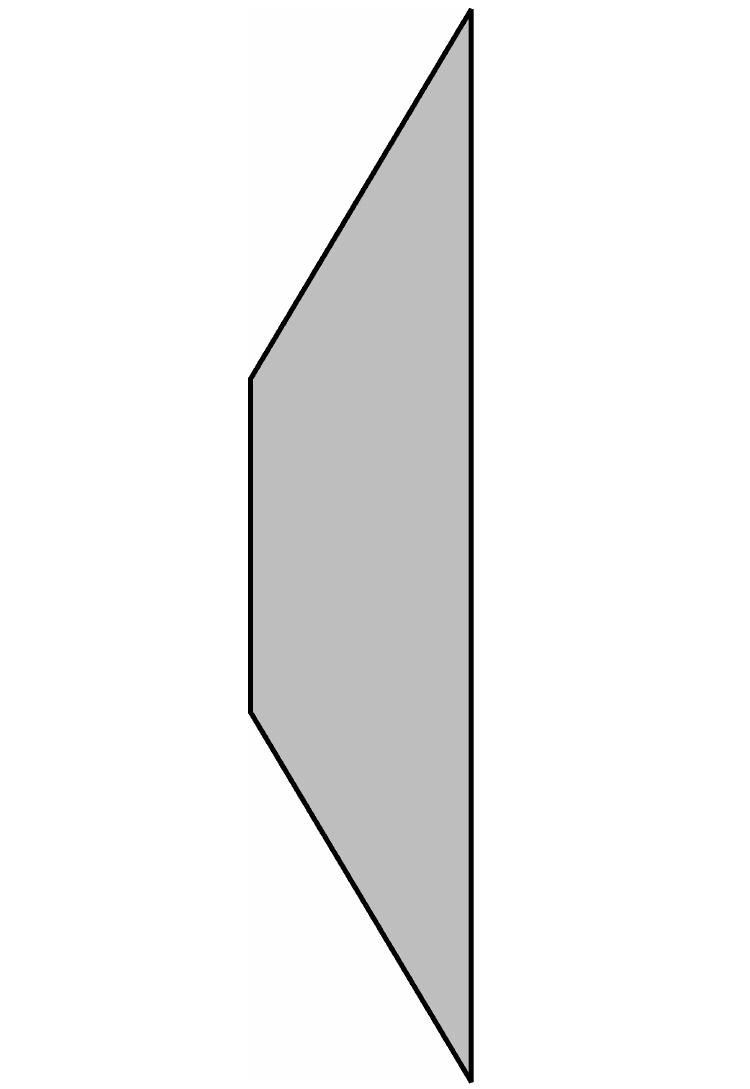}}\hskip-6pt 
\subfigure[\!$\frac12\!<s\!<\!1$]{\includegraphics[scale=0.55]{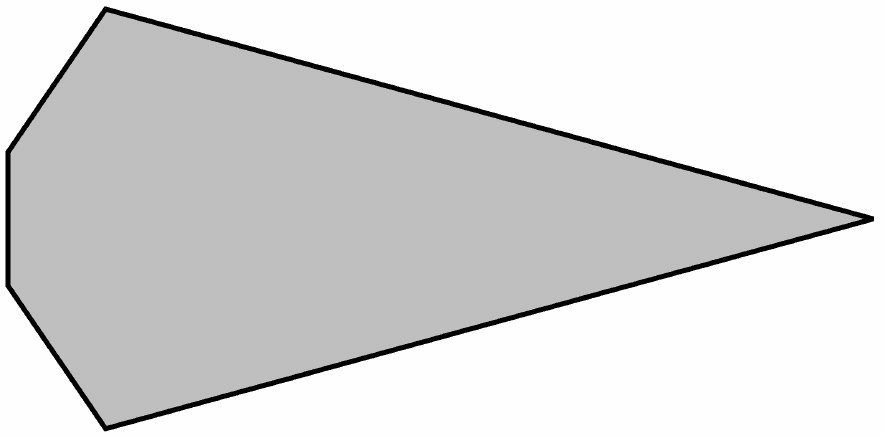}}\ 
\subfigure[\!$s\!=\!1$]{\includegraphics[scale=0.25]{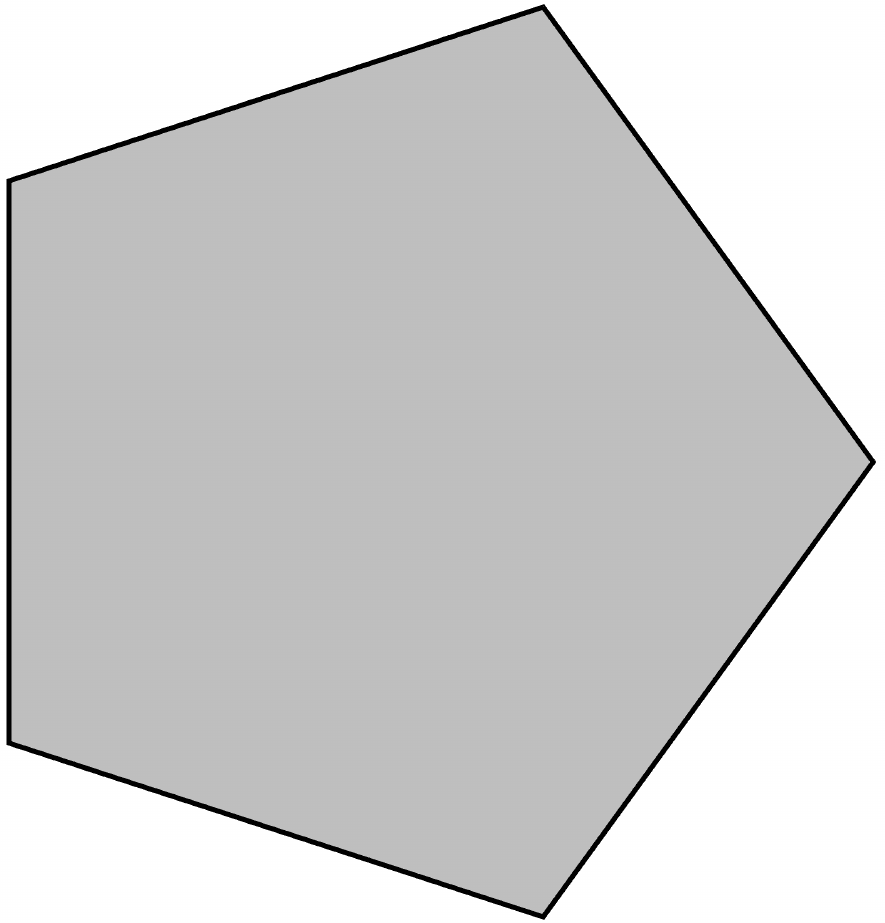}}
\caption{The periodgon for $k=4$, $\theta=0^-$, $\alpha=0$, and increasing values of $s\in [0,1]$. The ad hoc periodgon is degenerate with one side less for $s=\frac12$. For $s\in (0,\frac12)$, the periodgons for $\theta=0^-$ and $\theta=0^+$ on the two sides of the slit are symmetric one to the other with respect to the real axis.
}\label{periodgon_theta_0}\end{center}\end{figure}

\subsection{Non-convexity of the periodgon}

In the neighborhood of the parabolic situation $(s,\theta)=(\frac12,\frac{2\pi m}{k})$, the two very large periods $\nu_m$ and $\nu_{m+1}$, $m\in\Z_{k+1}$, can have arbitrary arguments. Hence, there is no hope that the ad hoc periodgon would be convex.

\begin{figure}\begin{center}
\subfigure[$k=2$]{\includegraphics[width=2.8cm]{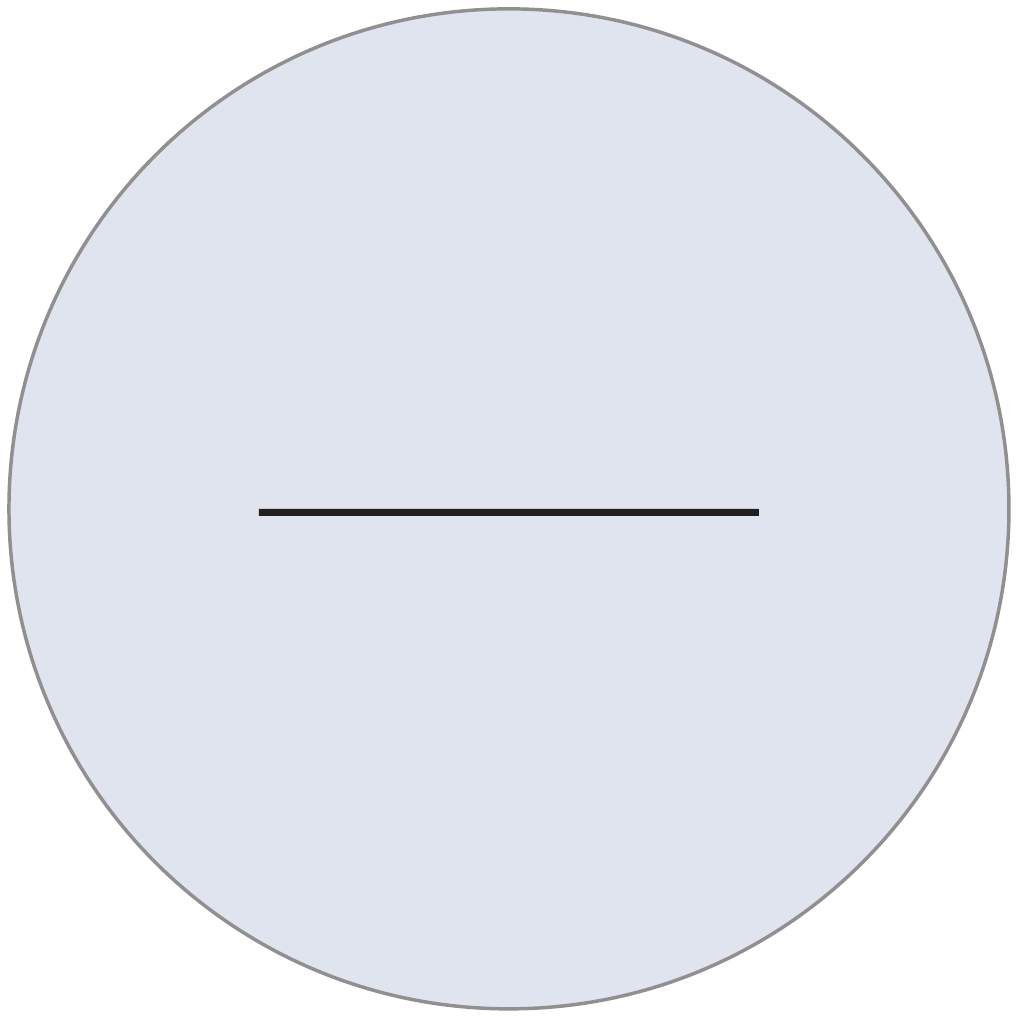}}\quad
\subfigure[$k=3$]{\includegraphics[width=2.8cm]{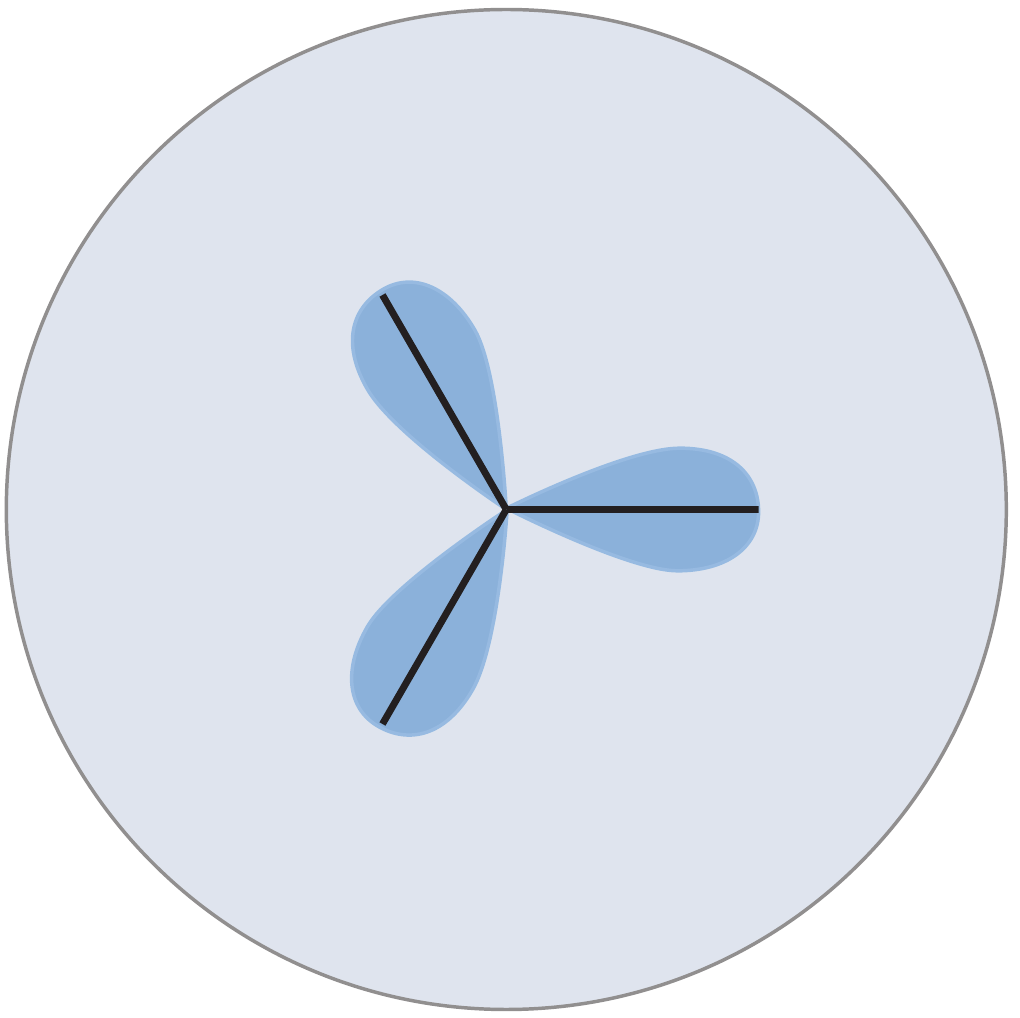}} \quad 
\subfigure[$k=4$]{\includegraphics[width=2.8cm]{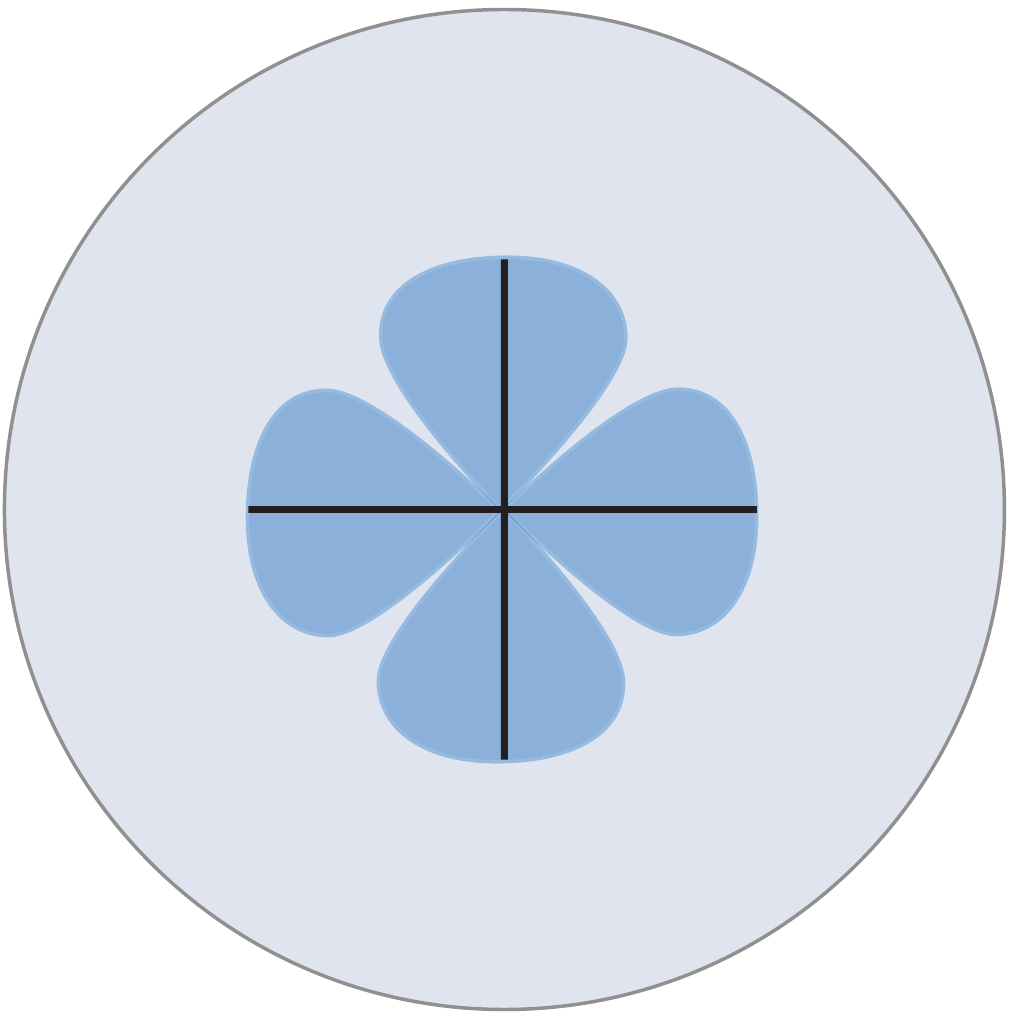}}\quad
\subfigure[$k\geq5$]{\includegraphics[width=2.8cm]{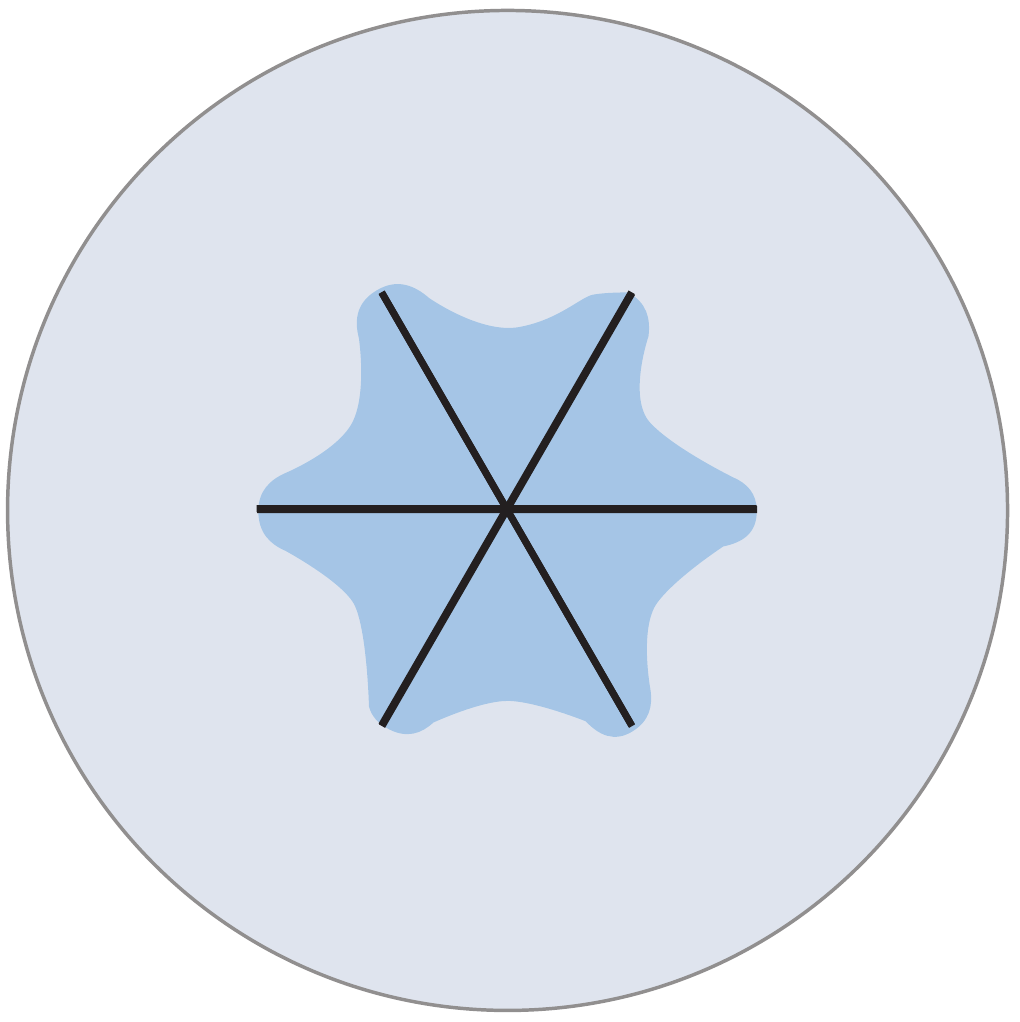}}
\caption{A subdomain of $\D$ where the ad hoc periodgon is non-convex. (It is conjectured that the ad hoc periodgon is convex elsewhere.) }\label{domain_theta_nc}\end{center}\end{figure}

\begin{proposition}\label{cor:nonconvex} 
For $k=2$, the ad hoc periodgon is always convex. For $k=3, 4$, the ad hoc periodgon is non-convex on some flower shape subdomain of $\D$, with petals around each slit
of sectoral opening $\frac{\pi}{3}$ for $k=3$, and $\frac{\pi}{2}$ for $k=4$, as in Figure~\ref{domain_theta_nc} (b),(c),
while it is convex on the rays $\theta=\frac{(2m+1)\pi}{k}$, $s\in (0,1]$.
For $k\geq5$ the ad hoc periodgon is non-convex on some open set containing $s<s^*$ for some $s^*>0$ as in Figure~\ref{domain_theta_nc} (d).
\end{proposition}
\begin{proof} 
We can suppose that $\alpha=0$, $\theta\in[-\frac{\pi}k,0]$.
The eigenvalues at $z_j$, $j\neq 0$, are given by
\begin{equation}\lambda_j=P_\epsilon'(z_j)=k(k+1)\big[(1-s)^k-\tfrac{s^{k+1}}{x_j}\big],\label{eigenvalues} \end{equation}
where $x_j=e^{-i\theta}z_j$. The vector field in the $x$-variable is
\begin{equation}\dot x=e^{-ik\theta}x-(k+1)(1-s)^kx+ks^{k+1}. \label{vf_x} \end{equation}

For $k=2$, the ad hoc periodgon is a triangle. 

For $k=3,4$, the ad hoc periodgon is convex on $s=1$ and depends continuously on decreasing $s$.
Let us show that it is also convex on the ray $\theta=-\frac{\pi}{k}$.
It suffices to prove that no two adjacent sides of the ad hoc periodgon can become aligned for some value of $s\in (0,1)$. 
For $k=3$, \eqref{vf_x} has two real singular points, $x_0\in\R^+$, $x_2\in R^-$ and two complex conjugate ones $x_1,x_3$. The conclusion follows from the fact that 
$\lambda_0,\lambda_2\in\R$ and $\lambda_1,\lambda_3\notin\R$. 
For $k=4$, \eqref{vf_x} has five singular points $x_0, \dots, x_4$ with increasing arguments, $x_0\in \R^+$, $x_1=\ov{x}_4$ and $x_2=\ov{x}_3$. 
From Proposition~\ref{prop:norm_sp} $|x_0|<|x_1|=|x_4|<|x_2|=|x_3|$. 
Since $\lambda_0\in\R$ while $\lambda_1=\ov\lambda_4,\lambda_2=\ov\lambda_3\notin\R$, and $\Re\lambda_2=\Re\lambda_3<0$,
it suffices to show that  $\lambda_1,\lambda_2$ cannot be aligned.
But $\left|\frac{s^{k+1}}{x_1}\right|>\left|\frac{s^{k+1}}{x_2}\right|$ and $\Re x_1>0$, $\Re x_2<0$, 
therefore the eigenvalues $\lambda_1,\lambda_2$ cannot indeed be aligned (see Figure~\ref{not_aligned}).

Let us show that for $k\geq 5$ and small $s$ the ad hoc periodgon is non-convex.
If $s$ is infinitesimally small, then for $j,l\neq 0$
\begin{align}\begin{split}
 \arg\lambda_j>\arg\lambda_l\ 
&\Leftrightarrow\ \text{either}\ \Im(\tfrac{-1}{x_j})>\Im(\tfrac{-1}{x_l}),\\
&\qquad\qquad \text{or}\ \Im(\tfrac{-1}{x_j})=\Im(\tfrac{-1}{x_l})\neq 0,\ \arg(\tfrac{-1}{x_j})>\arg(\tfrac{-1}{x_l}),\\
&\Leftrightarrow\ \text{either}\ \Im(x_j)>\Im(x_l),\\
&\qquad\qquad \text{or}\ \Im(x_j)=\Im(x_l)\neq 0,\ \arg(x_j)<\arg(x_l),\label{condition} 
\end{split}\end{align}
using that $|x_j|\sim |x_l|$, where the arguments are taken in $(-\pi,\pi)$.
In particular, if $k\geq 5$ then $\Im(x_2)>\Im(x_1)$ and $\arg\lambda_1<\arg\lambda_2$, hence $\arg\nu_1>\arg\nu_2$, and the ad hoc periodgon is non-convex.

The same is true if $k=4$ and $\theta\in(-\frac{\pi}4,0]$, and if $k=3$ and $\theta\in(-\frac{\pi}{6},0]$.
\end{proof}

\begin{figure} \begin{center}\includegraphics[width=3cm]{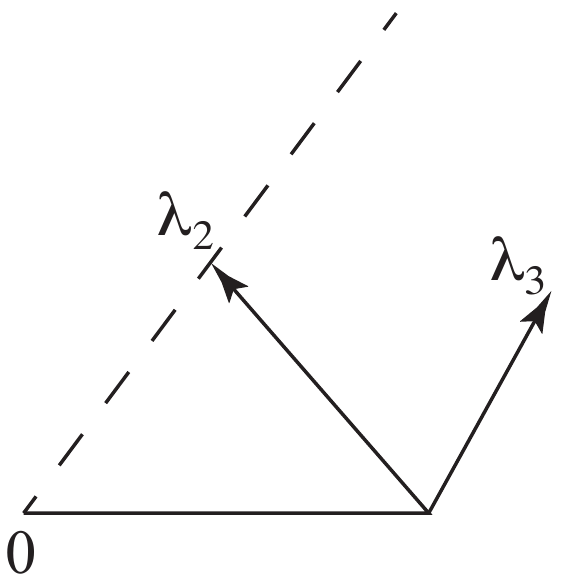}\caption{For $k=4$, the eigenvalues $\lambda_{1}$ and $\lambda_{2}$  are never aligned.}\label{not_aligned}\end{center}\end{figure}

\begin{conjecture}\label{conjecture1} The only region of non-convexity of the ad hoc periodgon is that described in Proposition~\ref{cor:nonconvex} and Figure~\ref{domain_theta_nc}.\end{conjecture}

\subsection{The parabolic situation}\label{sec: parabolic} 

When the discriminant vanishes, the system has the form
\begin{equation}\label{parabolic_alpha}\begin{aligned} 
\dot z &= z^{k+1} -(k+1)\big(\tfrac{e^{i\alpha}}{2}\big)^kz +k\big(\tfrac{e^{i\alpha}}{2}\big)^{k+1}\\
&=\big(z-\tfrac{e^{i\alpha}}{2}\big)^2\big(z^{k-1}+2z^{k-2}\tfrac{e^{i\alpha}}{2}+\ldots+k\big(\tfrac{e^{i\alpha}}{2}\big)^{k-1}\big).
\end{aligned}\end{equation}
The position of the singular points is given by the image through a rotation of the position of the singular points of the vector field
\begin{equation} \dot z = z^{k+1} -(k+1)\tfrac{1}{2^k}z +k\tfrac{1}{2^{k+1}}.\label{parabolic_model}\end{equation}

\begin{remark} 
There is numerical evidence that the periodgon is convex at the parabolic situation for $s=\frac12$ and $\theta=0$ (see Figure~\ref{pgon_parab}).
\end{remark} 

\begin{figure} \begin{center}
\subfigure[$k=4$]{\includegraphics[scale=.3]{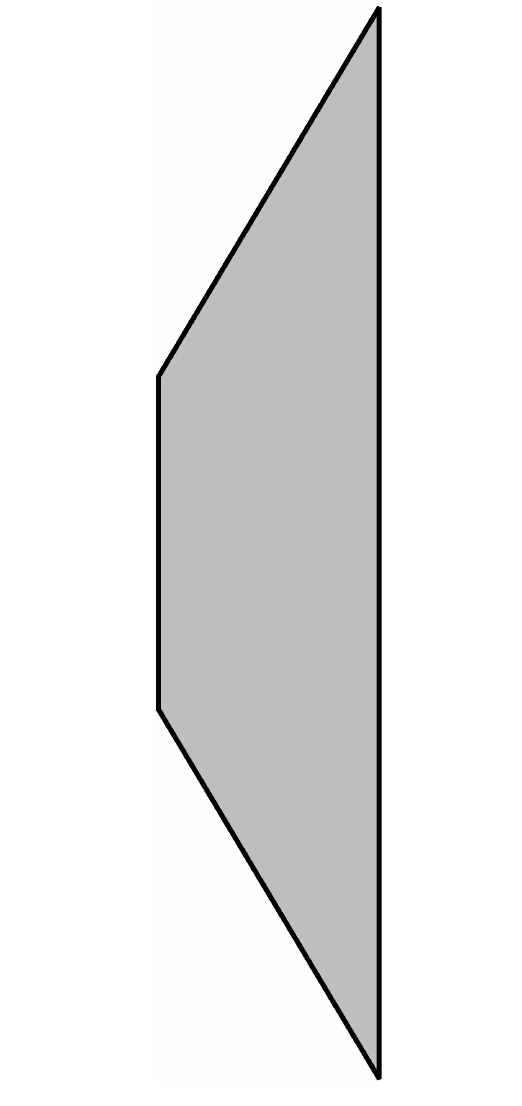}\qquad}
\subfigure[$k=8$]{\includegraphics[scale=.3]{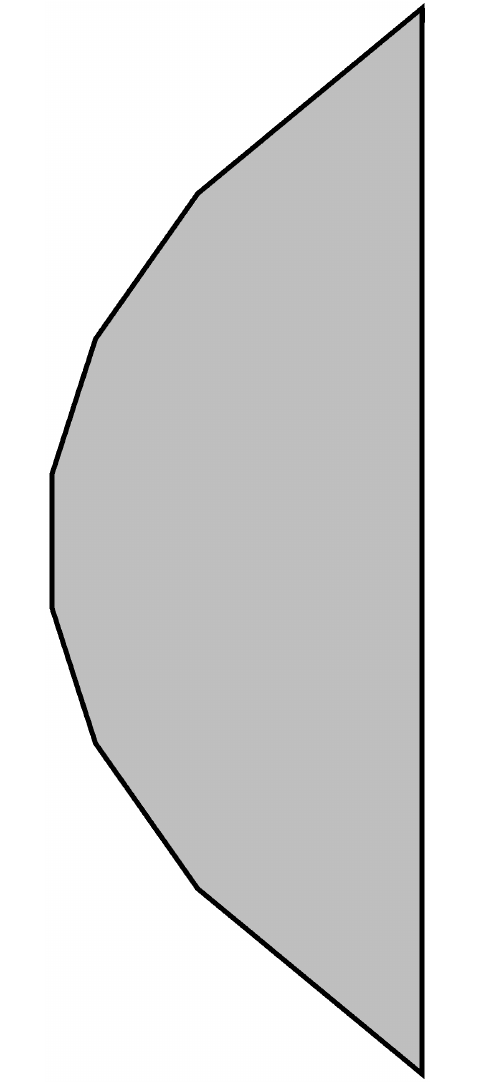}\qquad}
\subfigure[$k=11$]{\quad\includegraphics[scale=.3]{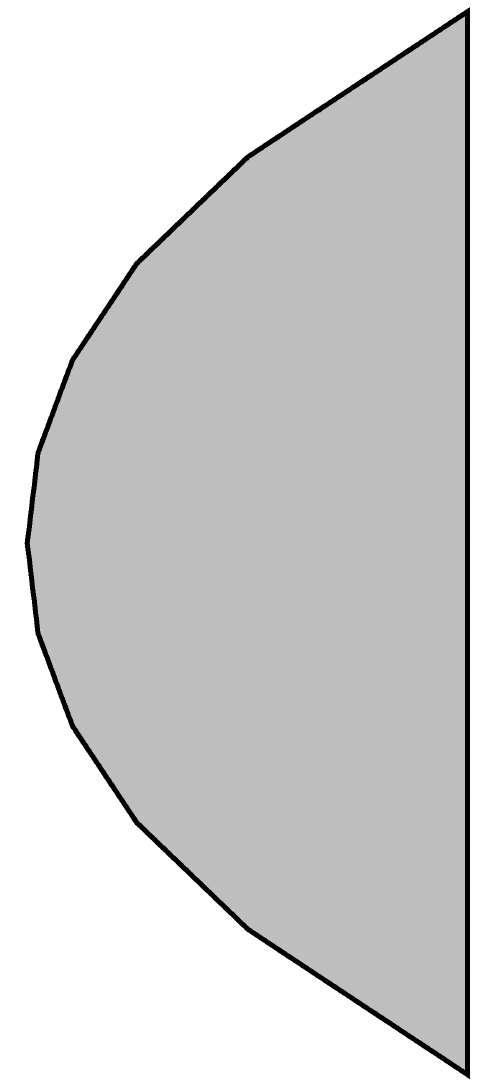}\quad}
\caption{ Numerical evidence suggests that the periodgon is convex at the parabolic point.}\label{pgon_parab}
\end{center}\end{figure}

\begin{figure}\begin{center}
\includegraphics[width=11cm]{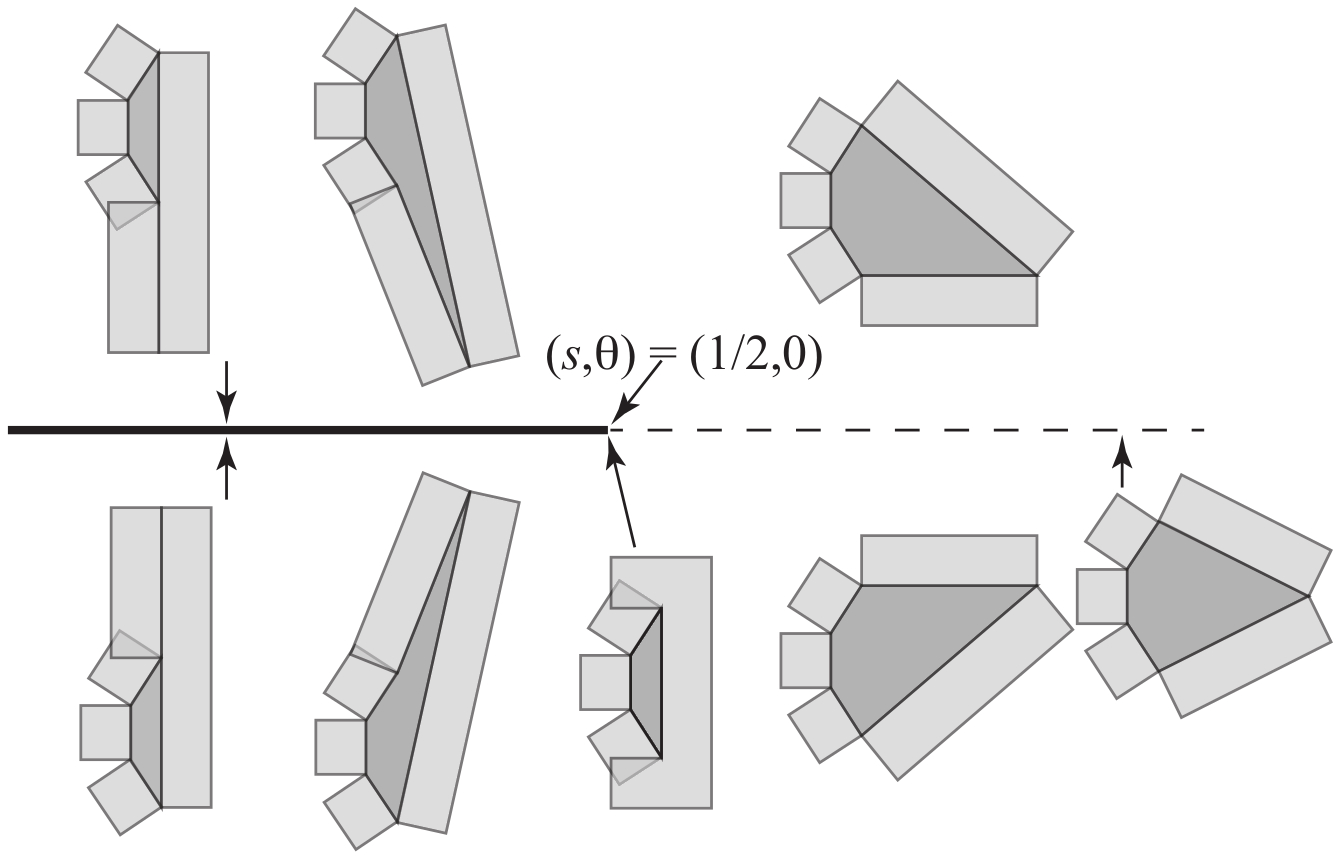}
\caption{The star domain in $t$-space for $\alpha=0$ near the parabolic situation $(s,\theta)=(\frac12,0)$.}\label{fig:star_parabolic_unfold}
\end{center}\end{figure}

\begin{proposition} 
For $\alpha\in \left(\frac{(2m-1)\pi}{2k},\frac{(2m+1)\pi}{2k}\right)$, $m\in \Z_{2k}$, 
the sepal zones of the parabolic point cover two neighboring sectors (ends) at infinity corresponding to $\arg z\in \left(\frac{(m-1)\pi}{k},\frac{m\pi}{k} \right)$ and $\arg z\in \left(\frac{m\pi}{k},\frac{(m+1)\pi}{k} \right)$ (see Figure~\ref{parabolic_k_4}).
This changes when $\alpha\in \frac{\pi}{2k}+\frac{\pi}{k}\Z$ where the system is reversible (see Figure~\ref{parabolic_k_4}(i)). 
This bifurcation is located on the adherence of (non-parabolic) homoclinic loop bifurcations. 
\end{proposition}

\begin{proof} 
It suffices to start at $\alpha=0$. There the system is symmetric and the real line is invariant with the parabolic point at $z=\frac12$ being the only singular point on $\R^+$. The half-line $\{z\in \R\; :\; z>\frac12\}$ is one separatrix from the parabolic point. The other singular points are repelling. Indeed they satisfy $|z_j|>\frac12$, from which it follows that $\Re(\lambda_j)>0$ for $j\geq2$ (see \eqref{eigenvalues}). 
Hence all repelling separatrices of $\infty$ have their $\omega$-limit at the parabolic point as in Figure~\ref{parabolic_k_4}(a). 
At $\alpha=\frac{\pi}{2k}$ we get two periodic zones as in Figure~\ref{parabolic_k_4}(i): indeed the system is reversible with the symmetry axis $\exp(-\frac{\pi i}{2k})\R$. Since the system is rotational, no other bifurcation can have occurred in between because of the monotonic movement of the separatrices. The  symmetries are used for the other values of $\alpha$.
\end{proof}

\begin{figure}\begin{center}
\subfigure[$\alpha=0$]{\reflectbox{\rotatebox[origin=c]{180}{\includegraphics[width=3.5cm]{fig/parab_4_1}}}}\qquad
\subfigure[]{\reflectbox{\rotatebox[origin=c]{180}{\includegraphics[width=3.5cm]{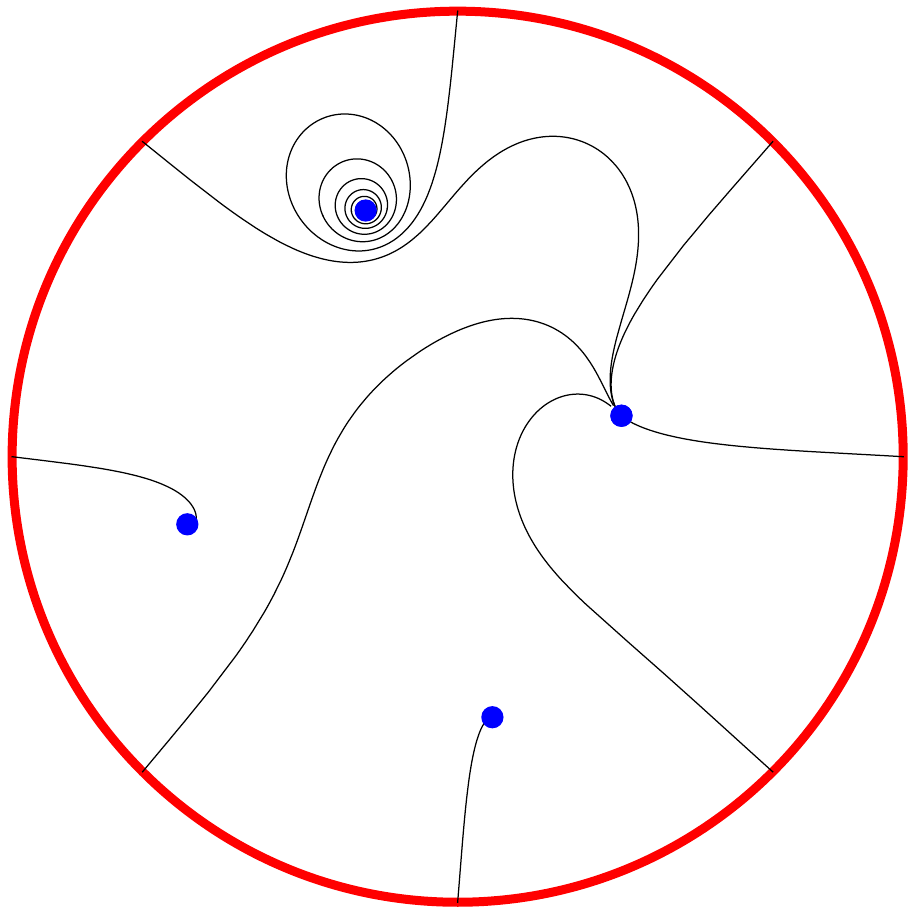}}}}\qquad 
\subfigure[]{\reflectbox{\rotatebox[origin=c]{180}{\includegraphics[width=3.5cm]{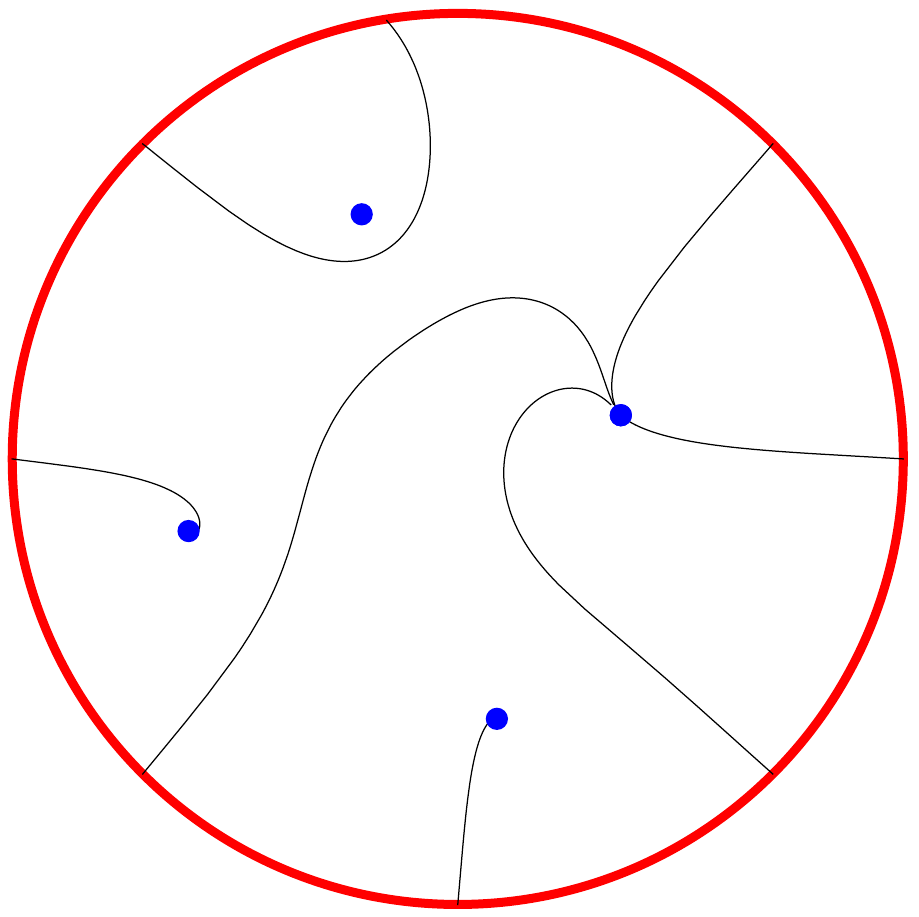}}}}\\
\subfigure[]{\reflectbox{\rotatebox[origin=c]{180}{\includegraphics[width=3.5cm]{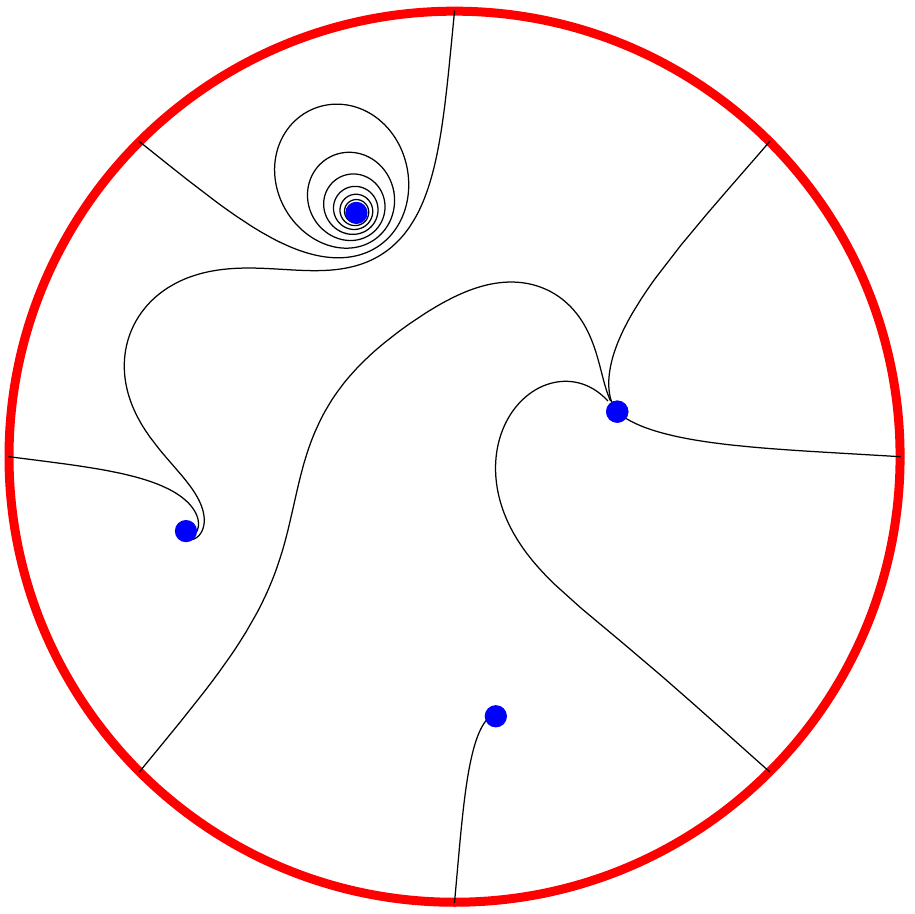}}}}\qquad
\subfigure[]{\reflectbox{\rotatebox[origin=c]{180}{\includegraphics[width=3.5cm]{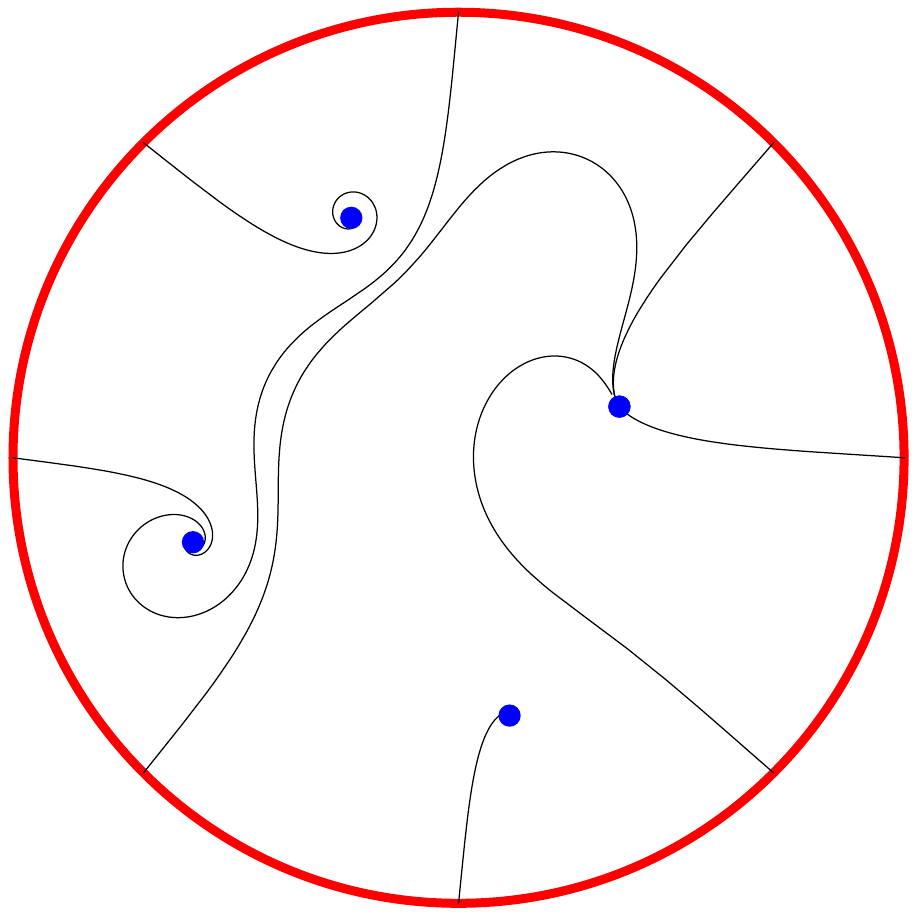}}}}\qquad 
\subfigure[]{\reflectbox{\rotatebox[origin=c]{180}{\includegraphics[width=3.5cm]{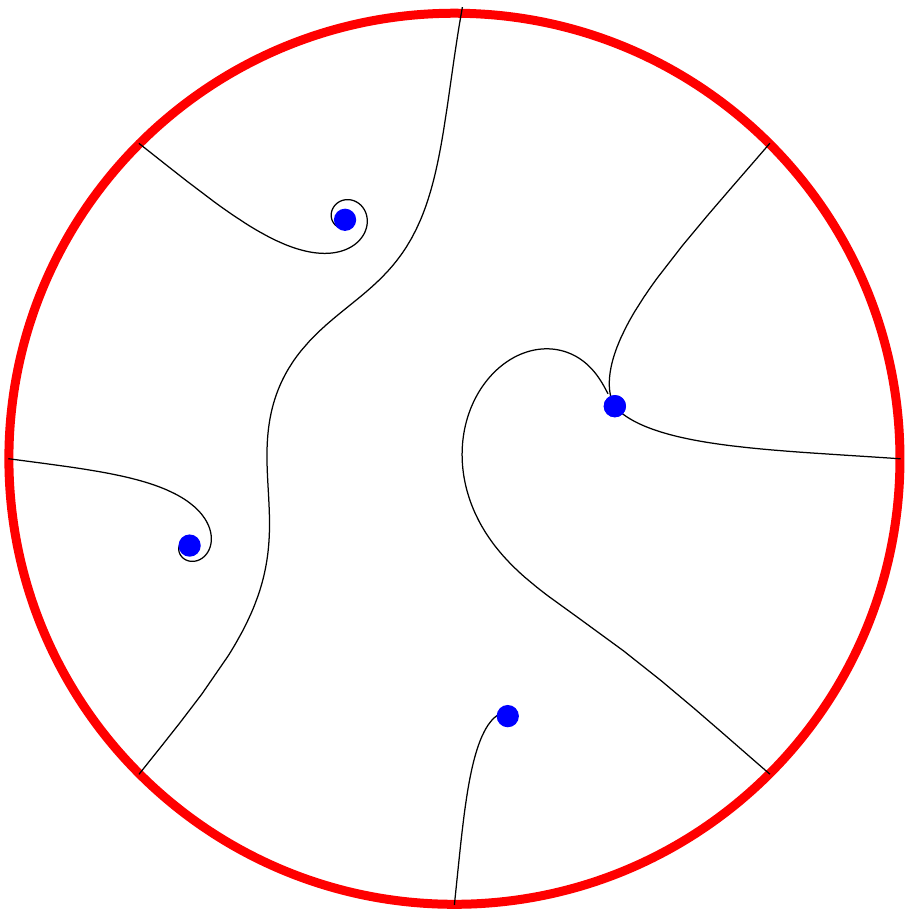}}}}\\
\subfigure[]{\reflectbox{\rotatebox[origin=c]{180}{\includegraphics[width=3.5cm]{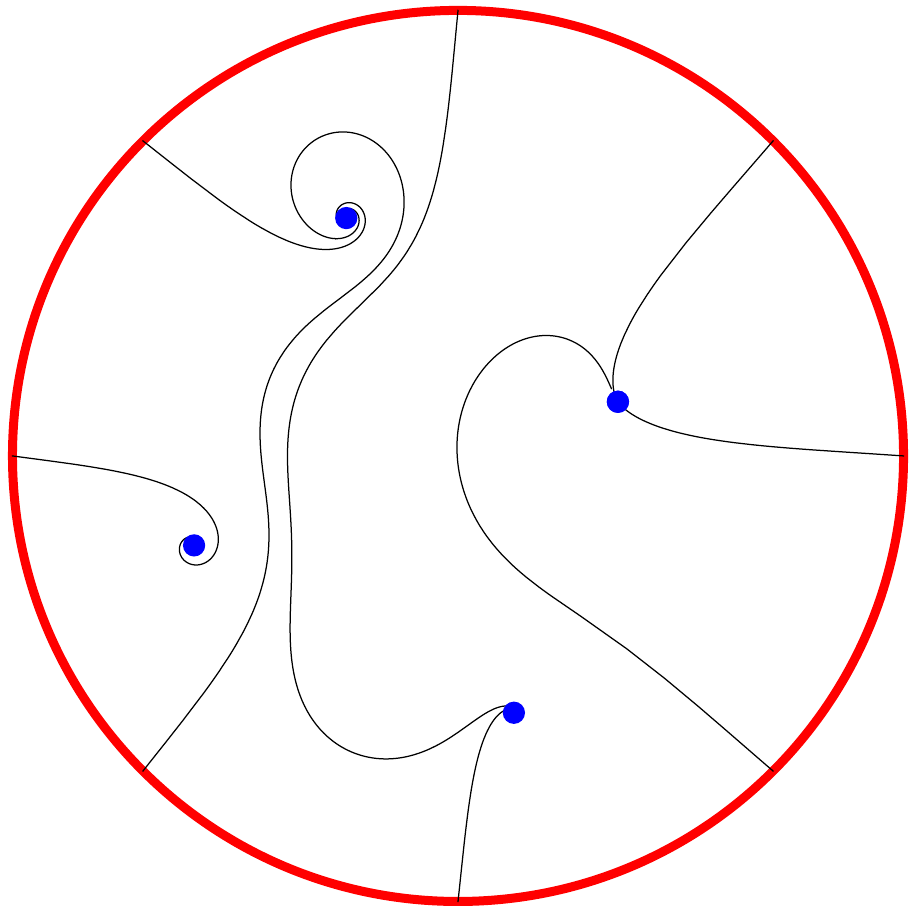}}}}\qquad
\subfigure[]{\reflectbox{\rotatebox[origin=c]{180}{\includegraphics[width=3.5cm]{fig/parab_4_6}}}}\qquad
\subfigure[$\alpha=\frac{\pi}{2k}$]{\reflectbox{\rotatebox[origin=c]{180}{\includegraphics[width=3.5cm]{fig/parab_4_8}}}}
\caption{The system \eqref{parabolic_alpha} for $k=4$ and increasing $\alpha$ in $[0, \pi/{2k}]$.}\label{parabolic_k_4}
\end{center}\end{figure}

\subsection{Towards the proof of Conjecture~\ref{conjecture_2}}

\begin{theorem}\label{thm:periodgon}
For $\theta\in[-\frac{\pi}k,0]$, the ad hoc periodgon has no self-intersection in the following cases 
\begin{enumerate} 
\item $k=2$ and $k=3$;
\item $s$ small nonzero and $\theta\neq0$;
\item $s$ is sufficiently large so that  $\Re_\theta z_j(s,\theta,0)\leq \frac{s^{k+1}}{(1-s)^k}$ for all $j$: in that case the ad hoc periodgon is convex; 
\item on an open neighborhood of $\theta=\frac{\pi (2m+1)}{k}$, $s\in(0,1)$; and of $\theta=\frac{2m\pi}{k}$,  $s\in[\frac12,1)$ in the slit disk.
\end{enumerate}
\end{theorem}

\begin{proof}  
\begin{enumerate} 
\item For $k=2$, the ad hoc periodgon is a triangle. The case $k=3$ is treated in Lemma \ref{lemma:k=3}.

\item This case follows from the values of the  $\frac{d\arg\nu_j}{ds}=-\frac{d\arg\lambda_j}{ds}$ near $s=0$ using the formula in Proposition~\ref{prop:dlambda} below.
Indeed, near $s=0$, then $x_0= \frac{k}{k+1}\frac{s^{k+1}}{(1-s)^k} + O(s^{(k+1)^2})$ and therefore
$\frac{d\arg\nu_0}{ds}=O(s^{k(k+1)-1})$, 
while $x_j=(k+1)^{\frac1k}e^{2\pi i\frac{j-1}{k}-i\theta}+O(s)$, which yields
$\frac{d\arg\nu_j}{ds}=-s^k(k+1)^{-\frac2k}\Im x_j+O(s^{k+1})=-s^k(k+1)^{-\frac1k}\sin(2\pi\frac{j-1}{k}-\theta)+O(s^{k+1})$. 
Hence the  $\nu_j$  move in the first (resp. second quadrant) when $\Im x_j>0$ (resp. $\Im x_j<0$). 
The result follows since the growth of variation of $|\arg(\nu_j)|$ for $j>0$ is much larger than of $|\arg(\nu_0)|$.

\item We know that for $s=1$ the periodgon and the ad hoc periodgon agree and are convex. We show that this stays true also, when $(s,\theta)$ is in the same connected component as the set $s=1$ of the parameter region where $\Re_\theta z_j(s,\theta,0)<\frac{s^{k+1}}{(1-s)^k}$ for all $j$. 
Indeed, in order to break convexity two consecutive eigenvalues $\lambda_j$ and $\lambda_{j+1}$ would have to become aligned,
either in opposite directions, which is impossible unless $j=0$, or in the same direction.
In the second case we can suppose that $\Im\lambda_j,\Im\lambda_{j+1}>0$ for example, i.e. that $\Im x_j,\Im x_{j+1}>0$, and we know that  $|x_j|\leq|x_{j+1}|$ by Proposition~\ref{prop:norm_sp}, i.e.  $|\lambda_j-k(k+1)(1-s)^k|\geq|\lambda_{j+1}-k(k+1)(1-s)^k|$ using \eqref{eigenvalues}.
If $\lambda_{j+1}$ were aligned with $\lambda_j$ (see Figure~\ref{argument}), this would mean that the ray $\lambda_j\R^+$ would need to have two intersections with the circle
$\{\lambda\in\C:\ |\lambda-k(k+1)(1-s)^k|=|\lambda_j-k(k+1)(1-s)^k|\}$, and $\lambda_{j+1}$ would lie on the chord between the two intersections.
Therefore $\lambda_j$ would have to lie inside the disk with diameter given by the segment $[0,k(k+1)(1-s)^k]$ (the blue disk in Figure~\ref{argument}), namely $\{\lambda\in\C:\Re(\frac1\lambda)>\frac{1}{k(k+1)(1-s)^k}\}$,
which is equivalent to $\Re (x_j)>\frac{s^{k+1}}{(1-s)^k}$, a contradiction.
  
\item The ad hoc periodgon has no self-intersection on $\theta=\frac{\pi (2m+1)}{k}$, $s\in(0,1)$ (resp. $\theta=\frac{2m\pi}{k}$, $s\in(\frac12,1)$), as was shown in the proof
of Proposition~\ref{prop:period_gon} using the fact that the system is symmetric for $\alpha=0$. By continuity it is true also in an open neighborhood.
 Near the parabolic bifurcation $\theta=\frac{2m\pi}{k}$, $s=\frac12$, we know that the only self-intersection can come from the alignment of $\nu_0$ and $\nu_1$, which happens exactly on the slits. To see this, write $e^{ik\theta}\big(\frac{s}{1-s}\big)^{k(k+1)}=1+\delta$, $\delta\in(\C,0)$, near the point of parabolic bifurcation. Then $(s,\theta)$ depend analytically on $\delta$, and denoting $X:=\frac{(1-s)^k z}{s^{k+1}e^{i\theta}}$, we have $(1+\delta)X^{k+1}-(k+1)X+k=0$, which has two bifurcating roots $X_j=1\pm\sqrt{\frac{-2\delta}{k(k+1)}}+O(\delta)$, $j=0,1$, depending analytically on $\sqrt\delta$, that are exchanged by $\delta\mapsto e^{2\pi i}\delta$. The associated eigenvalues $\lambda_j=k(k+1)(1-s)^k\big[1-\frac1{X_j}\big]$, $j=0,1$, are aligned if and only if they are both aligned with the vector $\Lambda:=\frac{\lambda_0\lambda_1}{\lambda_0+\lambda_1}\neq 0$, which depends analytically on $\delta$, since for $\delta\to 0$, $\frac{2\pi}{\Lambda}\to\nu_{par}\neq 0$ \eqref{nu_par}. Therefore there is a unique half-curve emanating from $\delta=0$ on which the alignment happens, and we already know from the real symmetry that this curve is the ray $\delta\in\R^-$.
\end{enumerate}
\end{proof}

\begin{figure}\begin{center} \includegraphics[width=5cm]{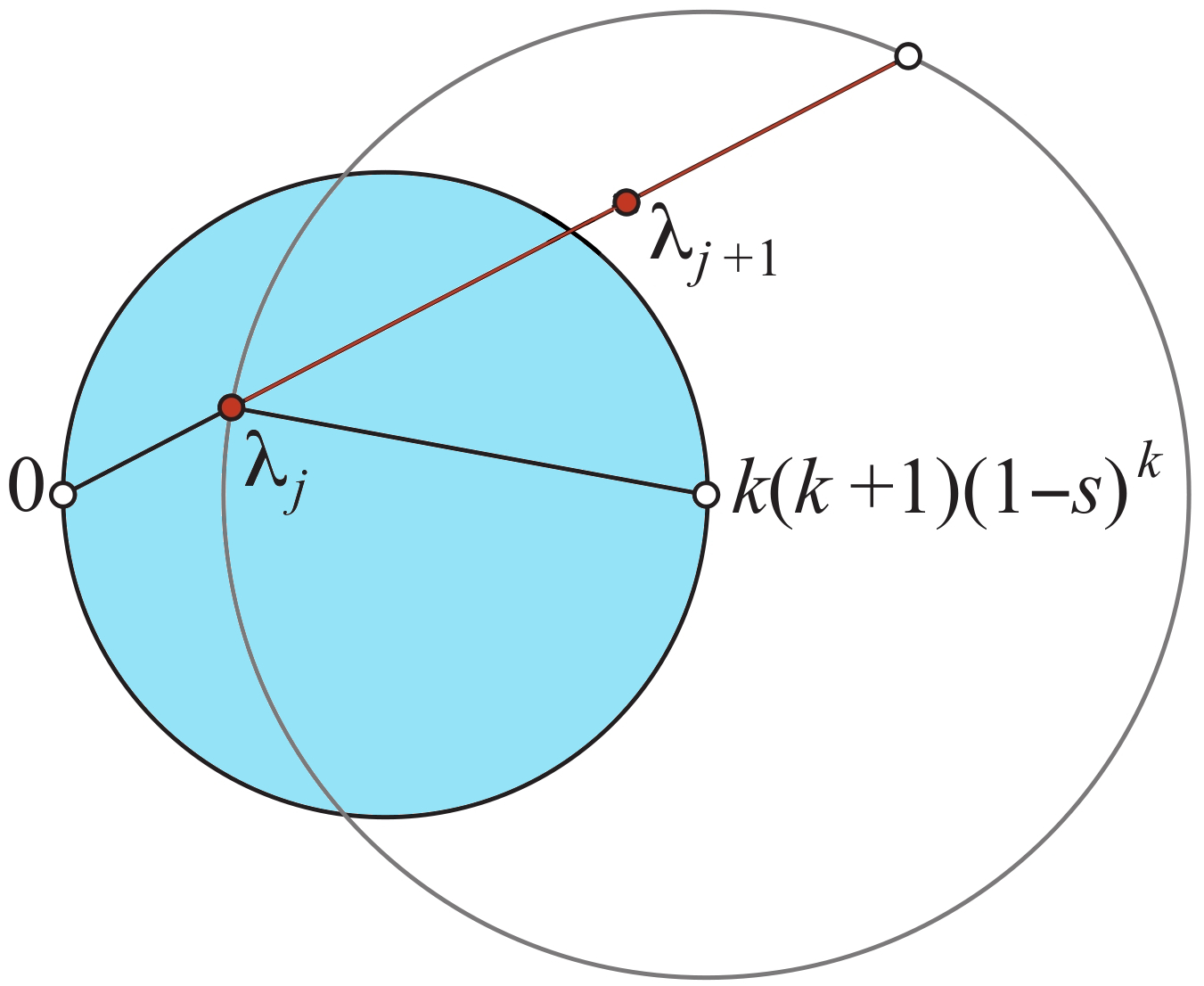}\caption{The proof of Theorem~\ref{thm:periodgon}:  if the eigenvalues $\lambda_j$ and $\lambda_{j+1}$ are aligned, then one of them must lie in the blue disk.}\label{argument}\end{center}\end{figure}

\begin{lemma}\label{lemma:k=3}
For $k=3$ the ad hoc periodgon has no self-intersection outside of the cuts $s\in[0, \frac12]$, $\theta=\frac{2\pi m}{k}$. 
\end{lemma}
\begin{proof}
The ad hoc periodgon has no self-intersection near $s=1$. The only way the ad hoc periodgon could bifurcate to a self-intersection would be by passing through a situation in which two successive sides are aligned and oriented in opposite directions.

Suppose that for $\alpha=0$ and some $s,\theta$, one has $\lambda_j,\lambda_l\in e^{i\beta}\R$ for a pair of indices $j,l$ and some $\beta\in\R$.
If $e^{i\beta}\in\R$ then $x_j,x_l\in\R$ and therefore $e^{ik\theta}\in\R$.  When $e^{ik\theta}=-1$, or when $e^{ik\theta}=1$ and $s>\frac12$, we know that in the case $k=3$ the periodgon is convex, therefore there is no self-crossing.

Hence we now consider the case $e^{i\beta}\notin\R$. 
Denoting $\rho=1-\frac{s^{k+1}}{(1-s)^kx}$, i.e. $\lambda=k(k+1)(1-s)^k\rho$, which yields $x =\frac{s^{k+1}}{(1-s)^k(1-\rho)}$, and replacing in $P_\eps(e^{i\theta}x)=0$ we have
\begin{equation*}
 (k+1)(1-\rho)^k-k(1-\rho)^{k+1}-e^{ik\theta}\big(\tfrac{s}{1-s}\big)^{k(k+1)}=0,
\end{equation*}
i.e.
\begin{equation*}
 \sum_{n=1}^k (-1)^n n\tfrac{(k+1)\ldots(k+1-n)}{(n+1)!}\rho^{n+1}= e^{ik\theta}\big(\tfrac{s}{1-s}\big)^{k(k+1)}-1,
\end{equation*}
which for $k=3$ is
\begin{equation}\label{eq:rho}
-6\rho^2+8\rho^3-3\rho^4= e^{i3\theta}\big(\tfrac{s}{1-s}\big)^{12}-1.
\end{equation}
Subtracting the equation \eqref{eq:rho} for $\rho_l$ from that for $\rho_j$ we have
\begin{equation*}
 -6t_1+8t_2e^{i\beta}-3t_3e^{2i\beta}=0,
\end{equation*}
where
$t_n:=\frac{\rho_j^{n+1}-\rho_l^{n+1}}{\rho_j-\rho_l}e^{-in\beta}\in\R$ since $\rho_j,\rho_l\in e^{i\beta}\R$.
Since $e^{i\beta}\notin\R$, the set of solutions $(t_1,t_2,t_3)$ of the above equation is a 1-dimensional real vector space given by
\begin{equation}t_2=\tfrac{3}2\cos\beta\cdot t_1,\qquad t_3=2t_1.\label{formula:t}\end{equation}
On the other hand, $(t_1,t_2,t_3)$ are symmetric polynomials of $\rho_je^{-i\beta}$, $\rho_le^{-i\beta}$ and satisfy an algebraic relation
$t_3=2t_1t_2-t_1^3,$
which can be rewritten as 
\begin{equation}\cos\beta=\tfrac{1}{3}(t_1+\tfrac{2}{t_1}),\qquad 1\leq |t_1|\leq 2.\label{formula:cos}\end{equation}
Using that $\rho_j+\rho_\ell=t_1e^{i\beta}$ and $\rho_j\rho_\ell=(t_1^2-t_2)e^{i2\beta}$ we get $$\rho_{j,l}=\tfrac12\left(t_1\pm\sqrt{4t_2-3t_1^2}\right)\cdot e^{i\beta};$$ 
therefore using \eqref{formula:t} we can express $\rho_j,\rho_l$ as functions of $t_1$ 
$$\rho_{j,l}=\tfrac{1}{2}\left(t_1\pm\sqrt{4-t_1^2}\right)\cdot\tfrac{1}{3}\left(t_1+\tfrac{2}{t_1}+i\sqrt{5-t_1^2-\tfrac4{t_1^2}}\right),$$
or the complex conjugate of it. 
We see that $\rho_j,\rho_l$ have opposite directions for $1< t_1<\sqrt2$, and have the same direction for $\sqrt2<t_1<2$.
When $t_1=\sqrt 2$, then  $\rho_j=\tfrac{4}{3}\pm i\tfrac{\sqrt 2}{3}$, $\rho_l=0$, which by \eqref{eq:rho} means that $(s,\theta)=(\tfrac12,0)$, and $x_l=\tfrac12$, $x_j=-\tfrac12\pm i\tfrac{\sqrt 2}2$. Therefore for $-\tfrac{\pi}{3}<\theta<0$,  $(j,l)=(2,0)$ for $1< t_1<\sqrt2$ (since the directions are opposite), and $(j,l)=(2,1)$ for $\sqrt2<t_1<2$. 
Neither of the two situations creates a self-intersection of the periodgon: the first one because the two sides of the periodgon are not adjacent, and hence parallel, the second one because the two sides becoming aligned have the same direction. 
\end{proof}

As a further step towards a proof of Conjecture~\ref{conjecture_2}, we can give numerical evidence that the ad hoc periodgon has no self-intersection in the neighborhood of each cut $\theta=\frac{2\pi m}{k}$, using the following proposition. 
 
\begin{proposition}\label{prop:numerical evidence} \begin{enumerate} \item On the cut $\theta=0$ and  $s\in (0,\frac12)$, 
$$\frac{d\arg(\lambda_j)}{d\theta}=\frac{s^{k+1} \left(ks^{k+1}-(k+1)(1-s)^k z_j\right)}{(k+1)(s^{k+1}-(1-s)^k z_j)^2}, \qquad j=0,1.$$
\item If \begin{equation}\frac{d\arg(\lambda_0)}{d\theta}-\frac{d\arg(\lambda_1)}{d\theta}>0 \label{num_evidence}\end{equation} 
on the cut $\theta=\frac{2\pi m}{k}$ and  $s\in (0,\frac12)$, then the ad hoc periodgon has no self-intersection in the neighborhood of the cut. 
\item The condition \eqref{num_evidence} is satisfied near $s=0$ and near $s=\frac12$. \end{enumerate}
\end{proposition} 
\begin{proof} \begin{enumerate} 
\item 
Using the formulas \eqref{lambda} and \eqref{dxdtheta} we calculate
\begin{equation}
\frac{d\log\lambda_j}{d\theta} = -is^{k+1}\frac{(k+1)(1-s)^kx_j-ks^{k+1}}{(k+1)\big((1-s)^kx_j-s^{k+1}\big)^2},
\end{equation}
and use that $z_0$ and $z_1$ are real positive on the cut.
\item  When $\theta=0$ and $s\in (0,\frac12)$, the only self-intersection of the ad hoc periodgon comes from the fact that $\nu_0$ and $\nu_1$ are aligned (see Figure~\ref{periodgon_theta_0}(b)). When moving to $\theta<0$, we have that $\arg(\nu_0)<-\frac{\pi}2$ and $\arg(\nu_1)<\frac{\pi}2$. The ad hoc periodgon will have no self-intersection below the cut if $\frac\pi2> \arg(-\nu_0)>\arg(\nu_1)> 0$  for small $\theta<0$, i.e. if 
$\frac{d}{d\theta}\big(\arg( \nu_0)-\arg(\nu_1)\big)<0$  along the cut, which is equivalent to \eqref{num_evidence}.
\item The condition \eqref{num_evidence} is satisfied near $s=0$ by Theorem~\ref{thm:periodgon} case (2).  On the cut  near $s=\frac12$, we let $s=\frac12-u^2$,  $u>0$. 
Then $z_0=\frac12-\sqrt{2}\,u + O(u^2)$ and $z_0=\frac12+\sqrt{2}\,u + O(u^2)$, yielding that 
$\frac{d\arg(\lambda_0)}{d\theta}-\frac{d\arg(\lambda_1)}{d\theta}=\frac{1}{\sqrt2 u}+O(1)>0$.
\end{enumerate}
 \end{proof}
 
\begin{remark}\label{rem:numerical_evidence} The numerical evidence for \eqref{num_evidence} comes from plotting the curve  $\left.\left(\frac{d\arg(\lambda_0}{d\theta}-\frac{d\arg(\lambda_1}{d\theta}\right)\right|_{\theta=0}$ as a function of $s\in (0,\frac12)$ (see Figure~\ref{fig:num_evidence}). 
 \end{remark}

\begin{figure}\begin{center} 
\includegraphics[width=5.5cm]{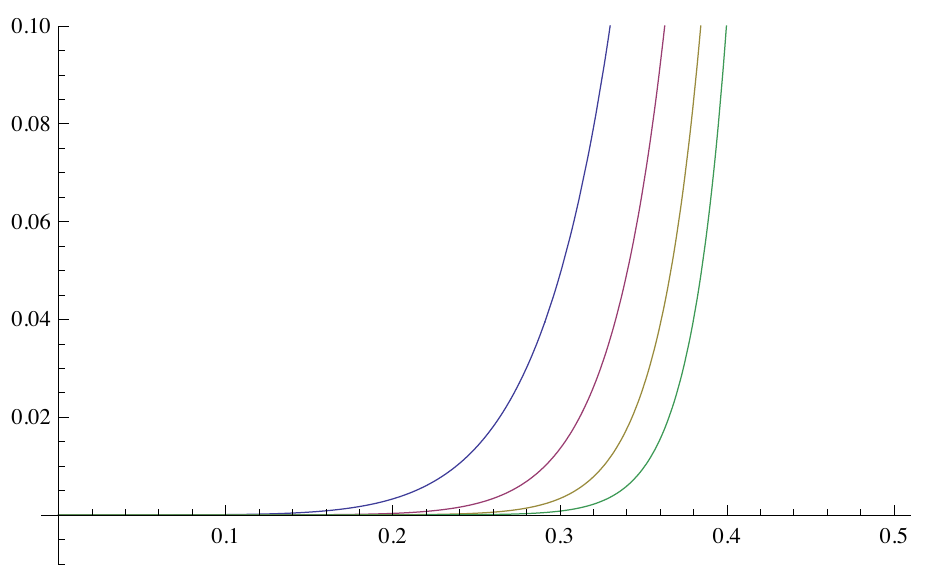}\caption{Graph of $s^{k+1}\left.\left(\frac{d\arg(\lambda_0)}{d\theta}-\frac{d\arg(\lambda_1)}{d\theta}\right)\right|_{\theta=0}$ as a function of $s\in (0,\frac12)$ for $k=4,6,8,10$ (the larger $k$, the flatter the curve). (A change of coordinate $z=s^{k+1}Z$ has been introduced together with a multiplication by the factor $s^{k+1}$ in order to control the numerical problems near $s=0$ because of the small denominator.) }\label{fig:num_evidence} \end{center}\end{figure}

\subsubsection{Shape of the periodgon}\label{subsec:shape_of_peridgon}

The sides of the ad hoc periodgon are the vectors $\nu_j$. The end of $\nu_{j+1}$ is attached to the origin of $\nu_j$ (indices are mod $k+1$). The orientation of its sides makes the ad hoc periodgon negatively oriented.

The eigenvalues $z_j$, $j>0$, given in \eqref{eigenvalues} have values as in Figure~\ref{fig:eigenvalues} with $x_j=z_je^{-i\theta}$.  Each eigenvalue is the sum  of the two terms $k(k+1)(1-s)^k$ and $$v_j=-k(k+1)\tfrac{s^{k+1}}{x_j} $$ (see Figure~\ref{fig:eigenvalues}.) Note that this expression is valid for $j=0$ also up to $s=0$ since $x_0(s)=e^{-i\theta}z_0(s) = \frac{k}{k+1} s^{k+1} +o(s^{k+1})$. Taking the arguments of $x_j$ in $[-\pi,\pi]$, then the terms $v_j$ have decreasing arguments in $[-\pi,\pi]$ when $j$ increases from $0$ to $k$. Also if $\arg x_j$ and $\arg x_{j'}$ have the same sign, then $|\arg x_j|<|\arg x_{j'}| \Rightarrow |v_j|< |v_{j'}|$.

\begin{figure}\begin{center} 
\subfigure[$s$ small]{\reflectbox{\rotatebox[origin=c]{180}{\includegraphics[width=5cm]{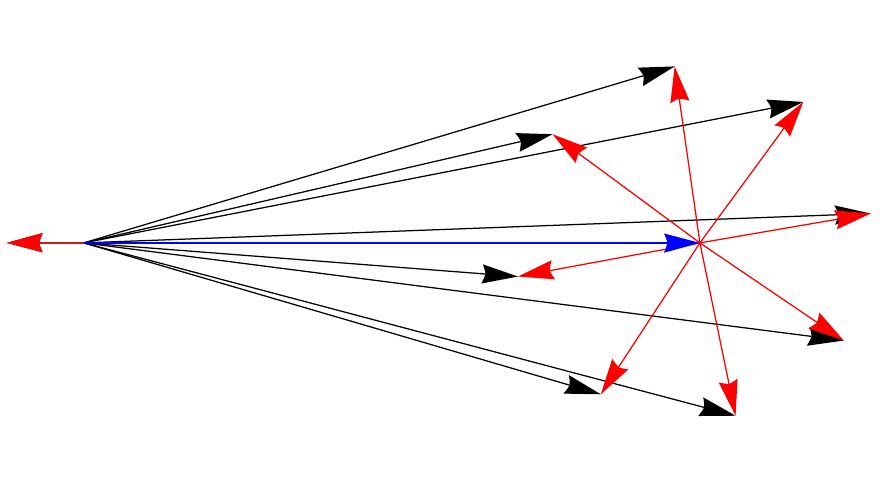}}}}\qquad
\subfigure[]{\reflectbox{\rotatebox[origin=c]{180}{\includegraphics[width=5cm]{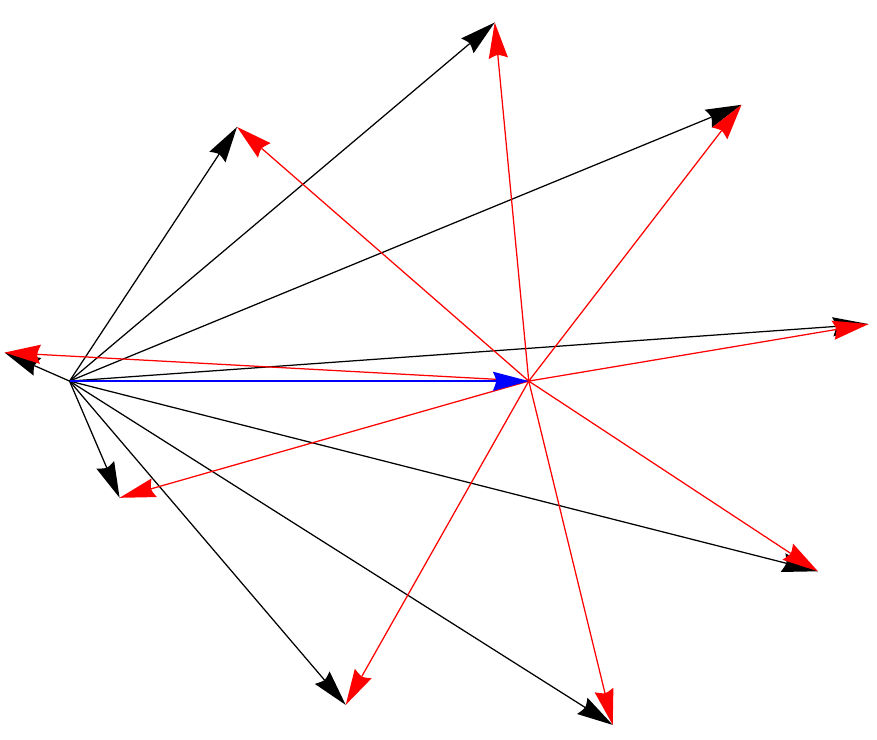}}}}\\
\subfigure[]{\reflectbox{\rotatebox[origin=c]{180}{\includegraphics[width=5cm]{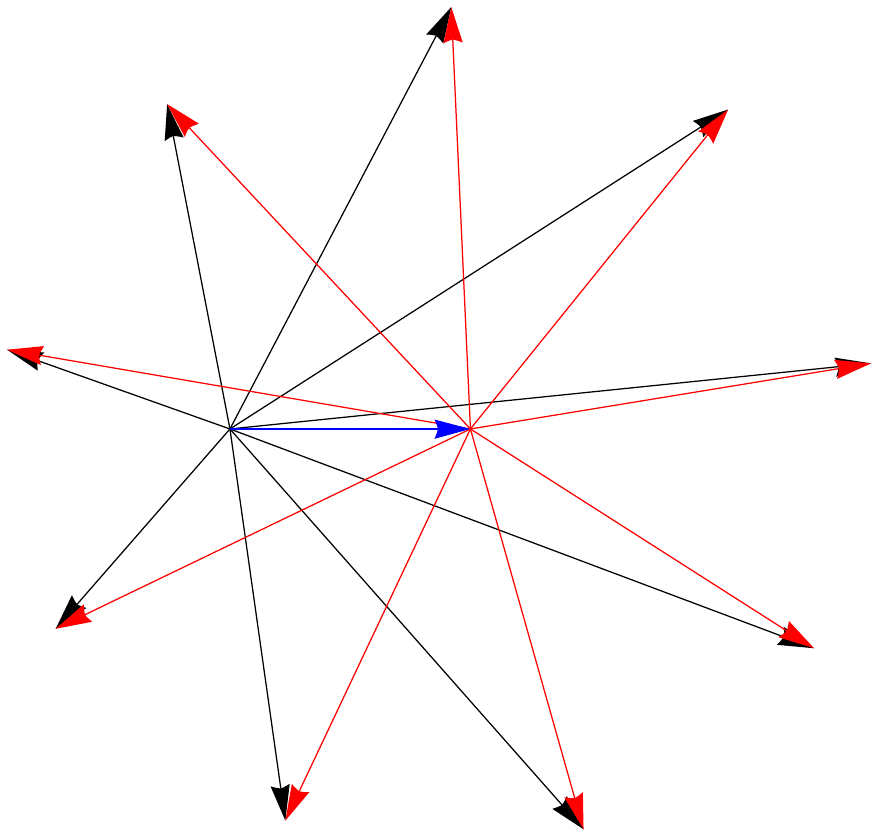}}}}\qquad
\subfigure[$s$ close to $1$]{\reflectbox{\rotatebox[origin=c]{180}{\includegraphics[width=5cm]{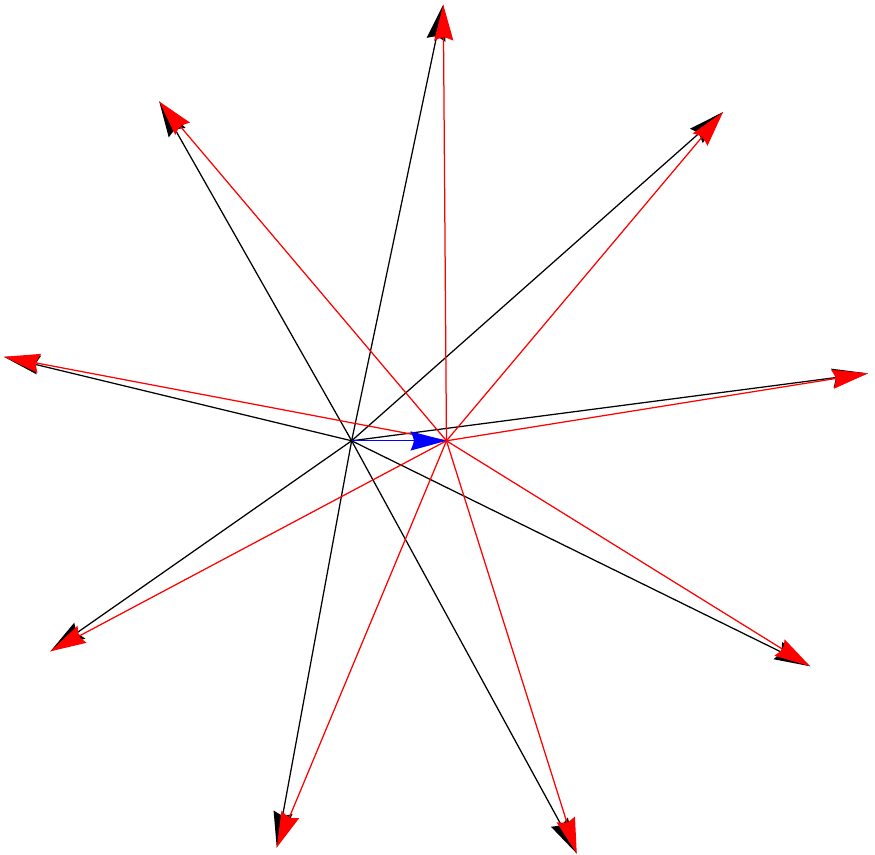}}}}
\caption{For $k=8$,  the eigenvalues $\lambda_j$ (in black) given in \eqref{eigenvalues},  for increasing values of $s$: they are given as the sums of the horizontal vector $k(k+1)(1-s)^k$ (in blue) and the vectors $v_j=-k(k+1)\tfrac{s^{k+1}}{x_j}$ (in red).}
\label{fig:eigenvalues}
\end{center}\end{figure}

\begin{figure}\begin{center} 
\subfigure[Region I]{\includegraphics[height=3cm]{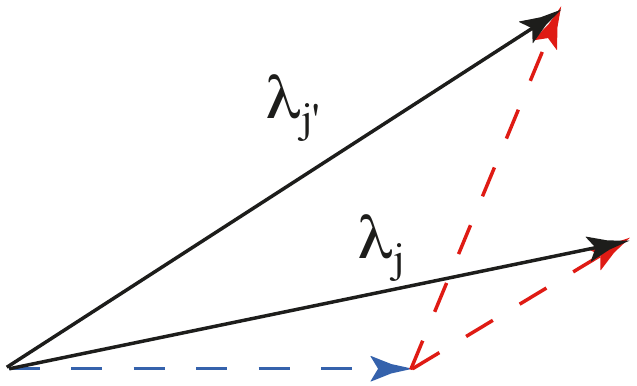}}\qquad\qquad 
\subfigure[Region II]{\includegraphics[height=3cm]{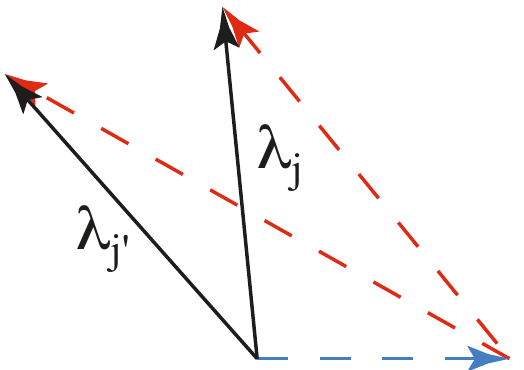}} 
\caption{The arguments of the eigenvalues $\lambda_j$ are in the same order as the arguments of the $v_j$ in regions I and II. }
\label{not_aligned23}
\end{center}\end{figure}

Hence, the eigenvalues can be divided into four subsets (see Figure~\ref{well_defined}): 
\begin{itemize}
\item The subset I of eigenvalues for which $\arg v_j\in [-\frac{\pi}2, \frac{\pi}2]$,  i.e. $\Re\lambda_j\geq k(k+1)(1-s)^k$: then the arguments of the eigenvalues are ordered as the arguments of the $v_j$. Since $k\geq3$, I contains at least two eigenvalues, one with positive argument and one with negative argument. The corresponding part of the ad hoc periodgon is convex. Indeed the sign of the argument of $\lambda_j$ is the same as that of the argument of $v_j$. Moreover, if $\arg v_j> \arg v_j'>0$, since $|v_j| >|v_{j'}|$, then $\arg \lambda_j >\arg \lambda_{j'}$ (see Figure~\ref{not_aligned23}). The same is true for the negative arguments. 
The corresponding part of the ad hoc periodgon starts at the bottom with an orientation in the second quadrant and ends on the top with an orientation in the first quadrant.
\item The subset II of eigenvalues for which $|\arg \lambda_j|\geq\frac{\pi}2$, i.e. $\Re\lambda_j\leq 0$: then the arguments of the eigenvalues are ordered as the arguments of the $v_j$. The corresponding part of the ad hoc periodgon is convex. There are two cases: either it contains only $\lambda_0$, in which case this unique side  is oriented  in the third quadrant, or it contains at least two sides in which case it starts at the top with an orientation in the fourth quadrant and ends at the bottom with an orientation in the third quadrant. 
\item The intermediate region III with $0<\arg \lambda_j<\frac{\pi}2$ and  $\frac{\pi}2<\arg v_j<\pi$, i.e. $0<\Re\lambda_j< k(k+1)(1-s)^k$, $\Im\lambda_j>0$. The corresponding part of the ad hoc periodgon has sides oriented in the second quadrant.
\item The intermediate region IV with $- \frac{\pi}2<\arg \lambda_j<0$ and  $-\pi\leq\arg v_j<-\frac{\pi}2$,  i.e. $0<\Re\lambda_j< k(k+1)(1-s)^k$, $\Im\lambda_j\leq 0$. The corresponding part of the ad hoc periodgon has sides oriented in the first quadrant.
\end{itemize}

If $\Re(x_l)>0$, i.e. if $\nu_l$ is of type II, III, or IV, then from Proposition~\ref{prop:norm_sp}  we can deduce that
$$\begin{tabular}{ll}
if $l<\tfrac{k+1}{2}$,  then & $\Re\lambda_1<\Re\lambda_2<\ldots<\Re\lambda_l$,\\	
if $l>\tfrac{k+1}{2}$,  then & $\Re\lambda_k<\Re\lambda_{k-1}<\ldots<\Re\lambda_l$.
\end{tabular}$$
This means that the union of the sides of the ad hoc periodgon that belong to each of the groups I-IV is connected (the set of indices $j$ for which $\lambda_j$ is of given type is a segment in $\Z_{k+1}$) (see Figure~\ref{well_defined} (a)). 
Putting all the pieces together, yields:

\begin{proposition}\label{prop:selfintersection}
The only potential self-intersection of the ad hoc periodgon could be between a side in the group II oriented in the 3rd quadrant and sides in the group IV (see Figure~\ref{well_defined} (b)), or between a a side in the group II oriented in the 4th quadrant and sides in the group III.
\end{proposition}

\begin{figure}\begin{center} 
\subfigure[ad hoc periodgon]{\includegraphics[width=4cm]{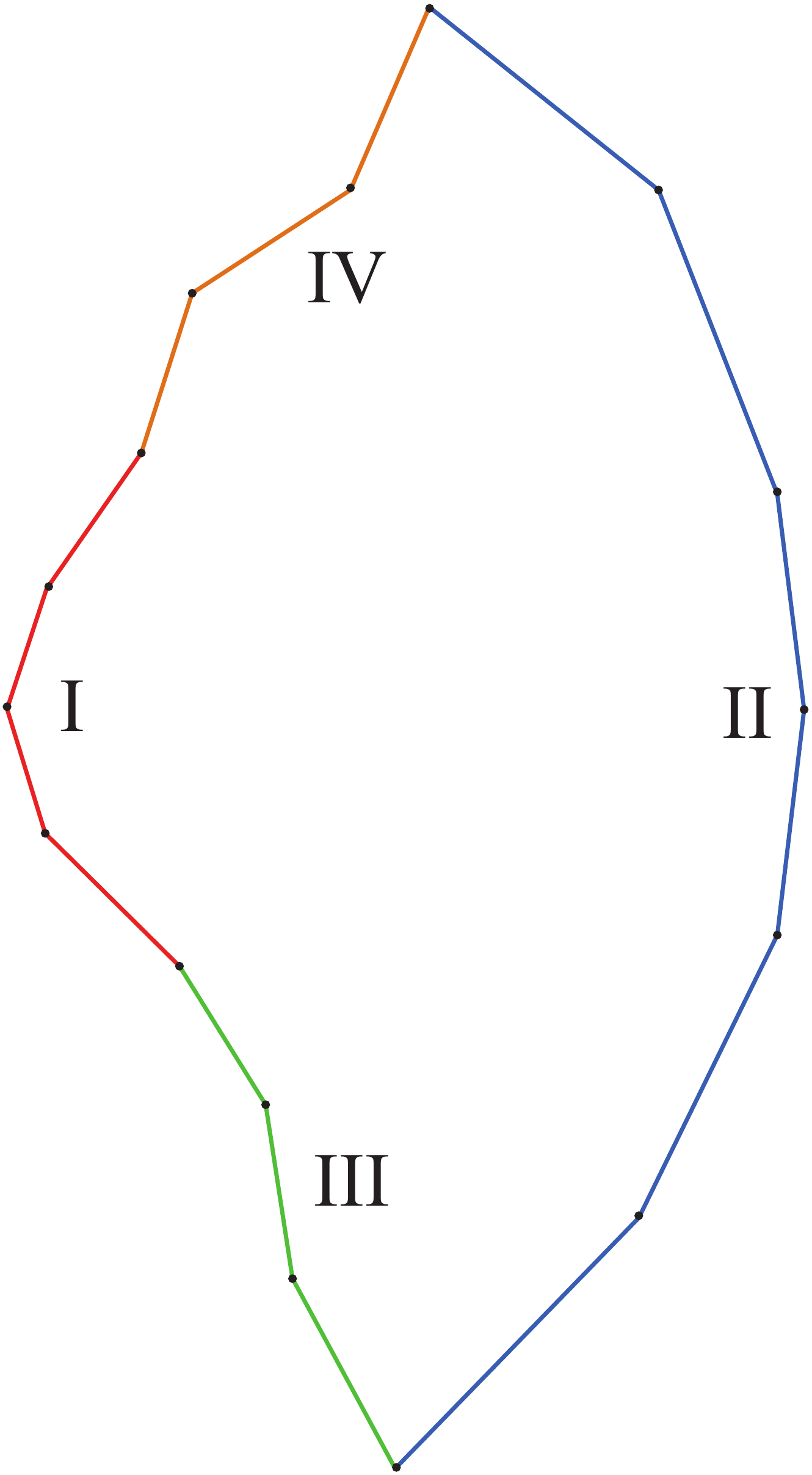}}\qquad\qquad\qquad
\subfigure[self-intersecting ad hoc periodgon]{\includegraphics[width=3.5cm]{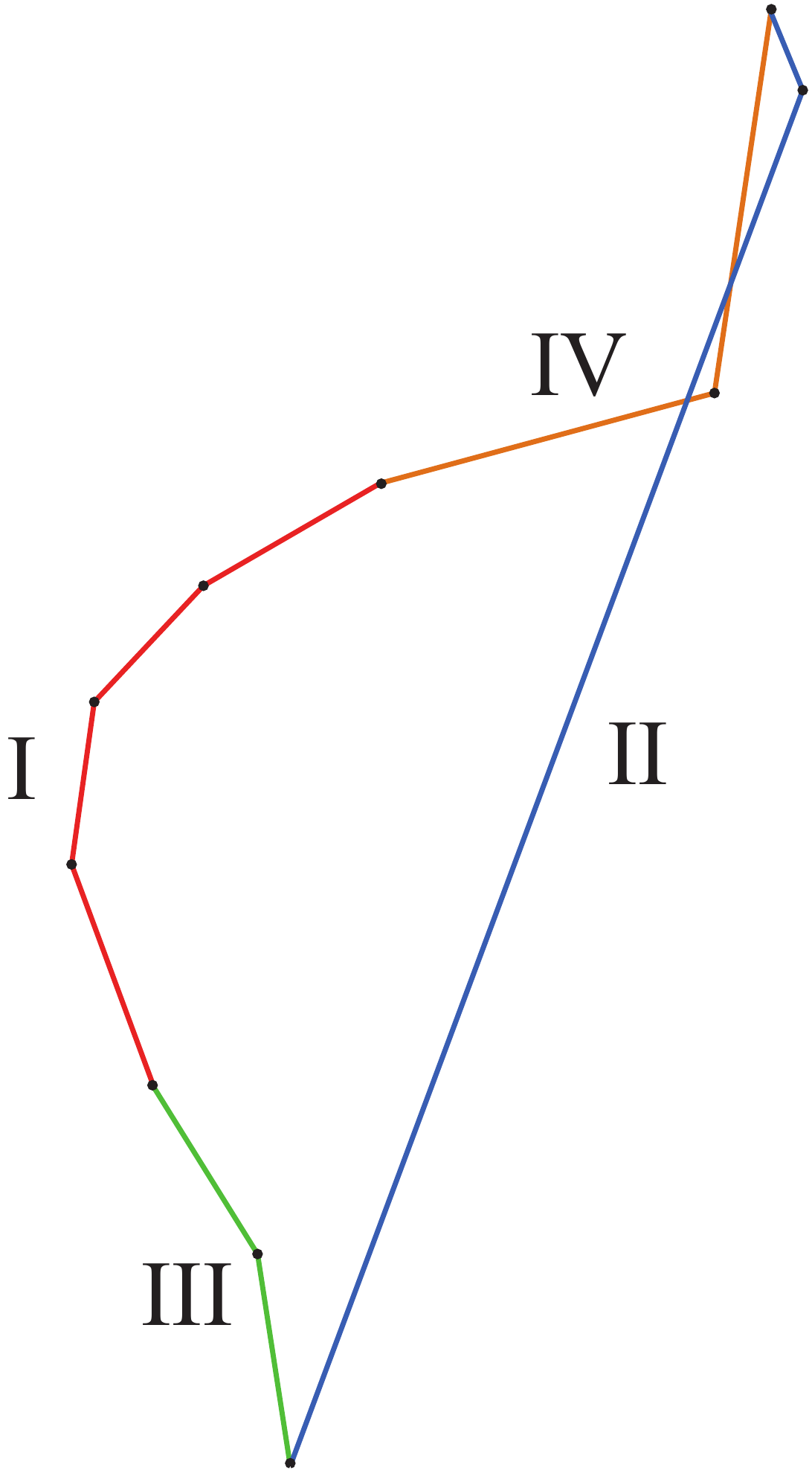}}\qquad
\caption{The parts of the ad hoc periodgon corresponding to regions I, II, III, IV are in red, blue, green and orange respectively. In (b) a potential self-intersection of a side in the group II oriented in the 3rd quadrant with a side in the group IV.}\label{well_defined}
\end{center}\end{figure}

\subsection{The regular movements of the sides of the periodgon} 

\begin{figure}\begin{center}
\subfigure[]{\includegraphics[scale=.6, angle=90]{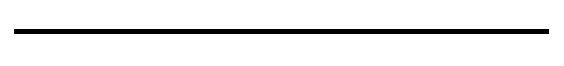}}
\subfigure[]{\includegraphics[scale=.3]{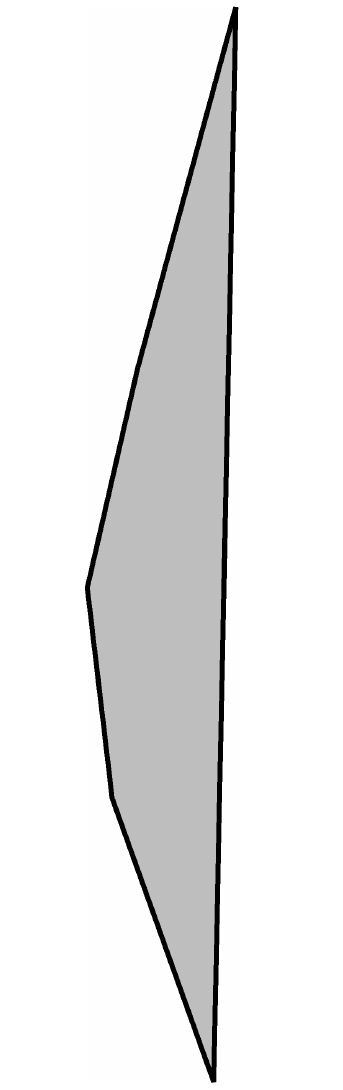}}
\subfigure[]{\includegraphics[scale=.3]{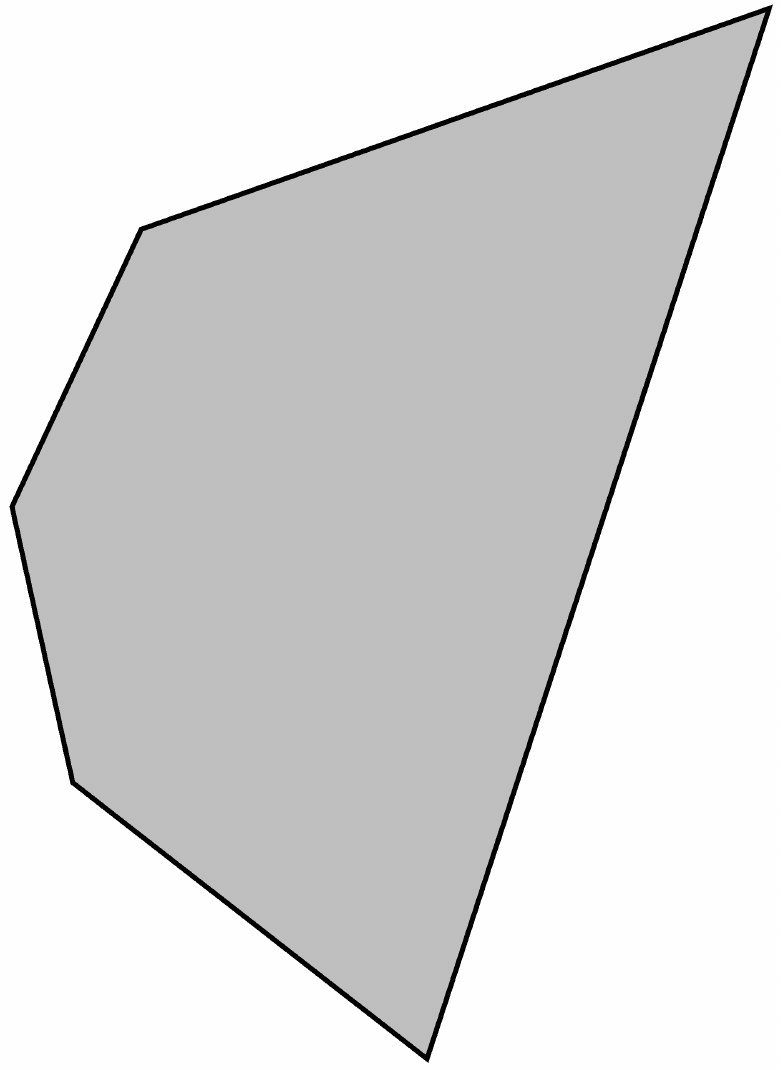}}\hskip-6pt
\subfigure[]{\includegraphics[scale=.3]{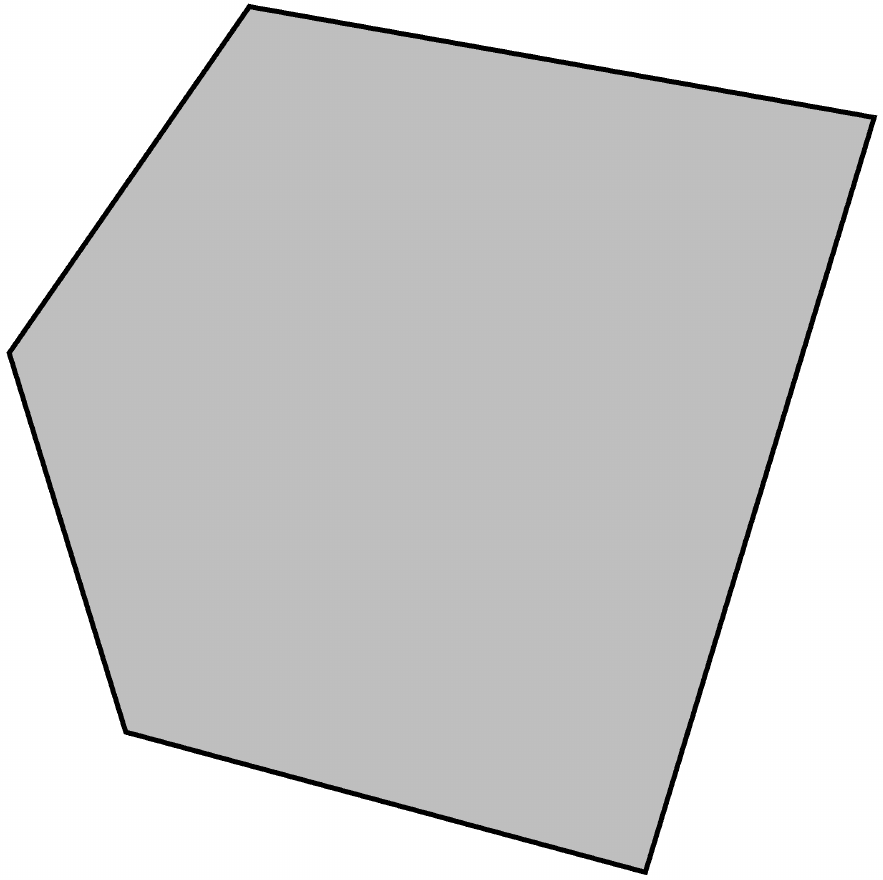}}
\subfigure[]{\includegraphics[scale=.3]{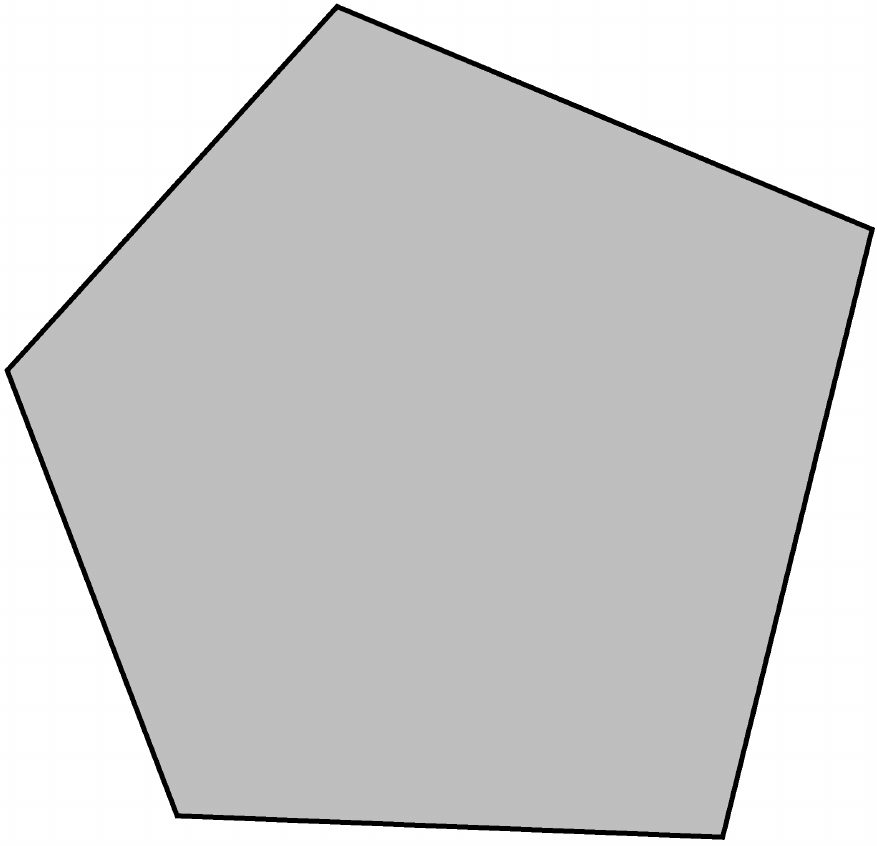}}\
\subfigure[]{\includegraphics[scale=.3]{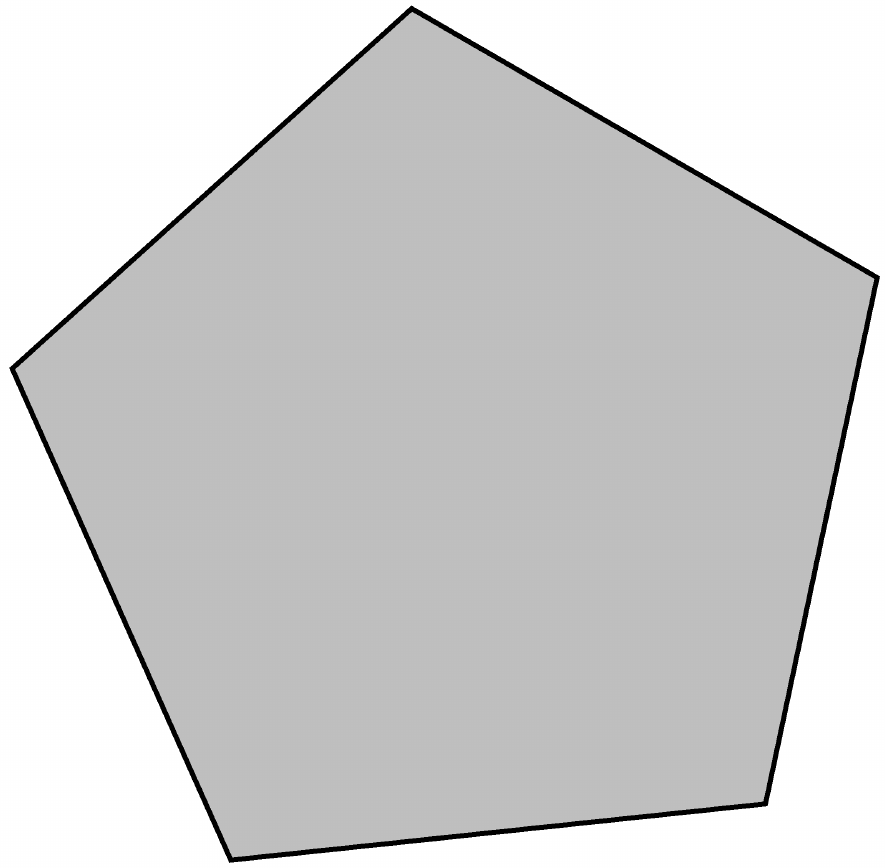}}
\caption{The periodgon for $k=4$, $\alpha=0$, $\theta=-\pi/6$, and increasing $s$ from $0$ to $1$. The rightmost side is the one corresponding to the root $z_0(s)$. } \label{fig:period0}\end{center}\end{figure}

\begin{figure}\begin{center}
\subfigure[]{\hskip-12pt\includegraphics[scale=.3]{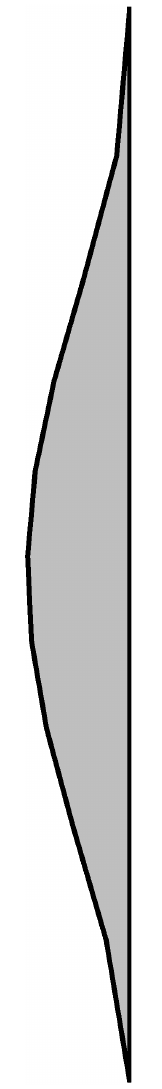}}\ \
\subfigure[]{\includegraphics[scale=.3]{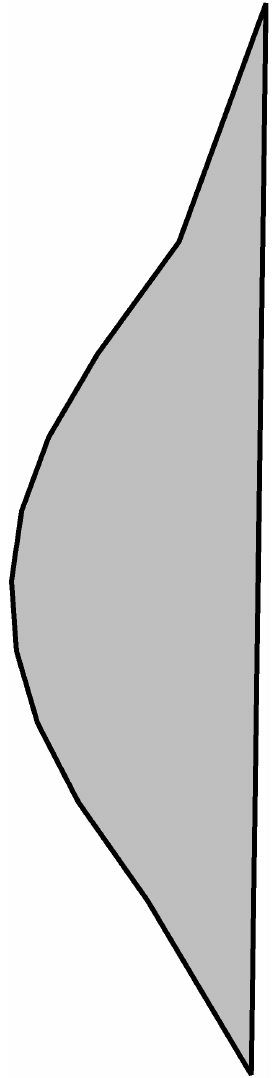}}\ \
\subfigure[]{\includegraphics[scale=.3]{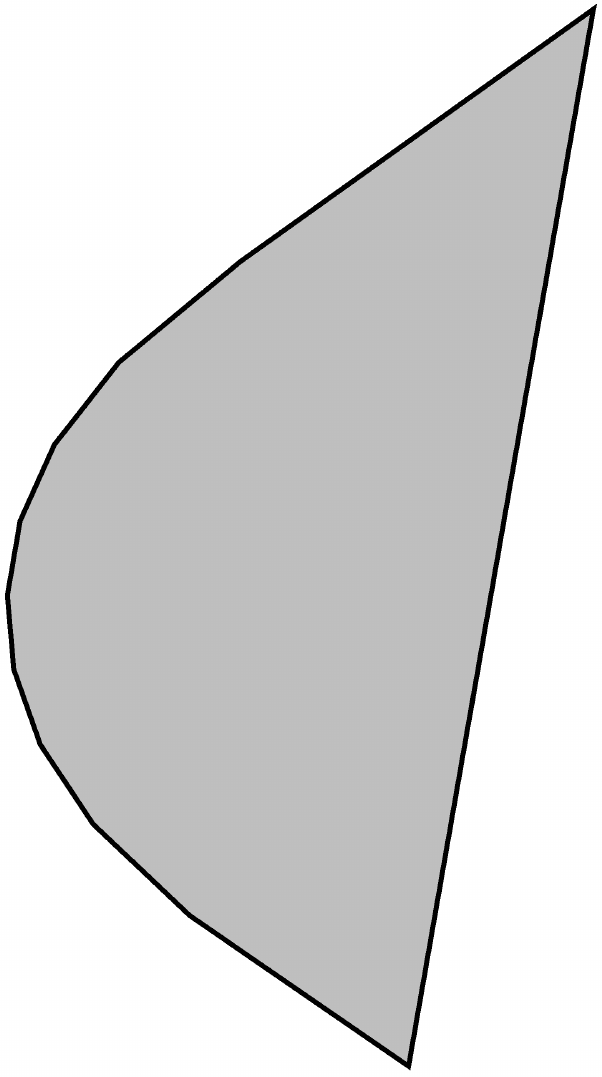}}\ 
\subfigure[]{\includegraphics[scale=.3]{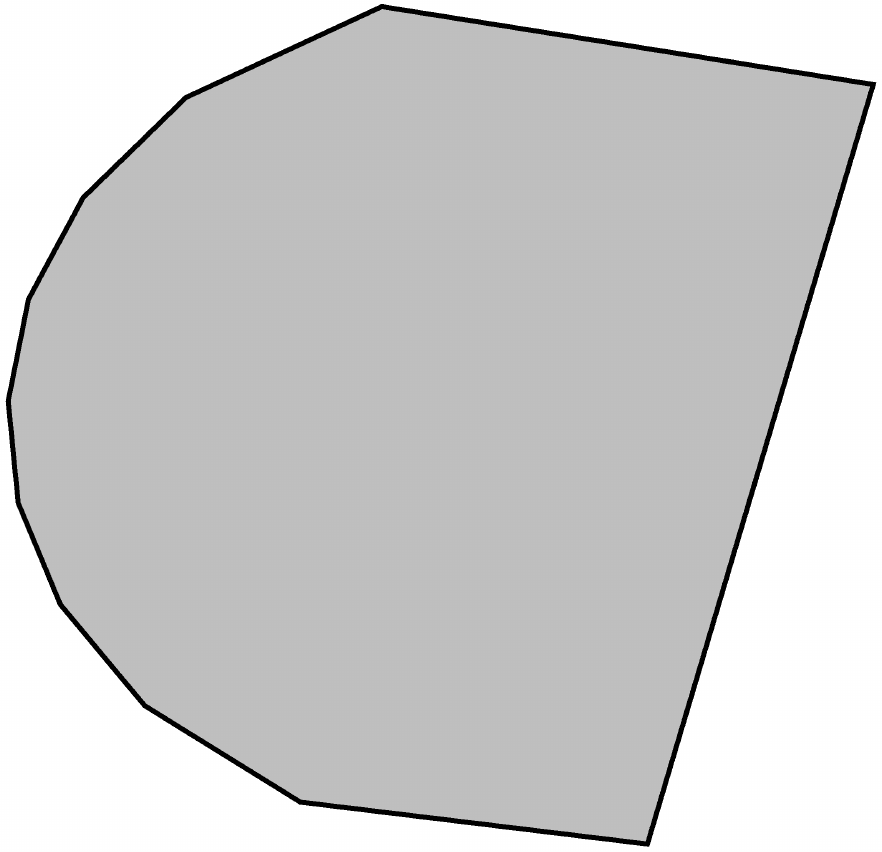}}\
\subfigure[]{\includegraphics[scale=.3]{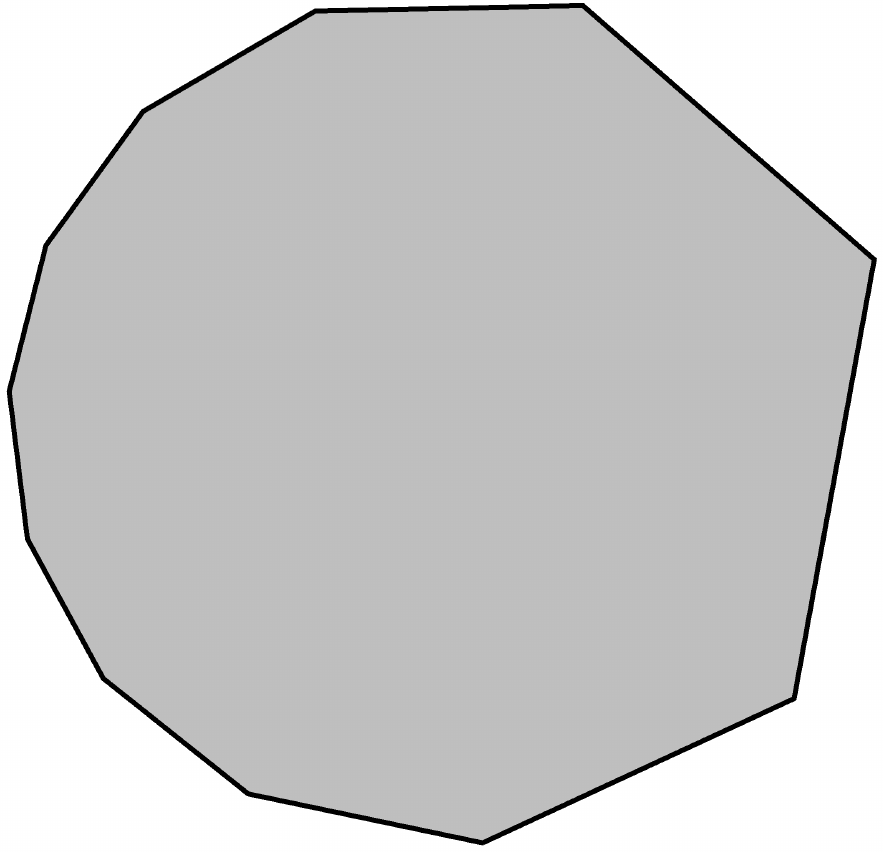}}\ \
\subfigure[]{\includegraphics[scale=.3]{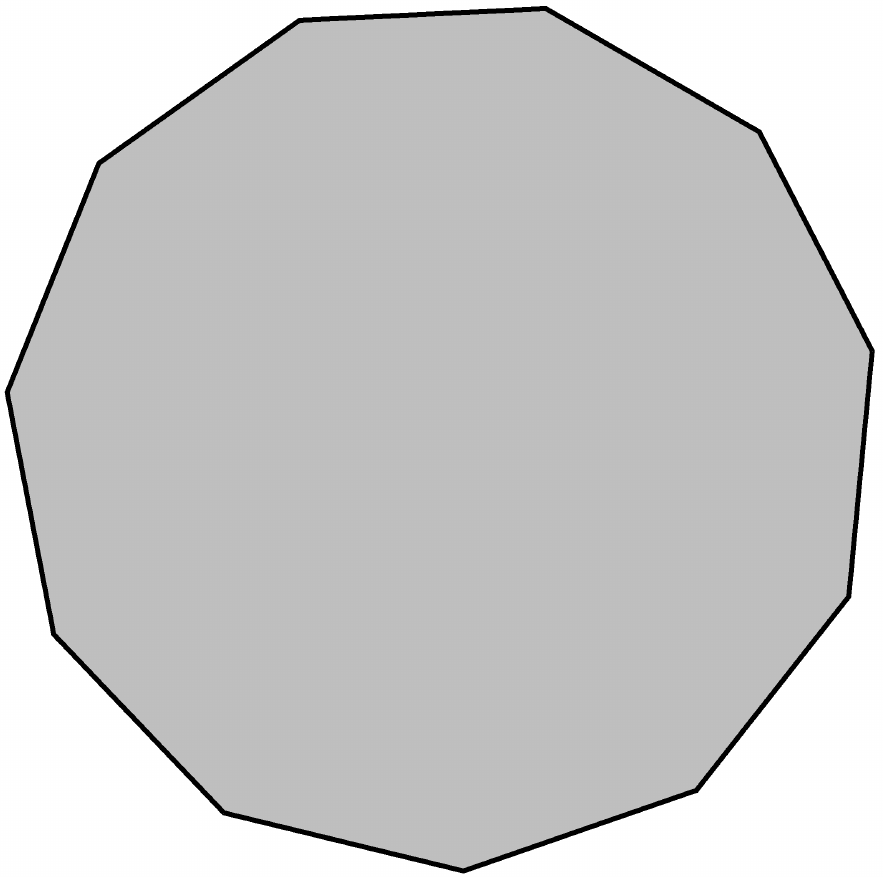}\hskip-12pt}
\caption{The periodgon for $k=10$, $\alpha=0$, $\theta=-\pi/15$, and increasing $s\in(0,1]$. The rightmost side is the one corresponding to the root $z_0(s)$.} \label{fig:period}\end{center}\end{figure}

At $s=0$ the periodgon is flat with one long side and $k$ small sides in the opposite direction, while it is regular at $s=1$. It seems that the sides of the periodgon are monotonically rotating when $s$ increases from $0$ to $1$ (see Figures~\ref{fig:period0} and~\ref{fig:period}). We have been able to prove this for some regions of parameter space. This is done in the following proposition.

\begin{proposition}\label{prop:dlambda} 
\begin{align*}
\frac{d\arg \lambda_j}{ds} &= (1-s)^{k-1} s^{3k+2} \Im(x_j)\cdot
\frac{(k+1)\big|\tfrac{(1-s)^k}{s^{k+1}}x_j\big|^2-2k\Re\big(\tfrac{(1-s)^k}{s^{k+1}}x_j\big) +k-1}{\big|(1-s)^kx_j-s^{k+1}\big|^4},\\
\frac{d\log|\lambda_j|}{d\theta} &= -(1-s)^{k} s^{3k+3}\Im(x_j)\cdot
\frac{(k+1)\big|\tfrac{(1-s)^k}{s^{k+1}}x_j\big|^2-2k\Re\big(\tfrac{(1-s)^k}{s^{k+1}}x_j\big) +k-1}{(k+1)\big|(1-s)^kx_j-s^{k+1}\big|^4},
\end{align*}
 where $x_j=e^{-i\theta}z_j$.
In particular if either $\Re(x_j)\leq 0$, or if $s\in (0,\frac12]$, $\theta\in(-\frac{\pi}k,0)$ and $j\neq 0$, then  $\frac{d\arg \lambda_j}{ds}$ and  $-\frac{d\log|\lambda_j|}{d\theta}$ have the sign of $\Im(x_j)$.
\end{proposition}
\begin{proof} Using the formulas \eqref{lambda}, \eqref{dxds} and \eqref{dxdtheta} we calculate
\begin{equation}\label{eq:dlambda}\begin{aligned}
(1-s)\frac{d\log\lambda_j}{ds} &=-k-s^k\frac{(k+1)(1-s)^kx_j-ks^{k+1}}{\big((1-s)^kx_j-s^{k+1}\big)^2},\\
\frac{d\log\lambda_j}{d\theta} &= -is^{k+1}\frac{(k+1)(1-s)^kx_j-ks^{k+1}}{(k+1)\big((1-s)^kx_j-s^{k+1}\big)^2},
\end{aligned}\end{equation}
from which the formulas follow.

Denoting \begin{align*}Q(X):=(k+1)X^2-2k X+(k-1)= \big((k+1)X-(k-1)\big)\big(X-1\big),\end{align*}
then $Q(X)>0$ for positive $X$ with $X\notin \big[\frac{(k-1)}{(k+1)}, 1\big]$.
We have
$$ (k+1)\big|\tfrac{(1-s)^k}{s^{k+1}}x_j\big|^2-2k\Re\big(\tfrac{(1-s)^k}{s^{k+1}}x_j\big) +k-1 \geq  Q(\big|\tfrac{(1-s)^k}{s^{k+1}}x_j\big|)>0,$$
since $|x_j|>\frac12$ for $j>0$ and $s<\frac12$. 
Indeed, for $s=0$ we have $|x_j|=(k+1)^{1/k}>\frac12$, so suppose that $|x_j|=\frac12$ for some $(s,\theta)$. Then, on the one hand, $|(k+1)(1-s)^kx_j-ks^{k+1}|=2^{-(k+1)}$. On the other hand, if $s<\frac12$ then $|(k+1)(1-s)^kx_j-ks^{k+1}|\geq(k+1)2^{-(k+1)}-k2^{-(k+1)}=2^{-(k+1)}$. The equality is possible only when $x_j=\frac12$ and $s=\frac12$. 
\end{proof}

\subsection{The potential homoclinic bifurcations of case (2)(e) in Theorem~\ref{main_thm}}\label{sec:birf_loops}
In order to complete the proof of Theorem~\ref{main_thm} (especially the part (3)(b)) we  discuss here the potential (but conjectured not to occur) bifurcations of codimension $3$ that can occur if the ad hoc periodgon has self-intersections elsewhere than when $s=0$ or $\theta=0$ and $s\in (0,\frac12)$. 
The ad hoc periodgon has no self-intersection for $s$ small, for $s$ close to $1$, in the neighborhood of $\theta=\frac{(2m+1)\pi}{k}$, and close to 
$\theta=\frac{2m\pi}{k}$, $s\in [\frac12,1]$.

There are only two kinds of generic bifurcations that can bring a self-intersection:
\begin{enumerate}
\item either $\lambda_0$ and $\lambda_1$ become collinear (of inverse orientation); 
\item or a vertex of the periodgon crosses another side (from the group II, see Proposition~\ref{prop:selfintersection}). 
\end{enumerate} 

The bifurcation diagram of case (1) is given in Figure~\ref{coallesce} when  $\arg(\lambda_0)-\arg(\lambda_1)$ crosses $\pi $ transversely. 
It represents the bifurcation diagram of $\theta=0$, $s\in (0,\frac12)$ when one varies the parameters $\theta$ and $\alpha$. 

\begin{figure}\begin{center}
\includegraphics[width=9cm]{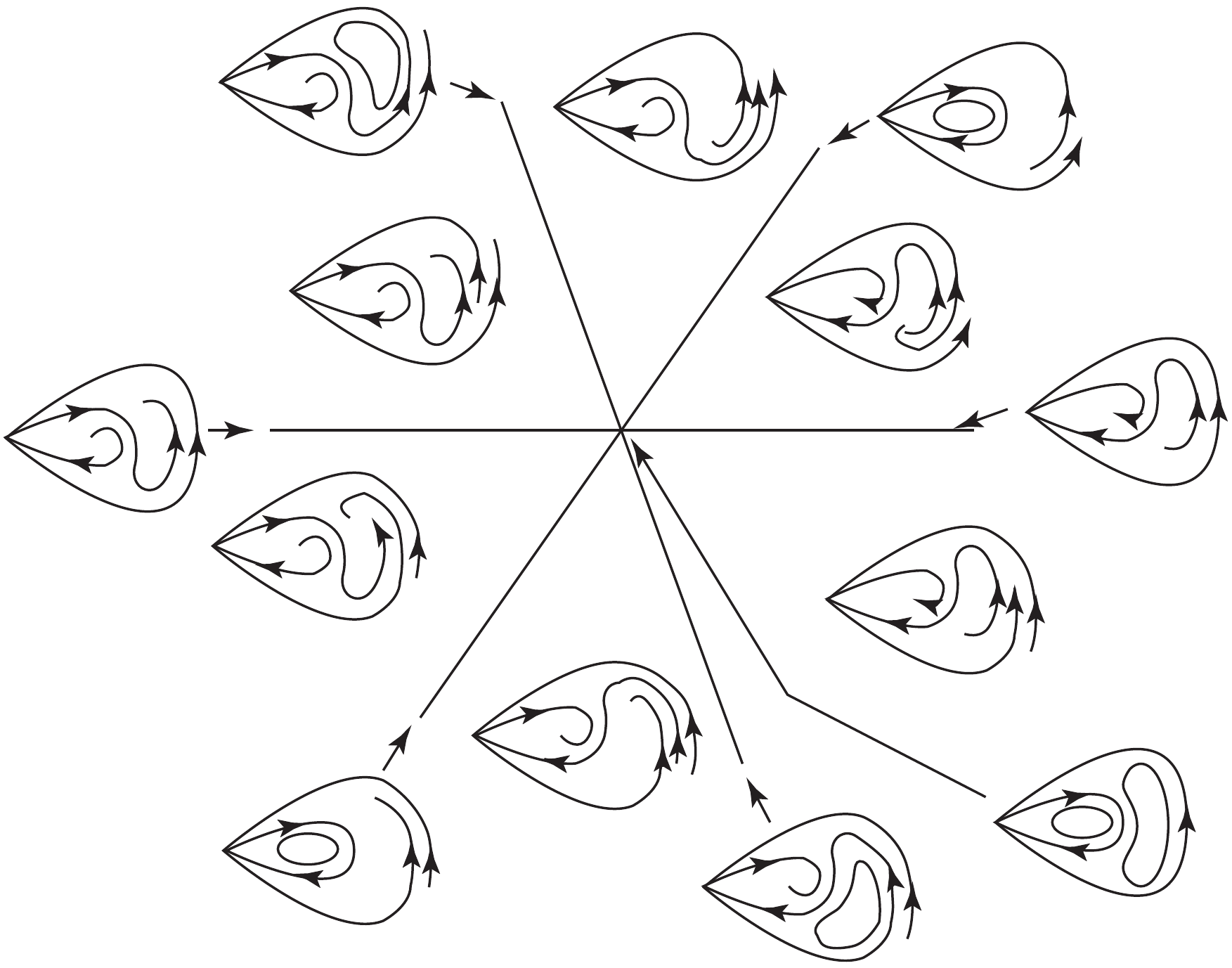}
\caption{The bifurcation diagram when $\lambda_0$ and $\lambda_1$ become collinear (of inverse orientation). (The other separatrices of infinity have not been drawn). This bifurcation occurs also when the system is reversible. But there are other homoclinic loops at the same time.}\label{coallesce}
\end{center}\end{figure}

The bifurcation diagram of case (2) is given in Figure~\ref{double_homoclinic3} in case the vertex crosses the side transversely.

\begin{figure}\begin{center}
\subfigure[The periodgon]{\includegraphics[width=2cm]{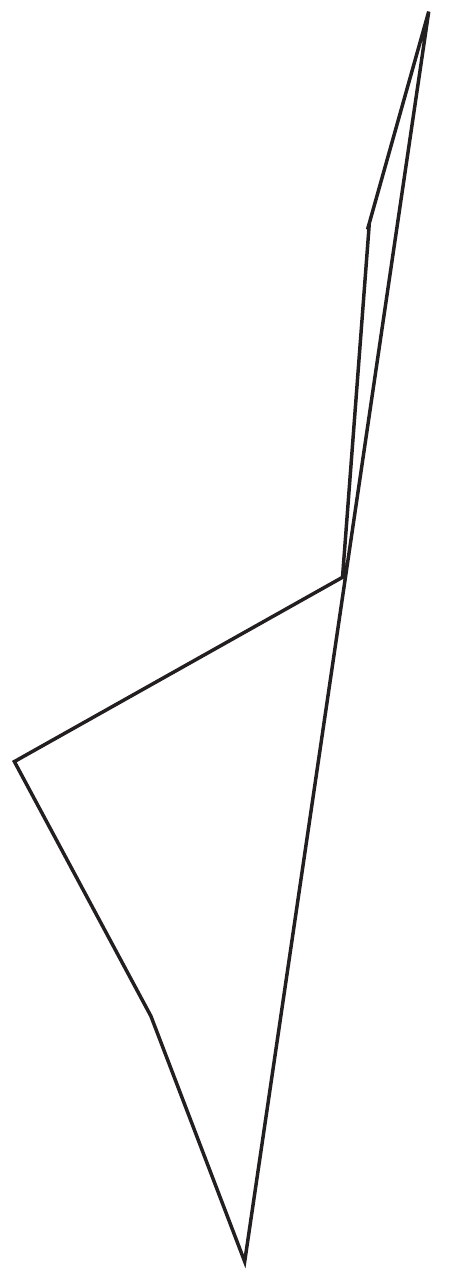}}\qquad
\subfigure[The bifurcation diagram]{\includegraphics[width=9cm]{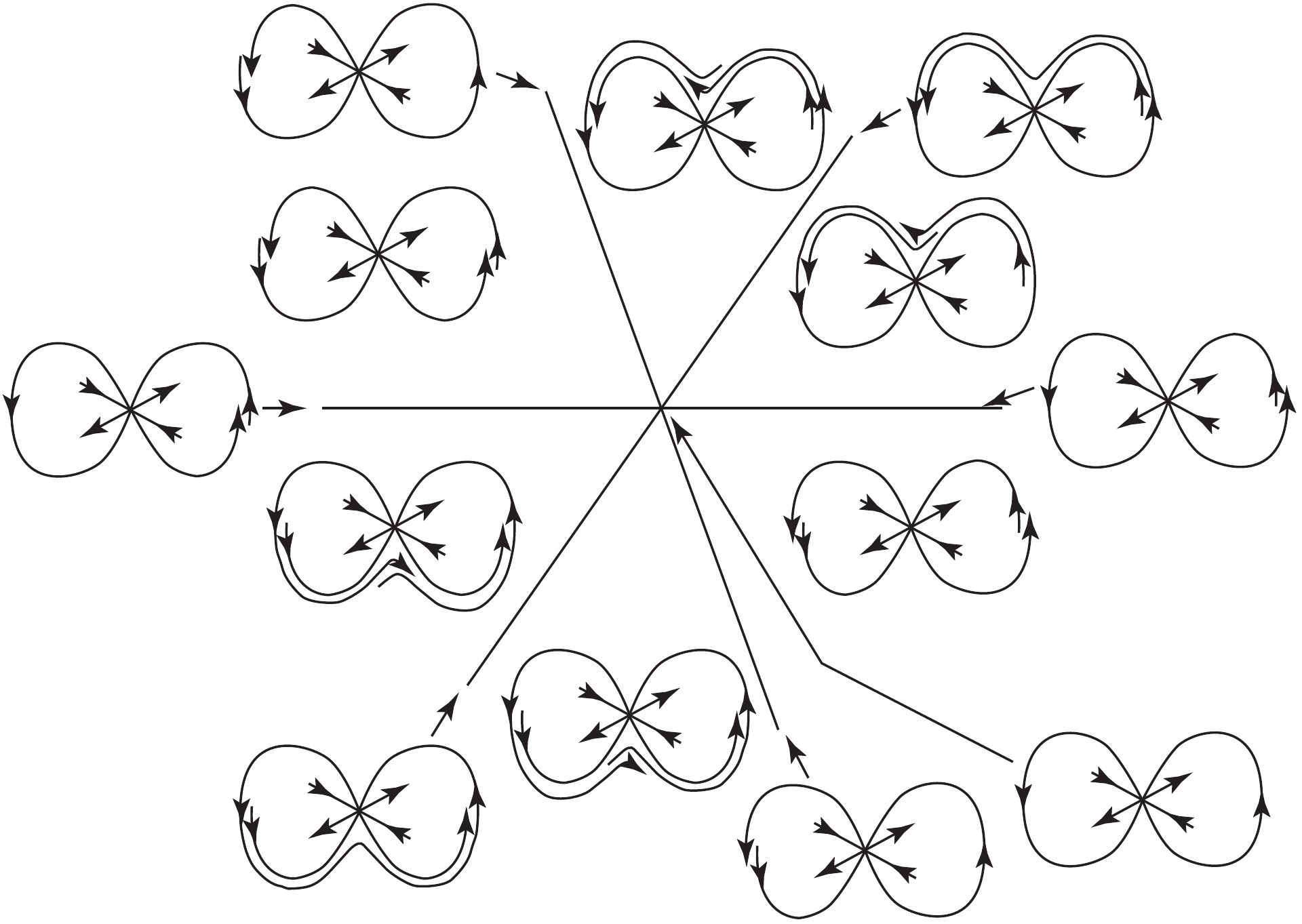}}
\caption{The bifurcation diagram when one vertex crosses another side of the periodgon. The singular points are separated in three groups: one group inside each loop, and $z_0$ outside the two loops.}\label{double_homoclinic3}
\end{center}\end{figure}

\section{Normal form of generic $\ell$-parameter perturbations of a vector field with a parabolic point of codimension $k$}

In this section we treat the more general case of a generic $\ell$-parameter perturbation of a parabolic point of codimension $k$ for $\ell\leq k$. 

\begin{definition}  Let $\dot z= \omega_\eps=\sum_{n\geq0} c_n(\eps)z^n$ be an $\ell$-parameter unfolding of a vector field having a parabolic singularity $\dot z=\omega_0(z) = z^{k+1} + O(z^{k+2})$, $k\geq 1$ of codimension $k$. The unfolding is \emph{generic} if \begin{equation}\left| \frac{\partial (c_{\ell-1}, \dots, c_1, c_0) }{\partial(\eps_{\ell-1},\dots,\eps_1,\eps_0)} \right|\neq0.\label{genericity_1}\end{equation} 
\end{definition}

We describe here a normal form inspired from Kostov \cite{Ko} and show its essential uniqueness and the rigidity of its parameter. Such a normal form appears also indirectly in Proposition 5.14 of \cite{Ri}, but with no discussion of uniqueness and canonical parameter.

\begin{theorem} Let $\omega_\eps$ be a generic $\ell$-parameter unfolding of a vector field having a parabolic singularity $\omega_0(z) = z^{k+1} + O(z^{k+2})$, $k\geq 1$. 
\begin{enumerate} \item There exists a change of coordinate and parameter $(z,\eps)\mapsto (\tilde{z},\tilde{\eps})$ to a normal form
\begin{equation}\dot{\tilde{z}} = P_{\tilde{\eps}}(\tilde{z})/(1 + A(\tilde{\eps}) \tilde{z}^k).\label{Kostov_normal_form}\end{equation}
where  
\begin{equation}
P_{\tilde{\eps}}(\tilde{z})= \tilde{z}^{k+1} + b_{k-1}(\tilde{\eps}) \tilde{z}^{k-1} + \dots + b_1(\tilde{\eps})\tilde{z}+b_0(\tilde{\eps}),\quad
b_j(\tilde{\eps})=\tilde{\eps}_j\ \text{for}\ j\leq \ell-1,\label{P_Kostov} 
\end{equation}  
with $b_i(0)=0$ for all $i$, and where $A(\tilde\epsilon)$ is analytic and equal to the sum of the inverses of the eigenvalues at the small zero points of $\omega_\eps$.
\item This normal form is almost unique: let $\dot{\tilde{z}} = P_{\tilde{\eps}}(\tilde{z})/(1 + A(\tilde{\eps}) \tilde{z}^k)$ and $\dot{\hat{z}}= \widehat{P}_{\hat{\epsilon}}(\hat{z})/(1 + \widehat{A}(\hat{\epsilon}) \hat{z}^k)$ be two $\ell$-parameter families of vector fields as in (1). Suppose that the two families are locally conjugate through a change of coordinate and parameter $(\tilde{z},\tilde{\eps})\mapsto (\hat{z}, \hat{\eps})= (\varphi(\tilde{z},\tilde{\eps}), h(\tilde{\eps}))$. Then there exist $\mu$ and $T\in \C\{\hat{\eps}\}$ such that $\mu^k=1$ and 
\begin{equation}\begin{cases} 
h(\tilde{\eps})= (\mu^{-(\ell-2)}\tilde{\eps}_{\ell-1}, \dots, \tilde{\eps}_1, \mu\tilde{\eps}_0),\\
\varphi_{\tilde{\eps}}(\tilde{z}):= \varphi(\tilde{z},\tilde{\eps})=\Phi_{\tilde{\eps}}^{T(\tilde{\eps})}\left(\mu\tilde{z}\right),\end{cases}\end{equation}
where $\Phi_{\tilde{\eps}}^{T(\tilde{\eps})}$ is the flow of  
$\dot{\tilde{z}} = P_{\tilde{\eps}}(\tilde{z})/(1 + A(\tilde{\eps}) \tilde{z}^k)$ 
at time $T(\tilde{\eps})$.
Moreover, $A(\tilde{\eps})=\widehat A(h(\tilde{\eps}))$ and $P_{\tilde{\eps}}(\tilde{z}) = \widehat{P}_{h(\tilde{\eps})}(\mu\tilde{z})$ hold for $\tilde{\eps}$ near $0$.

In particular, the parameters are \emph{canonical} in this normal form. 
\end{enumerate}
\end{theorem}

\begin{proof} 
\begin{enumerate}\item 
Let us suppose that  the $\ell$-parameter family of vector fields $\omega_\eps(z)= \sum_{n\geq0} c_n(\eps)z^n$ satisfies \eqref{genericity_1}. It can be enlarged to a generic $k$-parameter family $\dot z =\widehat{\omega}_\eta(z) = \omega_\eps(z) +\sum_{j=\ell}^{k-1} \eps_jz^j$ with the multi-parameter $\eta=(\eps_{k-1}, \dots, \eps_2,\eps_1,\eps_0)$. Kostov's Theorem states that there exists a change of coordinate and parameters to a normal form
\[\widetilde{\omega}_{a} = \widetilde{P}_{a}(\tilde z)/(1 + \tilde{A}(a) \tilde z^k) \]
with a new multi-parameter $a=(a_{k-1}, \dots, a_1, a_0)$ and
$\widetilde P_a(\tilde z) = \tilde z^{k+1} + a_{k-1} \tilde z^{k-1} + \dots + a_1 \tilde z+a_0 $.
This new family is again generic. Indeed the coefficients $a_0, \dots, a_{\ell-1}$ are symmetric functions of the singularities with same order of magnitude as before the change of coordinate and parameters. 
Hence the restriction of the change of coordinate and parameters to $\eps_{k-1}= \dots= \eps_\ell=0$ provides the required change to the normal form. We end up with a reparametrization letting $\tilde{\eps}_j=a_j(0, \dots,0, \eps_{\ell-1}, \dots,\eps_0)$ for $j=0,\dots, \ell-1$. 

\item For the uniqueness we use a method of infinite descent as in the proofs of Theorem 3.5 in \cite{RT} and Theorem 3.36 of \cite{CR}. 
Since the proof is completely similar, we will be brief on the details. Before starting the infinite descent, we must reduce the problem.

\smallskip \noindent{\bf First reduction.} We first consider the case $\tilde{\eps}=0$, for which the theorem follows from a mere calculation. Then $\varphi_0'(0)=  \mu$, where $\mu^k=1$. 
We change $(\tilde{z}, \tilde{\eps}_1,\tilde{\eps}_0)\mapsto \left(\mu\tilde{z},\mu^{-(\ell-2)}\tilde{\eps}_{\ell-1}, \dots, \tilde{\eps}_1,\mu\tilde{\eps}_0\right)$ in the first vector field, so as to limit ourselves to the case $\varphi_0'(0)=1$.

\smallskip \noindent{\bf Second reduction.} It is easily checked that the flow $\Phi_0^t$ of $\dot{\tilde{z}} = \tilde{z}^{k+1}/(1+A(0)\tilde{z}^k)$ at time $t$ has the form
\begin{equation}\Phi_0^t(\tilde{z}) = \tilde{z}(1+ g_t(\tilde{z}^k))=\tilde{z}+ t\tilde{z}^{k+1} +tO(\tilde{z}^{2k+1}).\label{flow_0}\end{equation}
Let $\Phi_{\tilde{\eps}}^t$ be the flow of the first equation at time $t$, let $$\psi_{\tilde{\eps}}(t,\tilde{z} )= \Phi_{\tilde{\eps}}^t\circ \varphi_{\tilde{\eps}}(\tilde{z} ),$$ and let
$$K(\tilde{\eps},t)=\left(\tfrac{\partial}{\partial \tilde{z}}\right)^{k+1}\psi_{\tilde{\eps}}(t,0).$$
We want to find a solution $T(\tilde\eps)$ to $K(\tilde{\eps},T(\tilde\eps))=0$ by the implicit function theorem,
and then change 
$$\tilde{z}\mapsto \Phi_{\tilde\eps}^{T(\tilde{\eps})}\tilde{z}$$ 
in the first system.
We know that there exists $t_0$ such that  $K(0,t_0)=0$ because of the form of $\Phi_0^t$ in \eqref{flow_0}. Moreover, $K(0,t)= \varphi_0^{(k+1)}(0)+t(k+1)!$, yielding
$\frac{\partial K}{\partial t}(0,0)= (k+1)!\neq0$. Hence there exists a unique analytic germ $T(\tilde{\eps})$ such that $K(\tilde{\eps},T(\tilde{\eps}))\equiv0$ and $T(0)=t_0$.

\smallskip \noindent{\bf The infinite descent.} After the two reductions, we can suppose that  $\varphi_0 = id$
 and that $\varphi_\eps^{(k+1)} (0) \equiv 0$. 
We now show that $\varphi_\eps= {\rm Id}$ and $h(\tilde{\eps})=\tilde{\eps}$. Note that $A\equiv \widetilde{A}\circ h$  since the sum of the residues at the singular points is invariant. Let 
$$\begin{cases} \widetilde{P}_{\tilde{\eps}}(\tilde{z})= \tilde{z}^{k+1} + b_{k-1}(\tilde{\eps}) \tilde{z}^{k-1} + \dots + b_1(\tilde{\eps})\tilde{z}+b_0(\tilde{\eps}),\\
\widehat{P}_{\hat{\eps}}(\hat{z}) = \hat{z}^{k+1}+ c_{k-1}(\tilde{\eps}) \hat{z}^{k-1} + \dots + c_1(\tilde{\eps})\hat{z}+c_0(\tilde{\eps}),\\
\hat{z}=\varphi_{\tilde{\eps}}(\tilde{z}) = \tilde{z} + \sum_{j\geq0} f_j(\tilde{\eps})\tilde{z}^j,\end{cases}$$ 
where all $b_j, c_j, f_j\in \C\{\tilde\eps\}$ (note that we really wish the $c_j$ to depend on $\tilde{\eps}$, which we can do since $\hat{\eps}=h(\tilde{\eps})$), and we simply write $b_j$ instead of $b_j(\tilde\eps)$, etc. for the functions $b_j, c_j, f_j$ and $h$. 
We introduce the principal ideal $I= \langle \tilde\eps\rangle$ in $\C\{\tilde\eps\}$, and show by induction that $b_j-c_j,  f_j\in I^n$ for all $j$ and for all $n\in \N^*$, from which it will follow that they are identically zero. 
Note that $h-\tilde\eps=(c_{\ell-1}-b_{\ell-1}, \dots, c_1-b_1,c_0-b_0)$.

The conjugacy condition is
\begin{align} \begin{split}
&(1+A\tilde{z}^k)\left(\left(\tilde{z}+\sum_{j\geq0} f_j\tilde{z}^j\right)^{k+1} + \dots +c_1\left(\tilde{z}+\sum_{j\geq0} f_j\tilde{z}^j\right)+ c_0\right) - \\
&\qquad \left(1+A\left(\tilde{z}+\sum_{j\geq0} f_j\tilde{z}^j\right)^k\right)\left(\tilde{z}^{k+1} +  \dots + b_1\tilde{z}+b_0\right)\left(1+\sum_{j\geq1} jf_j\tilde{z}^{j-1}\right)=0.\label{conjugacy_eq}\end{split}\end{align}
which we simply write as $\sum_{j\geq0}g_j\tilde{z}^j=0$. Hence we want to show that all $g_j$ must be identically $0$. 
The $g_j$ are quite complicated but they have a very simple structure of linear terms and this is what we will exploit. 
\begin{itemize}
\item From the two reductions, it is clear that $ b_j, c_j, f_j\in I$. This is our starting point.  
\item The only linear terms in the equations $g_j=0$ for $j=0, \dots, k-1$, are $b_j-c_j$. Hence $b_j-c_j\in I^2$. 
\item The equations $g_{k+j}=0$ with $0\leq j\leq k$ yield $f_j\in I^2$, since the only linear terms are
$A(c_j-b_j) + (k+1-j)f_j=0$ when $j<k$ and $Af_0+f_k$ when $j=k$. 
\item Remember that $f_{k+1} \equiv 0$ because of the reduction. 
\item The equations $g_\ell=0$ with $\ell>2k+1$ yield $f_{\ell-k}\in I^2$, since the only linear terms in $g_\ell$ are $-(\ell-2k-1)(f_{\ell-k}+Af_{\ell-2k})$.
\item Hence, all $b_j-c_j, f_j\in I^2$. 
\item We now suppose that $b_j-c_j, f_j\in I^n$, and we want to show that there are in $I^{n+1}$. 
\item The equations $g_j=0$ for $j=0, \dots k-1$, yield $b_j-c_j\in I^{n+1}$. 
\item The equations $g_{k+j}=0$ with $0\leq j\leq k$ yield $f_j\in I^{n+1}$. 
\item The equations $g_\ell=0$ with $\ell>2k+1$ yield $f_{\ell-k}\in I^{n+1}$.
\item Hence, all $b_j-c_j, f_j\in I^{n+1}$. 
\end{itemize} This concludes the proof.\end{enumerate}
\end{proof}

This gives us a classification theorem

\begin{theorem} Let $\ell\in\{1, \dots, k\}$. Then two germs $\omega_{1,\eps}$ and $\omega_{2,\eta}$ of generic $\ell$-parameter unfolding of a vector field having a parabolic singularity $\omega_0(z) = z^{k+1} + O(z^{k+2})$, $k\geq 1$ are conjugate if and only if their normal forms \eqref{Kostov_normal_form} are conjugate under $$(\tilde{z}_1,\tilde{\eps})\mapsto (\tilde{z}_2,\tilde{\eta})= \left(\mu\tilde{z}_1,(\mu^{-(\ell-2)}\tilde{\eps}_{\ell-1}, \dots, \tilde{\eps}_1, \mu\tilde{\eps}_0)\right)$$ for some $\mu$ such that $\mu^k=1$. 
 \end{theorem}

\section{Bifurcation diagram of a generic $2$-parameter perturbation of a vector field with a parabolic point of codimension $k$}
In this section we study the bifurcation diagram of the vector field

\begin{equation}\dot z = P_\eps(z)/(1 + A(\eps) z^k).\label{Kostov_nf}\end{equation}
where 
\begin{equation}P_\eps(z)= z^{k+1} + b_{k-1}(\eps) z^{k-1} + \dots + b_2(\eps)z^2+\eps_1z +\eps_0, \label{P_K} \end{equation} 
depending on the multi-parameter $\eps=(\eps_1, \eps_0)$, over a small disk $\D_r$ for small values of the parameter. Close to $|z|=r$ the vector field looks like in Figure~\ref{fleur}. It is natural to write $\eps$ as 
$$\eps= \left(-(k+1) \zeta^k(1-s)^ke^{-ik\alpha},\ k\zeta^{k+1}s^{k+1}e^{i(\theta-(k+1)\alpha)}\right),$$
with $s\in [0,1]$, $\theta\in [-2\pi,0]$, $\alpha\in [0,2\pi]$ and $\zeta\in [0,\rho)$ for some small $\rho$, and to add to the  quotient relations  \eqref{rel:quotient} the relation \begin{equation}
(\zeta, s, \alpha, \theta)|_{\zeta=0}\sim (\zeta, 0,0,0)|_{\zeta=0},\label{rel:quotient2}\end{equation}
for all $s,\theta,\alpha$.
The new (fourth) parameter $\zeta=\|\eps\|$ takes into account that the bifurcation diagram is no more exactly a cone, although it has a conical structure. 
\begin{figure}\begin{center} \includegraphics[width=4cm]{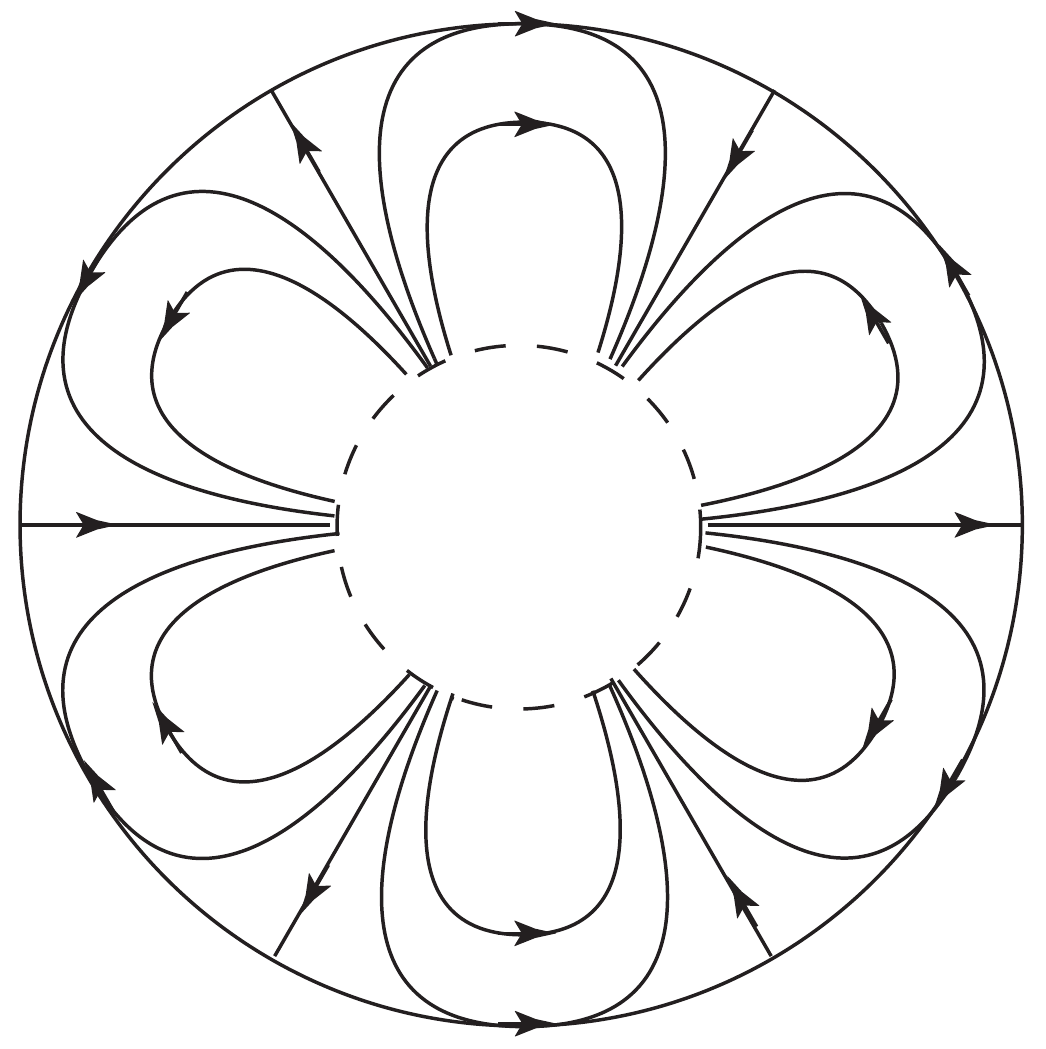}\caption{The phase portrait near $|z|=r$ for sufficiently small $r$ and sufficiently small $\eps$ so that the singular points stay inside $\D_r$ at some distance from the boundary.}\label{fleur}\end{center}\end{figure}

An alternative useful way of looking at the problem is to rescale $(z,t)\mapsto(Z,T)=(\frac{z}{\zeta}, \zeta^kt)$, which brings the system to a system of the form
\begin{equation} \frac{dZ}{dT}=Q(Z) +O(\zeta),\label{rescaled}\end{equation}
where $$Q(Z)=Z^{k+1} -(k+1)(1-s)^ke^{-ik\alpha}Z + ks^{k+1}e^{i(\theta-(k+1)\alpha)},$$ i.e. a small perturbation of the system \eqref{3_par} studied in Section~\ref{sec:model}. There is a price to pay: the system has to be studied on the disk $\D_{\frac{r}{\zeta}}$ whose size grows to infinity.

\subsection{Holomorphic vector fields in a disc}

\begin{definition}[Zone decomposition]~\\
Let $\dot z= \omega(z)$ be a holomorphic vector field in an open disk $\D_r$ of radius $r$.
\begin{enumerate}
\item A \emph{separating trajectory} of the vector field in $D_r$ is a trajectory which either leaves or enters the disk. 
A separating trajectory that both enters and leaves the disk is called  a \emph{dividing trajectory}.
The separating trajectories play the same role inside the disk as the separatrices for polynomial vector fields in $\C=\D_\infty$, while dividing trajectories play the role of homoclinic separatrices.
\item A connected component of the complement of all the separating trajectories in the disk is called \emph{zone}. It consist of trajectories that stay in the disk for all the time.
As in Definition~\ref{def:zones} a zone can be either \emph{periodic}, or an \emph{$\alpha\omega$}-zone, or a \emph{sepal}.
A boundary of a zone consists of trajectories tangent to the boundary $\partial\D_r$.
\item The \emph{skeleton graph} is now defined in the same way as in Definition~\ref{def:zones}. If there is a dividing trajectory in the disk then the skeleton graph is \emph{broken}. 
\end{enumerate}
\end{definition}

\begin{remark}
The skeleton graph may be connected inside $\D_R$ but broken inside $\D_r$, $0<r<R\leq+\infty$.
\end{remark}

For a holomorphic vector field in $\D_r$, the sum of the periods of the singular points in the disk is in general nonzero and the polygon of periods does not close anymore.
But one can nevertheless extend the definitions of periodgon and star domain to this context.

\begin{definition}[Generalized periodgon and star domain]\label{def:generalised_periodgon}~\\
Let $\dot z= \omega(z)$ be a holomorphic vector field in an open disk $\D_r$ of radius $r$, and suppose that all the singular points in $\D_r$ are simple.\begin{enumerate}
\item The \emph{periodic domain in $\D_r$} of a singular point $z_j$ is the union of periodic trajectories surrounding $z_j$ inside $\D_r$ for the rotated vector field
$\dot z=e^{i\arg\nu_j}\omega(z).$ 
The boundary is a periodic trajectory around $z_j$, and which is tangent to $\partial\D_r$. 
\item Inside the periodic domain for $\D_r$ of a singular point $z_j$ we consider a cut  from $z_j$ to the tangency point of the boundary of the periodic domain with $\partial\D_r$ which is orthogonal to the periodic trajectories. 
\item The rectifying chart $$t(z)=\int_{z_\infty}^z\frac{dz}{\omega(z)},$$ with $z_\infty\in\partial\D_r$, is well defined on the cut disk. 
\item The \emph{generalized periodgon for $\D_r$} is the image in the Riemann surface of  $t(z)$ of the complement in $\D_r$ of the union of all periodic domains of the singular points in $\D_r$ (see Figure~\ref{pseudo-gon}). 
\item The \emph{star shape domain} is the image in  the Riemann surface of $t(z)$ of the disk with the cuts (see Figure~\ref{pseudo-gon}). 
It is obtained by gluing to each side $\nu_j$ of the generalized periodgon a perpendicular half-strip of the same width.
\end{enumerate}
\end{definition}

Note that on Figure~\ref{pseudo-gon} the size of the disk $\D_r$ controls the size of the \lq\lq holes\rq\rq\ close to the vertices. 

For the vector field \eqref{Kostov_nf} with small $\zeta=\|\eps\|$, the shape of each hole is close to that of a disk of radius $\frac1{k}r^{-k}$. The width of the branches grows like $\zeta^{-k}$. The parameters $s,\theta$ control the \emph{general shape} of the star domain up to homotheties: by \emph{general shape} we mean that there are small errors of order $o(\zeta^{-k})$. For fixed $s, \zeta, \theta$ the general shape of the star domain rotates with angular speed $k\alpha$.
\medskip

\begin{figure}\begin{center} \includegraphics[width=5cm]{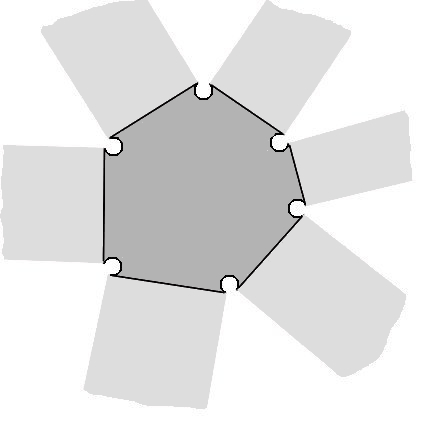}\caption{The star domain and the generalized periodgon of a vector field in $\D_r$. 
The black segments represent the periods.}
\label{pseudo-gon}\end{center}\end{figure}

\subsection{Bifurcation of the skeleton graph inside the disk} 
The normal form \eqref{Kostov_nf} is only local. Hence it does not make sense anymore to speak of homoclinic bifurcations through infinity. The same is of course true for the particular case of the system  \eqref{vector_field} when we restrict it to a disk. But there remains something. Indeed, each time there was a homoclinic loop through infinity, this would break the skeleton graph, because of the existence of many dividing trajectories inside the disk. When we restrict the system to a disk, the skeleton graph is broken inside the disk for a  region of parameter values with nonempty interior, which is a thickening of the former homoclinic  
bifurcation diagram. It consists of the parameter values for which there exist dividing trajectories in the disk, and it has non-empty interior. The Douady-Sentenac combinatorial invariant (equivalent to the skeleton graph and its attachment to the boundary of $\D_r$) changes type when one crosses this parameter region. On the boundaries of the parameter region some trajectories have double tangency with $|z|=r$ (see Figure~\ref{double_tangency}).

\begin{figure}\begin{center}
\subfigure[]{\includegraphics[width=3.5cm]{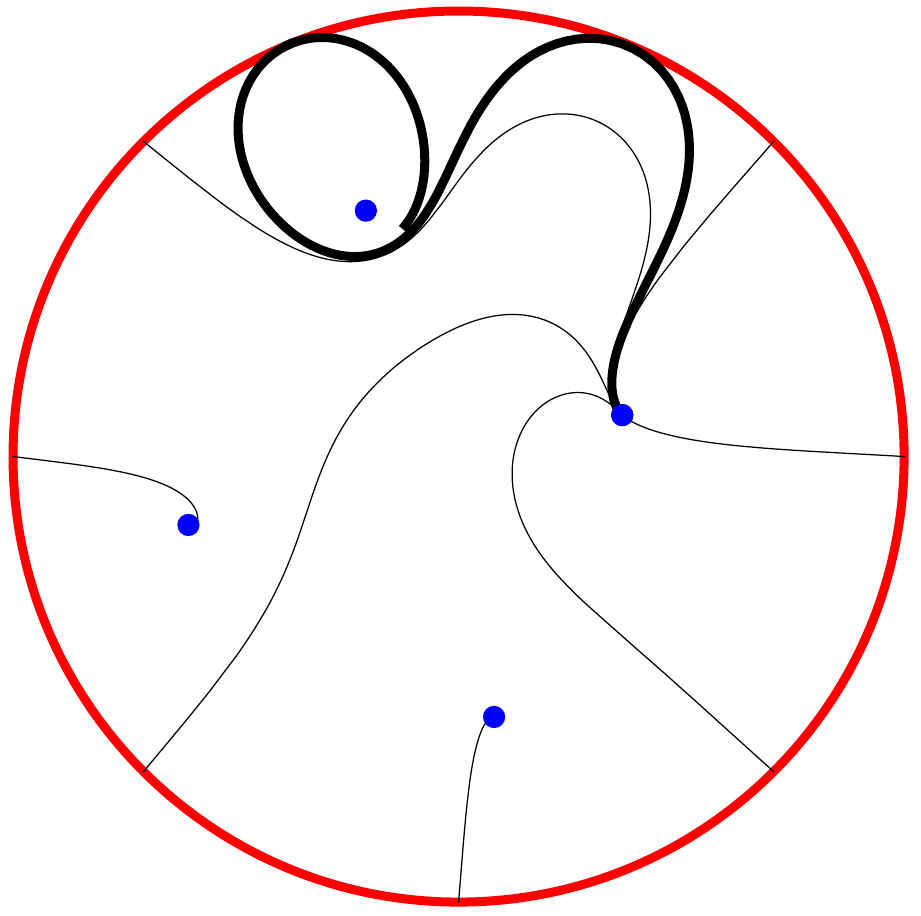}}\qquad
\subfigure[]{\includegraphics[width=3.5cm]{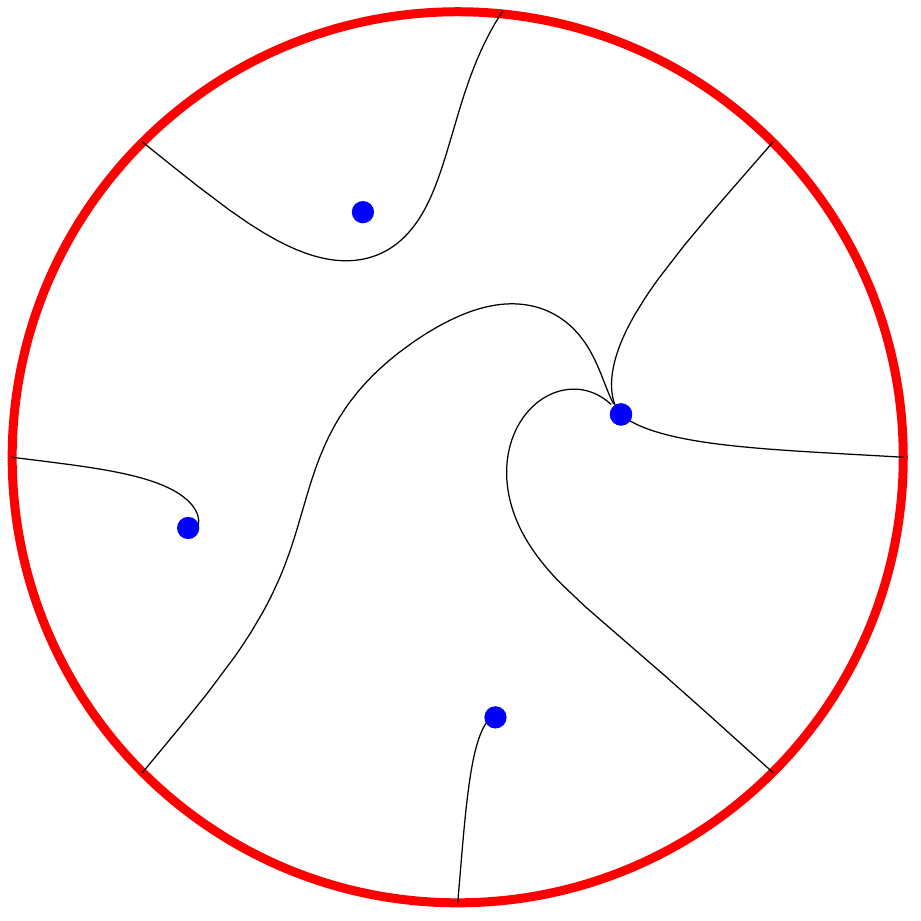}}\qquad
\subfigure[]{\includegraphics[width=3.5cm]{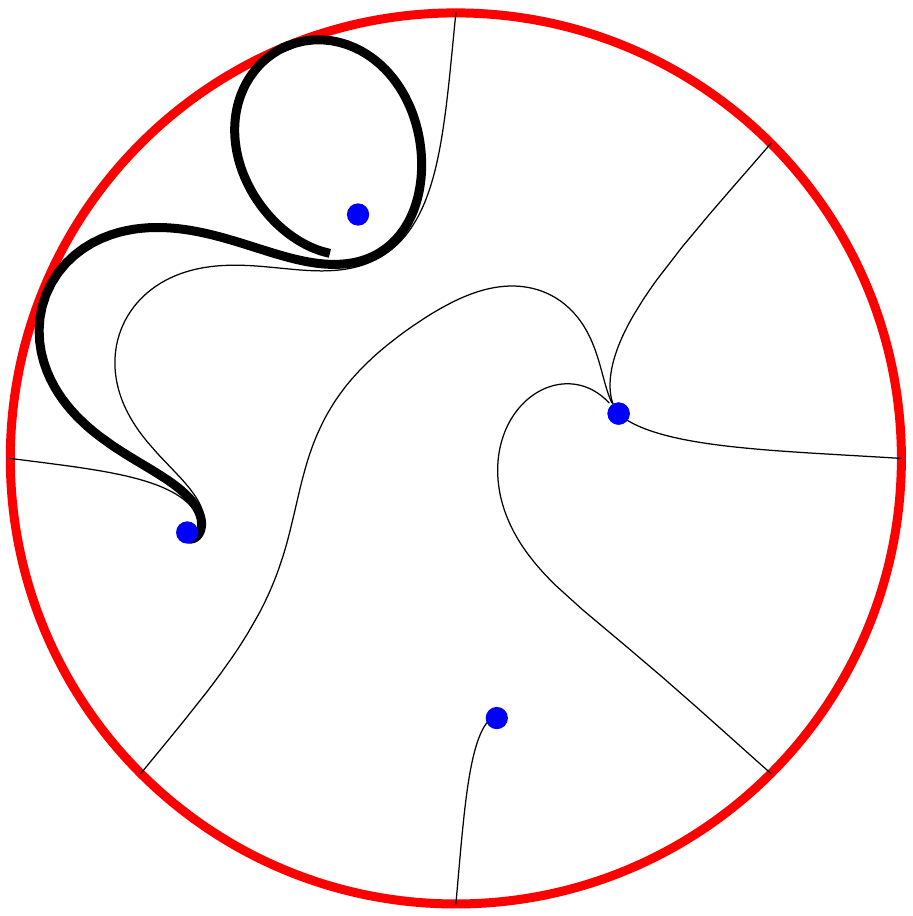}}
\caption{The structurally stable breaking of the skeleton graph inside the disk $\D_r$ in (b) between two double tangencies with the boundary occurring in (a) and (c) is illustrated here for the system \eqref{parabolic_alpha} (which has a parabolic point) with $k=4$ and increasing $\alpha$. In (b) there is an open set of dividing trajectories.}\label{double_tangency}
\end{center}\end{figure}

\begin{figure}\begin{center}
\subfigure[]{ \includegraphics[width=3.8cm]{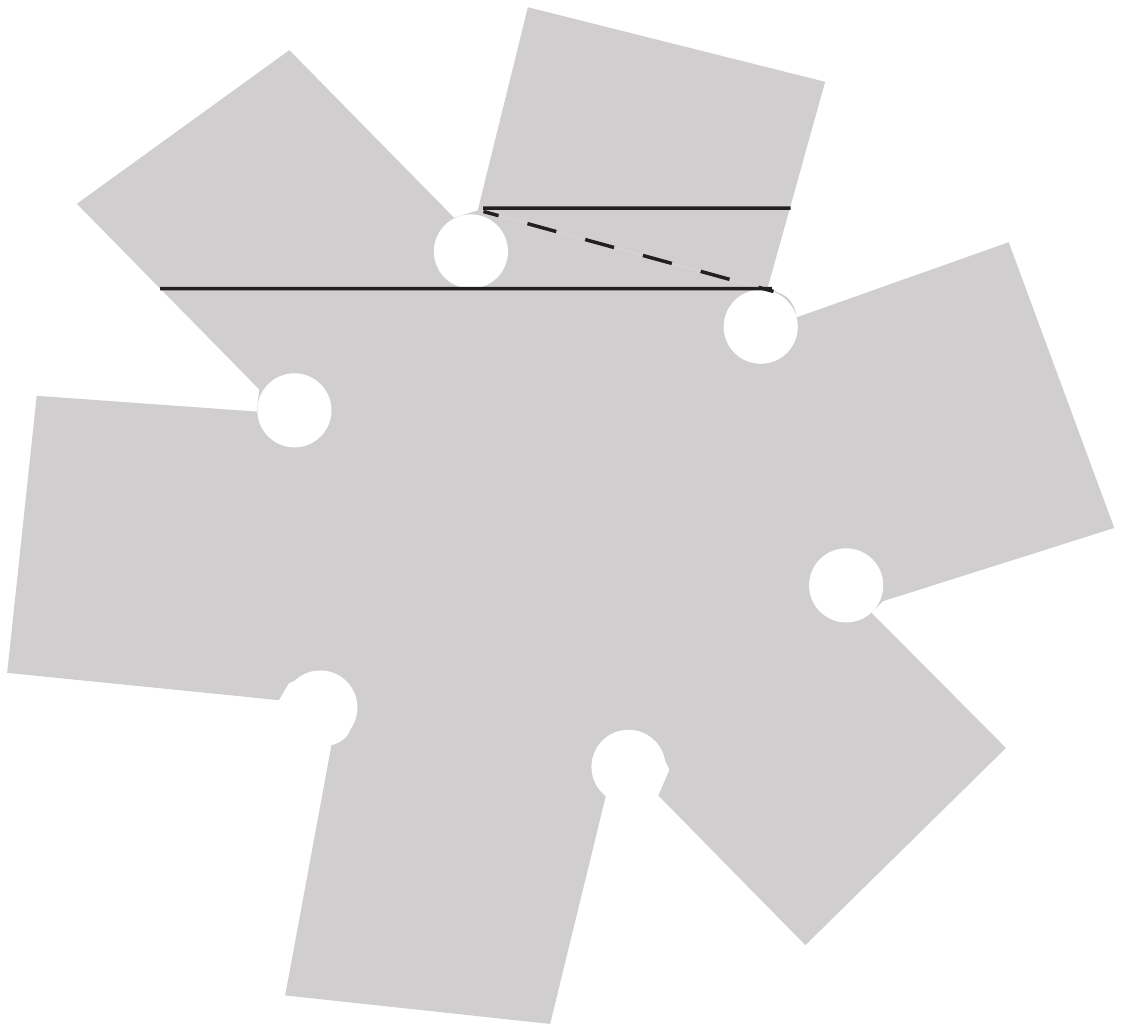}}\quad 
\subfigure[]{ \includegraphics[width=3.8cm]{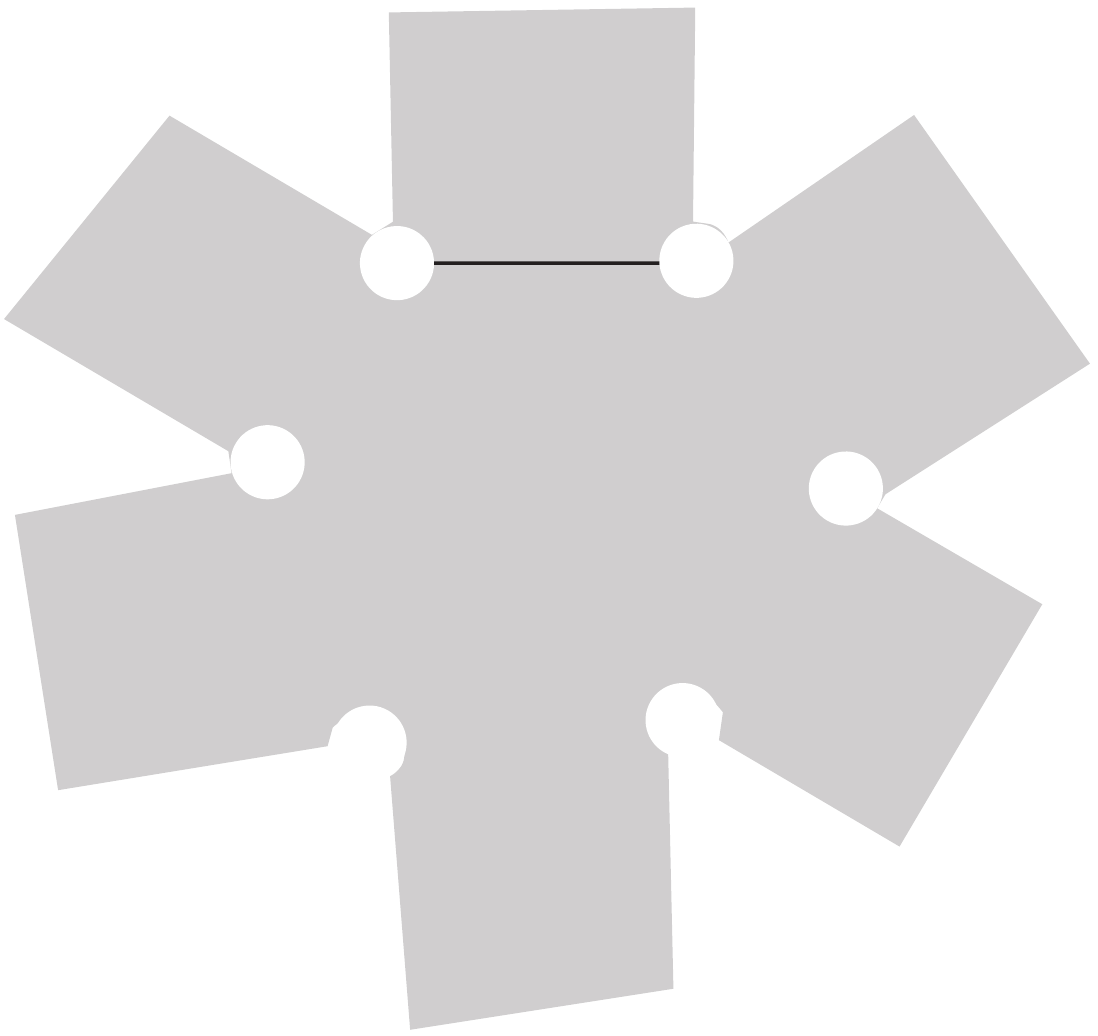}}\quad
\subfigure[]{ \includegraphics[width=3.8cm]{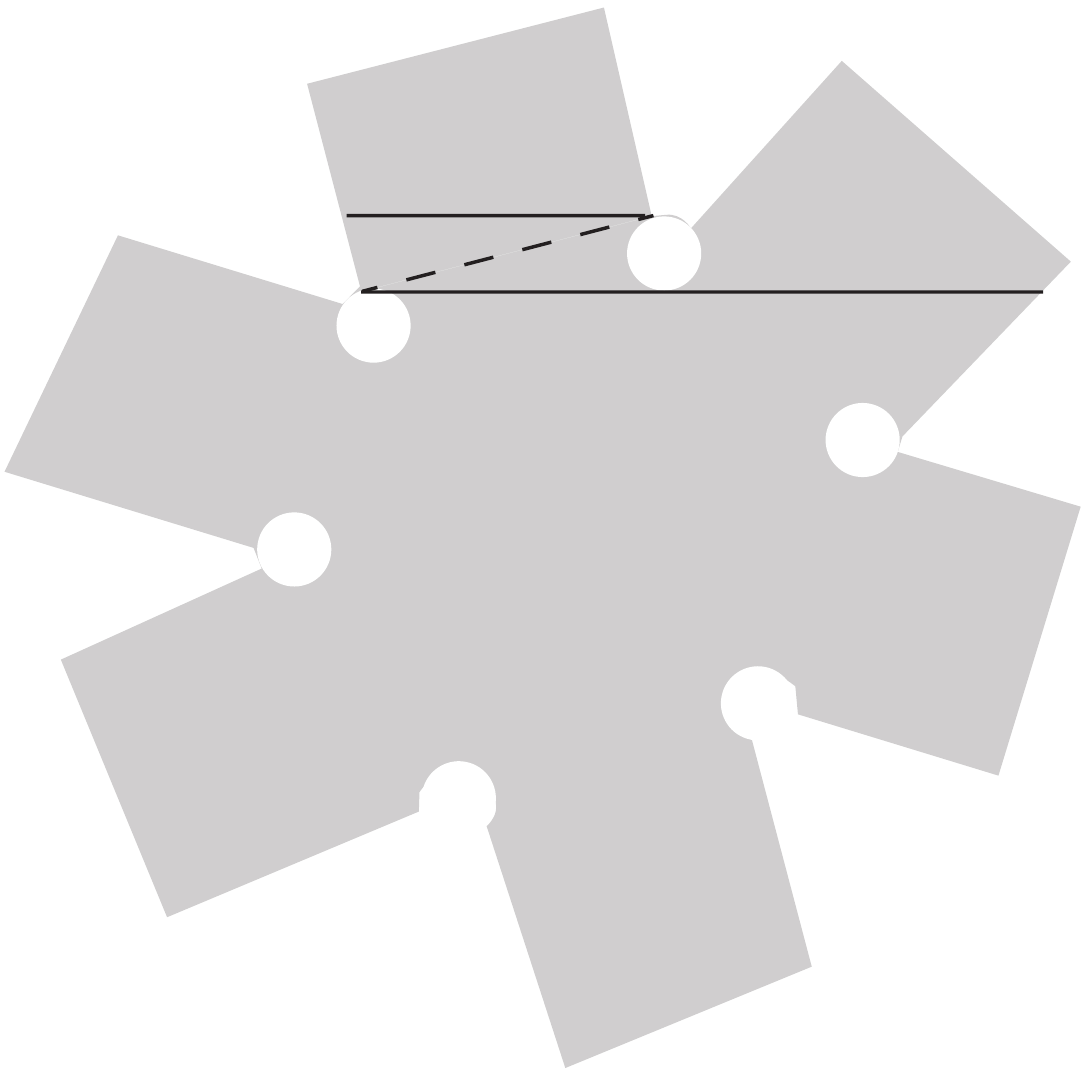}}\caption{The star domain in $t$-space. When $\alpha$ is increasing the skeleton graph inside the disk is broken in (b) between the two limit positions corresponding to double tangencies. }\label{pseudo-gon_nopassage}\end{center}\end{figure}

The tool to describe this bifurcation for the vector field \eqref{vector_field} is the generalized periodgon and the star domain of Definition~\ref{def:generalised_periodgon} (see Figure~\ref{pseudo-gon_nopassage}).

\subsection{Bifurcation of the singular points} 
The singular points are located at the zeros of $P_\eps(z)$. The discriminant is 
\begin{align*}\begin{split}
\Delta(\eps_1,\eps_0)&=C\left[\big(\tfrac{\eps_0}k\big)^k-\big(-\tfrac{\eps_1}{k+1}\big)^{k+1}\right] +o(\pl\eps\pl^{k(k+1)})\\
&=C \zeta^{k(k+1)}e^{-ik(k+1)\alpha}\left[s^{k(k+1)}e^{ik\theta}-(1-s)^{k(k+1)}\right]+o(\zeta^{k(k+1)}),\end{split}\end{align*} 
where $C=(-1)^{\lfloor\frac{k+1}2\rfloor}k^k(k+1)^{k+1}$.  Hence it is  the closure of a  2-dimensional real surface close to  $(s,\theta)=(\frac12,0)$ inside the 4-dimensional parameter space. (Remember that $(s,\theta+\frac{2m\pi}{k}, \alpha+\frac{2m\pi}{k})\sim(s,\theta,\alpha)$ by \eqref{rel:quotient} .) This surface  cuts each topological sphere $\zeta=\text{cst}\neq0$ along a $(k+1,k)$ torus knot. When $\Delta=0$ and $\zeta \neq0$, there is exactly one parabolic point of multiplicity $2$. The closure contains $\zeta=0$ on which the parabolic point has multiplicity $k+1$.

\subsection{The slit domain in parameter space} 

Part of this section is numerical. To see how to slit the parameter space we use the description \eqref{rescaled} (but we still denote the variable by $z$). Indeed, for system \eqref{vector_field} the slits $\theta=\frac{2\pi m}{k}$, $s\in [0, \frac12)$, correspond to parameter values where  $z_0$ and $z_1$ have eigenvalues in opposite direction.

By symmetry, let us again concentrate on $\theta=0$ in \eqref{rescaled}. Since $z_0$ and $z_1$ are isolated for $s\in [0,\frac12)$ they depend analytically on the parameters. Numerical evidence is provided in Remark~\ref{rem:numerical_evidence} that $\left(\frac{d\arg(\lambda_0)}{d\theta}-\frac{d\arg(\lambda_1)}{d\theta}\right)|_{\theta=0}<0$. Hence, by the implicit function theorem there exists for small $\zeta$ a surface $\theta= \Theta_0(s,\zeta, \alpha)$ on which $\arg(\lambda_1)+\arg(\lambda_0)=\pi$, which is the natural cut in parameter space.  When we limit $s$ to a compact subinterval $s\in [0,s_0]$ for $s_0\in(0,\frac12)$, then the cut can be defined for all $|\zeta|<\delta$ for some positive $\delta$. In order to define it uniformly over $s\in [0,\frac12]$ we need to cover a neighborhood of $(s,\theta)=(\frac12,0)$. For that purpose we use that the vector field \eqref{rescaled} can be brought in the neighborhood of the parabolic point $(s,\theta)=(\frac12,0)$ (modulo a reparametrization) to the form 
$$\dot W= (W^2-\delta)(1+h(W,\delta,\eta)),$$ 
where $\delta$ is the discriminant, $\eta$ represents the remaining parameters and $h(W,\delta,\eta)=O(|W,\delta,\eta|)$. Then the eigenvalues are given by $\lambda_\pm= \pm2\sqrt{\delta}(1+h(\pm\sqrt{\delta},\delta,\eta))$. The cut we want to define should correspond to 
$\lambda_+=C\lambda_-$ for some $C\in \R^-$. Moreover $\frac1{\lambda_+}+\frac1{\lambda_-} = A(\delta, \eta)=a+O(\delta,\eta)$, where $a=\frac{\nu_{par}}{2\pi i}=-e^{ ik\alpha}\frac{2(k-1)}{3k(k+1)} 2^k\neq0$ by \eqref{nu_par}. Hence along the cut we should have that both $\frac1{\lambda_+}$ and $\frac1{\lambda_-}$ should be aligned with $A(\delta, \eta)$, which is satisfied as soon as $\lambda_-$ is aligned with $\frac1{A(\delta,\eta)}$, i.e. as soon as $\arg(\delta)=2\arg\left(\frac1{A(\delta,\eta)(1+h(-\sqrt{\delta}, \delta, \eta))}\right)$. This can be solved by the implicit function theorem for $\arg(\delta)$ as a function of $|\delta|$ and $\eta$.

\section{Perspectives} 
\subsection{Analytic classification of polynomial vector fields on $\CP^1$} The definition of the periodgon given in  Section~\ref{section:periodgon} is valid for any polynomial vector field on $\CP^1$. From it, we can recover the Douady-Sentenac combinatorial and analytic invariants. Hence the periodgon provides an analytic invariant for polynomial vector fields on $\CP_1$ under affine conjugacies. We will explore this new invariant in a forthcoming publication.

\end{document}